\documentclass[12pt]{amsart}
\usepackage{setspace}
\usepackage[margin=.75in]{geometry}
\usepackage{amsmath}
\usepackage{amsxtra}
\usepackage{amscd}
\usepackage{amsthm}
\usepackage{amssymb}
\usepackage[normalem]{ulem}
\usepackage{comment}
\usepackage{color}
\usepackage{tikz}
\usetikzlibrary{matrix,arrows,decorations.pathmorphing,automata, matrix, positioning, calc, shapes.multipart}
\usepackage{tikz-cd}
\usepackage{graphicx}
\usepackage{hyperref}
\usepackage{url}
\usepackage[shortlabels]{enumitem}
\usepackage{float}
\usepackage{epigraph}
\usepackage{mathrsfs}
\usepackage{bbm}
\usepackage{cite}

\newcommand{\GU}{{\rm GU}}
\newcommand{\Sp}{\mbox{\rm Sp}}

\newcommand{\GL}{{\rm GL}}
\newcommand{\St}{\mbox{\rm St}}
\newcommand{\st}{{\rm st}}
\newcommand{\End}{\mbox{\rm End}}
\newcommand{\Hom}{\mbox{\rm Hom}}

\newcommand{\Ind}{\mbox{\rm Ind}}

\newcommand{\cA}{\mathcal{A}}

\newcommand{\cF}{\mathcal{F}}

\newcommand{\cO}{\mathcal{O}}
\newcommand{\cP}{\mathcal{P}}

\newcommand{\cU}{\mathcal{U}}

\newcommand{\Z}{\mathbb{Z}}

\newcommand{\N}{\mathbb{N}}
\newcommand{\Q}{\mathbb{Q}}

\newcommand{\bs}{\mathbf{s}}
\newcommand{\bt}{\mathbf{t}}

\newcommand{\la}{\lambda}

\newcommand{\bla}{\boldsymbol{\la}}

\newcommand{\bmu}{\boldsymbol{\mu}}

\newcommand{\sle}{\widehat{\mathfrak{sl}}_e}

\newcommand{\charac}{\text{\rm char}}

\newcommand{\rG}{G}

\newcommand{\rL}{L}
\newcommand{\rM}{M}
\newcommand{\rP}{P}
\newcommand{\rT}{T}
\newcommand{\rU}{U}

\numberwithin{equation}{section}

\theoremstyle{plain}
\newtheorem{theorem}[equation]{Theorem}
\newtheorem{lemma}[equation]{Lemma}

\newtheorem{corollary}[equation]{Corollary}

\newtheorem{definition}[equation]{Definition}

\newtheorem*{main}{Main Theorem}

\theoremstyle{remark}
\newtheorem{example}[equation]{Example}
\newtheorem{remark}[equation]{Remark}

\newcommand\TikZ[1]{\begin{matrix}\begin{tikzpicture}#1\end{tikzpicture}\end{matrix}}

\newcommand{\emp}{\emptyset}

\newcommand{\slinf}{\mathfrak{sl}_\infty}

\usepackage[foot]{amsaddr}

\title[]{On modular Harish-Chandra series of finite unitary groups}
\author{Emily Norton}
\address[E.N.]{Department of Mathematics, TU Kaiserslautern, Gottlieb-Daimler-Strasse 48, 67663 Kaiserslautern, Germany} 
\email{norton@mathematik.uni-kl.de}
\begin{document}

\begin{abstract} 
In the modular representation theory of finite unitary groups when the characteristic $\ell$ of the ground field is a unitary prime, 
 the $\widehat{\mathfrak{sl}}_e$-crystal on level $2$ Fock spaces graphically describes the Harish-Chandra branching of unipotent representations restricted to the tower of unitary groups. However, how to determine the cuspidal support of an arbitrary unipotent representation has remained an open question.
 We show that for $\ell$ sufficiently large, the $\mathfrak{sl}_\infty$-crystal on the same level $2$ Fock spaces provides the remaining piece of the puzzle for the full Harish-Chandra branching rule. 
\end{abstract}

\maketitle

\sloppy


\section*{Introduction}

 Unipotent blocks are a distinguished set of blocks of the category of finitely generated $\mathbbm{k}\rG(q)$-modules for $\rG(q)$ a finite group of Lie type over a ground field $\mathbbm{k}$ of characteristic $\ell$ coprime to $q$. Their definition glances back at geometry of Deligne-Lusztig varieties and they possess features reminiscent of Kazhdan-Lusztig theory. 
 Some very basic problems about unipotent blocks remain open, demonstrating the thorny difficulties of modular representation theory even in non-defining characteristic. First, outside of type $A$ we lack a conceptual, algorithmic, or combinatorial description of the decomposition numbers of unipotent blocks.\footnote{For finite unitary groups $\GU_n(q)$ when the characteristic $\ell$ of the ground field is a unitary prime, a result of Gruber and Hiss shows that the decomposition numbers are given in terms of those for finite general linear groups \cite{GruberHiss}; when $\ell$ divides $q+1$, 
 a conjectural formula for decomposition numbers, arising from geometry of the Hilbert scheme of $n$ points in $\mathbb{C}^2$, was recently announced in conference talks by Dudas and Rouquier \cite{DR}. }
 Another basic problem concerns the partitioning of the simple modules into Harish-Chandra series: we would like to know the rule determining which isomorphism classes of simple modules occur in the maximal semisimple quotient of a parabolically induced simple module of a Levi subgroup. This branching rule describes how simple modules are built up from modules of smaller subgroups; this is basic information about simple modules, but additionally, knowing the rule can help in computing decomposition numbers. 
 
 The rather dry name of Harish-Chandra series belies a concept akin to a garden plot thick with perennial flowering plants. Let $\rG_n$ be a finite classical group in an infinite series indexed by $n\in\N$, such as $\Sp_{2n}(q)$ or $\GU_{2n}(q)$. When we look at the set of simple modules for $\rG_n$ at a particular value of $n$, we are looking at the blossoms in our garden at time $n$. Over time, the plants grow, shed their old blossoms and produce new ones, and new plants may appear from seeds that have taken root in the ground. The ``seed" is a cuspidal simple module of a standard Levi subgroup $\rL\leq \rG_n$, that is, a simple $\mathbbm{k}\rL$-module which does not appear in the head of any module Harish-Chandra induced from a smaller Levi subgroup; the plant that grows from the seed is the endomorphism algebra of the induced cuspidal, a Hecke algebra;  the blossoms are the simple modules in the head of the induced cuspidal, and they are in bijection with the simple modules of the Hecke algebra \cite{GHM2}. 
 The data of a cuspidal pair $(L,X_L)$ where $X_L$ is a cuspidal $\mathbbm{k}L$-module, the associated Hecke algebra, and its simple modules, taken over all cuspidal pairs where $L$ is a standard Levi subgroup of $G_n$, yields a partitioning of the simple $\mathbbm{k}\rG_n$-modules into what are called Harish-Chandra series \cite{GHM2}. 
 
 In this paper we address the following problem for an infinite series of finite classical groups: \textit{given a simple module $X$ in a unipotent block of $\mathbbm{k}\rG_n(q)$, determine the Harish-Chandra series of $X$.} We formulate and prove a complete solution to this problem for finite unitary groups $\GU_n(q)$ in the case of a large unitary prime $\ell$ (i.e. the quantum characteristic $e$ given by the order of $-q$ mod $\ell$ is odd and at least $3$), the case that is not yet understood, topping off the substantial partial results of Gerber, Hiss, Jacon, Dudas, Varagnolo, and Vasserot \cite{GHJ},\cite{GH},\cite{DVV1},\cite{DVV2}. As happens for finite general linear groups, the simple modules in unipotent blocks of $\mathbbm{k}\GU_n(q)$ are parametrized by partitions of $n$. For $j\in \N$ we let $\st_j$ denote the  Steinberg  module of $\mathbbm{k}\GL_j(q^2)$, which is the simple unipotent $\mathbbm{k}\GL_j(q^2)$-module labeled by $(1^j)$.
 Our main result is a description of the Harish-Chandra series of a simple module in terms of two crystals on level $2$ Fock spaces:
 
 \noindent \begin{main} (Theorem \ref{conj}) 
 Assume that the order $e$ of $-q$ mod $\ell$ is odd and at least $3$ and that $\ell=\charac(\mathbbm{k})>n/e$. 
 Let $\lambda$ be a partition of $n$ and let $X_\lambda$ be the unipotent $\mathbbm{k}\GU_{n}(q)$-representation labeled by $\lambda$. Then the cuspidal support of $X_\lambda$ is given by $$(\GU_{n'}(q)\times \GL_e(q^2)^{\times k}\times \GL_1(q^2)^{\times r}, X_{\lambda_0}\otimes \st_e^{\otimes k}\otimes \st_1^{\otimes r} )$$ 
 where $r$ is the depth of $\tau(\lambda)$ in the $\sle$-crystal and $k$ is the depth of $\tau(\lambda)$ in the $\slinf$-crystal, $n'=n-2(ek+r)$, and $\tau(\lambda_0)$ is the source vertex of the connected component of these crystals containing $\tau(\lambda)$.
 \end{main}

 \noindent Here, $\tau(\lambda)$ is the twisted $2$-quotient map to a level $2$ Fock space with charge $(s_1,s_2)\in\Z^2$ determined by the $2$-core of $\lambda$, see Section \ref{unip unitary combinat}. The reason for the condition that $e$ be odd is this: the representation theory in unipotent blocks depends only on the order $e$ of $-q\mod\ell$ so long as $\ell$ is big enough, and then there are two kinds of behavior: when $e$ is even ($\ell$ is a ``linear prime") the theory reduces to that of finite general linear groups and is known, while when $e$ is odd ($\ell$ is a ``unitary prime") the representation theory behaves as a type $B$ phenomenon and a complete description is an open problem. We require $\ell>n/e$ so that the answer only depends on $e$, not on $\ell$; this is the case needed for studying decomposition matrices that are likewise dependent on $e$ but independent of $\ell$.
 
 The Main Theorem (Theorem \ref{conj}) augments and completes a conjecture of Gerber-Hiss-Jacon \cite{GHJ}, now a theorem of Dudas-Varagnolo-Vasserot \cite{DVV1} (see also Gerber-Hiss \cite{GH} for an earlier solution in a special case), that describes the ``weak cuspidals" and ``weak Harish-Chandra series" of simple modules in unipotent blocks using the $\sle$-crystal  
on level $2$ Fock spaces. 
 Weak Harish-Chandra theory for finite unitary groups restricts Harish-Chandra theory to the tower of groups $\GU_\iota(q)\subset \GU_{\iota+2}(q)\dots \subset \GU_{2m+\iota}(q)\subset \GU_{2m+2+\iota}(q)\subset\dots$ for $\iota\in\{0,1\}$ and thus to Levi subgroups that are of the same Dynkin type $^2A_n$ \cite{GHJ}. However, the group $\GL_e(q^2)$  also admits a (unique) unipotent cuspidal and can appear as a factor in a Levi subgroup. The information about branching with respect to such type $A$ Levi subgroups cannot be obtained using the $\sle$-crystal. To understand the actual Harish-Chandra branching rule of arbitrary simple modules in unipotent blocks, we have to study more than the $\sle$-crystal.

Theorem \ref{conj} says that for $\ell$ a sufficiently large unitary prime, the Harish-Chandra series of an arbitrary simple module in such a unipotent block is combinatorially encoded by the joint presence of the $\sle$- and $\slinf$-crystals on a  certain level $2$ Fock space. The idea is that this second simple directed graph, which we call the $\slinf$-crystal because each of its connected components is isomorphic to that graph, should keep track of the contribution to the cuspidal support from those Levi subgroups involving some number of copies of $\GL_e(q^2)$.  The $\slinf$-crystal has been proven to play this role in the analogous situation of describing the Harish-Chandra series of level $2$ cyclotomic rational Cherednik algebras, that is, rational Cherednik algebras of type $B$ Weyl groups 
\cite{ShanVasserot},\cite{Losev}. This crystal was first defined indirectly using the categorical action of an infinite-dimensional Heisenberg Lie algebra on cyclotomic Cherednik category $\cO$ \cite{ShanVasserot}. A combinatorial description crystallized through the work of several authors over the following years \cite{Losev}, \cite{Gerber1}, \cite{Gerber2}, \cite{GerberN}.  As a directed graph, each of its connected components is isomorphic to Young's lattice, the branching graph of the symmetric group in characteristic $0$ \cite{ShanVasserot},\cite{Losev}. 
 Each arrow in the crystal adds a vertical strip of $e$ boxes to a bipartition \cite{Losev},\cite{Gerber2}. The precise rule is subtle but similar in spirit to the rule for adding a good box in the $\sle$-crystal \cite{GerberN}. 
 
 In work by Gerber \cite{Gerber2}, the $\slinf$-crystal on level $2$ Fock spaces was called the Heisenberg crystal because of its relationship with the action of  the infinite-dimensional ``Heisenberg Lie algebra."\footnote{We prefer not to use the name Heisenberg crystal because (a) the $\sle$-crystal is closely related to the action of the so-called quantum Heisenberg category \cite{BSW1},\cite{BSW2}, and this tends to confuse outsiders who then conflate the two ``Heisenberg" actions, and (b) while Heisenberg definitely spent time thinking about uranium for the Nazis, he never thought about this cute graph structure on higher level Fock spaces.}
 We warn the reader that our preferred nomenclature of $\slinf$-crystal is a linguistic misdemeanor because we are only using ``crystal" to mean ``a simple directed graph carrying representation theoretic meaning which happens to be isomorphic as a graph to the crystal graph of a Lie algebra." We do not find it using any quantum group or limit as $q$ goes to $0$.
 
 Theorem \ref{conj} is inspired by the full Harish-Chandra branching rule for cyclotomic rational Cherednik algebras and is a direct translation of that rule to the setting of finite groups of Lie type. The Harish-Chandra branching rule for cyclotomic rational Cherednik algebras follows from the construction of two categorical actions on the associated tower of category $\cO$'s, an $\sle$-categorical action and a categorical action of the infinite-dimensional Heisenberg Lie algebra $\mathfrak{H}$ \cite{Shan}, \cite{ShanVasserot}. The proof of the weak Harish-Chandra branching rule for finite unitary groups follows from the construction of an $\sle$-categorical action on unipotent blocks of the finite unitary groups \cite{DVV1} and the classification of cuspidal unipotent modules follows from constructing some version of the action of $\mathfrak{H}$ as well \cite{DVV2}. The reason the full branching rule we give in Theorem \ref{conj} has not yet appeared in the literature is that there are apparently problems in extending this  approach using a categorical action of $\mathfrak{H}$ to the setting of finite unitary groups. We have attempted to complete the branching rule for unipotent representations of finite unitary groups in a naive and elementary way avoiding further use of Heisenberg algebras or categorical actions. We use algebraic arguments, such as can be found in \cite{DM} for computing decomposition matrices, that appeal to unitriangularity of the decomposition matrix, and we take advantage of the fact that the combinatorial structure of the $\slinf$-crystal on level $2$ Fock spaces is now well-understood and explicit. We need the classification of cuspidals \cite{DVV2}, then we use induction on the rank of the Levi factor of the form $\GU_{n'}(q)$ in the cuspidal support of a simple module. Certain paths in the $\slinf$-crystal on the twisted $2$-quotients of partitions are compatible with the dominance order on partitions, whence Theorem \ref{projectives thm}.\footnote{For in-depth studies of orders on $d$-partitions versus orders on partitions of $n$, arising from geometry of the Hilbert scheme of $n$ points in $\mathbb{C}^2$, see \cite{Prz}, \cite{Gordon}.}

 	We have only dealt with the case that $\ell>n/e$; for the solution to the Harish-Chandra branching problem for $\GL_n(q)$ without bounds on $\ell$, see \cite{DipperDu}.  When computing decomposition matrices which depend only on $e$ not on $\ell$ as in \cite{DM}, it is always assumed that $\ell$ is at least as big as $n$; in general $\ell$ should be much bigger than $n$, as the decomposition matrix of the Hecke algebra associated to the Frobenius-fixed points of the Weyl group of $G(\overline{F_q})$ is a submatrix of the decomposition matrix of $\mathbbm{k}\rG$, and the bounds for such a decomposition matrix to be independent of $\ell$ are known to be extremely large \cite{Williamson}.
 	
 	A similar result to Theorem \ref{conj} should hold for finite classical groups of types $B$ and $C$. In that case the weak Harish-Chandra branching rule is given by the $\sle$-crystal on a sum of level $2$ Fock spaces \cite{DVV2}, but there is a little more work to do as the analogue of \cite[Theorem 5.10]{DVV2} identifying cuspidals has not been checked. The validity of any of these results relies on unitriangularity of the decomposition matrix, which was conjectured by Geck for all types and very recently proven in \cite{BDT}.

 \textbf{Overview of the paper.} Section \ref{crystals} explains the combinatorics needed. Section \ref{unipunit} mostly reviews the representation theory part of the story, but also contains two new theorems  on cuspidals (Theorems \ref{cusp column} and \ref{blocks with cuspidals}). Section \ref{heart of the artichoke} states and proves the main theorem, Theorem \ref{conj}. The Appendix, Section \ref{pile of artichoke leaves with teeth scrapes}, consists mainly of explicit, direct case-by-case checks of the statement of Theorem \ref{conj} in small rank for $e=3$ and was written for an earlier version of this paper, in which the main theorem was stated as a conjecture. The Appendix also has a theorem about the ordinary unipotent characters occurring in the Harish-Chandra induction of the projective cover of the Steinberg module of a maximal type $A$ Levi subgroup (Theorem \ref{inducing the type A Steinberg}). We have kept this material from the older version of the paper in case the reader would like to see explicit computations and tables or how Theorem \ref{conj} can be checked by hand in small rank or for specific unipotent modules, like those in the head of the induced Steinberg module.

\section{Fock spaces and crystal graphs}\label{crystals}
\subsection{Charged bipartitions} A partition $\lambda=(\lambda_1,\lambda_2,\dots)$ is a nonincreasing, finite sequence of positive integers $\lambda_1\geq\lambda_2\geq\dots$. We write $|\lambda|=\lambda_1+\lambda_2+\dots$ and call $|\lambda|$ the size of $\lambda$. Let $\mathcal{P}$ denote the set of all partitions, including the empty partition $\emp$ and let $\mathcal{P}(n)$ denote the set of all partitions of size $n$. A bipartition $(\lambda^1,\lambda^2)$ is an ordered pair of partitions $\lambda^1,\lambda^2\in\mathcal{P}$. We denote by $\mathcal{P}^{(2)}$ the set of all bipartitions and by $\mathcal{P}^{(2)}(n)$ the set of all bipartitions of size $n$, meaning that $|\lambda_1|+|\lambda_2|=n$. 
We identify a partition $\lambda$ with its Young diagram, the upper-left-justified array of boxes in the plane whose first row has $\lambda_1$ boxes, the second row $\lambda_2$ boxes, and so on. If $b=(x,y)$ is a box of $\lambda$ in row $x$ and column $y$, then the content of $b$ is $y-x$. In upgrading the definition of content to bipartitions there is some flexibility, as we might want to fix some notion of distance between the two partitions $\lambda^1$ and $\lambda^2$; this is done by shifting the contents of the boxes in the two partitions using a ``charge." A charged bipartition is a pair $|\bla,\bs\rangle$ where $\bla=(\lambda^1,\lambda^2)\in\mathcal{P}^{(2)}$ and $\bs=(s_1,s_2)\in\Z^2$. If $b=(x,y,j)\in|\bla,\bs\rangle $ is a box in $\lambda^j$, the content of $b$ is defined to be $\mathrm{ct}(b):=y-x+s_j$.

 A charged bipartition $|\bla,\bs\rangle$ can be visualized as an abacus $\cA|\bla,\bs\rangle$ on two runners. Define the $\beta$-set of $(\lambda^j,s_j)$ to be the infinite set of integers 
$\beta^j:=\{\lambda^j_1+s_j,\lambda^j_2+s_j-1,\lambda^j_3+s_j-2,\dots,\lambda^j_m+s_j-m+1,\dots\}=:\{\beta^j_1,\beta^j_2,\beta^j_3,\dots,\beta^j_m,\dots\}$ where we extend $\lambda^j$ by an infinite sequence of $0$'s so that it has infinitely many parts. 
We then make a single-line abacus for $\beta^j$ by filling the integer number line with beads and spaces -- a bead in position $z$ if $z\in\beta^j$, a space otherwise. We make the abacus $\cA|\bla,\bs\rangle$ of $|\bla,\bs\rangle$ by stacking the two single-line abaci for our $\beta$-sets, putting the single-line abacus for $\beta^2$ on top and the single-line abacus for $\beta^1$ on the bottom. We refer to the top row as row $2$ and the bottom row as row $1$. The abacus is infinitely full of beads to the left, and infinitely empty with spaces to the right, and all mixture of beads and spaces occurs within a finite interval; when we picture abaci, we draw only a relevant finite region with the wealth of beads to the left and spaces to the right being understood. The charge $(s_1,s_2)$ can be found by swiping all the beads to the left and then reading off the position of the rightmost bead in each row. The charge $(s_1,s_2)$ is usually only defined up to adding $(c,c)$ for some $c\in\Z$. We write $(x,j)\in\cA|\bla,\bs\rangle$ if $x\in\beta^j$. Thus $(x,j)\in\cA|\bla,\bs\rangle$ means that the abacus has a bead in row $j$ and column $x$, and $(x,j)\notin\cA|\bla,\bs\rangle$ means that it has a space in that position.

\subsection{Fock spaces of level $2$ and the $\widehat{\mathfrak{sl}}_e$-crystal}

Fix $e\in\N_{\geq 2}$ and $\bs=(s_1,s_2)\in\Z^2$. The level $2$ Fock space $\mathcal{F}_{e,\bs}$ is the $\Q$-vector space with basis all charged bipartitions $|\bla,\bs\rangle$, $\bla\in\mathcal{P}^{(2)}$. 
 The basis elements $|\bla,\bs\rangle$ of $\mathcal{F}_{e,\bs}$ provide the vertices for a simple directed graph called the $\sle$-crystal. Recall that a simple graph is a graph with at most one edge between any two vertices, and a directed graph is a graph in which a direction is assigned to each edge; we call the edges of a directed graph arrows; a source vertex of a directed graph is a vertex with no incoming arrows. 
 There is an edge $|\bla,\bs\rangle$ to $|\bmu,\bs\rangle$ in the $\sle$-crystal if and only if $|\bmu|=|\bla|+1$ and $\bmu$ is obtained from $\bla$ by adding a ``good" box $b$ of content $i$ for some $i\in\Z/e\Z$ ($b$ is called a good addable $i$-box) \cite{FLOTW}. We then write $\bmu=\tilde{f}_i(\bla)$, whereas if $\bla$ has no good addable $i$-box then we write $\tilde{f}_i(\bla)=0$. Conversely, if $\tilde{f}_i(\bla)=\bmu$ then $\bla=\tilde{e}_i(\bmu)$ where $\tilde{e}_i$ is the operator that removes a ``good" box of content $i$ (or is $0$ if such a box does not exist). The box removed is then called a good removable $i$-box. The identification of a good addable/removable $i$-box depends not only on the charge $\bs$ but also on a choice of a total order on the set of addable and removal boxes of $\bla$: $b>b'$ if $\mathrm{ct}(b)>\mathrm{ct}(b')$ or $\mathrm{ct}(b)=\mathrm{ct}(b')$ and $b\in\lambda^1$, $b'\in\lambda^2$.
 For each $i\in\Z/e\Z$ there is at most one good addable/removable $i$-box in a given charged bipartition $|\bla,\bs\rangle$. 
 The  $\sle$-crystal partitions the $\sle$-crystal graph on $\mathcal{F}_{e,\bs}$ into connected components, each of which emanates from a unique source vertex.

\subsection{Source vertices of the $\sle$-crystal and $e$-periods}
There is a nice combinatorial description of the source vertices of the $\sle$-crystal:

\begin{theorem}\label{tot per}\cite[Thm 5.9]{JL1} 
A charged bipartition $|\bla,\bs\rangle$ is a source vertex of the $\sle$-crystal on $\mathcal{F}_{e,\bs}$ if and only if its abacus $\cA|\bla,\bs\rangle$ is totally $e$-periodic.
\end{theorem}
\noindent We need to explain what it means for an abacus to be ``totally $e$-periodic" \cite[Definition 5.4]{JL1}. Informally, this means that the beads in the abacus of $|\bla,\bs\rangle$ can be partitioned into chains of length $e$ called $e$-periods which have the following shape: the rightmost $u$ beads are in the top row, the leftmost $e-u$ beads are in the bottom row for some $u\in\{0,\dots,e\}$, and the $\beta$-numbers corresponding to the beads fill out an interval of length $e$. For example, when $e=5$ these $e$-periods can have the following shapes, which we draw with red lines connecting the beads in the $e$-period:
 $$
  \TikZ{[scale=.5]
 	\draw
 	(1,1)node[fill,circle,inner sep=3pt]{}
 	(2,1)node[fill,circle,inner sep=3pt]{}
 	(3,1)node[fill,circle,inner sep=3pt]{}
 	(4,1)node[fill,circle,inner sep=3pt]{}
 	(5,1)node[fill,circle,inner sep=3pt]{}
 	;
 	\draw[very thick,red] plot [smooth,tension=.1] coordinates{(5,1)(1,1)};
}
 \quad ,\quad 
 \TikZ{[scale=.5]
 	\draw
 	(1,0)node[fill,circle,inner sep=3pt]{}
 	(2,1)node[fill,circle,inner sep=3pt]{}
 	(3,1)node[fill,circle,inner sep=3pt]{}
 	(4,1)node[fill,circle,inner sep=3pt]{}
 	(5,1)node[fill,circle,inner sep=3pt]{}
 	;
 \draw[very thick,red] plot [smooth,tension=.1] coordinates{(5,1)(2,1)(1,0)};
}
\quad ,\quad 
\TikZ{[scale=.5]
	\draw
	(1,0)node[fill,circle,inner sep=3pt]{}
	(2,0)node[fill,circle,inner sep=3pt]{}
	(3,1)node[fill,circle,inner sep=3pt]{}
	(4,1)node[fill,circle,inner sep=3pt]{}
	(5,1)node[fill,circle,inner sep=3pt]{}
	;
\draw[very thick,red] plot [smooth,tension=.1] coordinates{(5,1)(3,1)(2,0)(1,0)};
}
\quad ,\quad 
\TikZ{[scale=.5]
	\draw
	(1,0)node[fill,circle,inner sep=3pt]{}
	(2,0)node[fill,circle,inner sep=3pt]{}
	(3,0)node[fill,circle,inner sep=3pt]{}
	(4,1)node[fill,circle,inner sep=3pt]{}
	(5,1)node[fill,circle,inner sep=3pt]{}
	;
\draw[very thick,red] plot [smooth,tension=.1] coordinates{(5,1)(4,1)(3,0)(1,0)};
}
\quad ,\quad 
\TikZ{[scale=.5]
	\draw
	(1,0)node[fill,circle,inner sep=3pt]{}
	(2,0)node[fill,circle,inner sep=3pt]{}
	(3,0)node[fill,circle,inner sep=3pt]{}
	(4,0)node[fill,circle,inner sep=3pt]{}
	(5,1)node[fill,circle,inner sep=3pt]{}
	;
\draw[very thick,red] plot [smooth,tension=.1] coordinates{(5,1)(4,0)(1,0)};
}
$$
where the completely horizontal shape can go in either row $1$ or row $2$ of the abacus. An example of a totally $5$-periodic abacus corresponding to the charged bipartition $|6^32^4.4^23^61^4,(-2,2)\rangle$ is:

$$\TikZ{[scale=.5]
	\draw
		(-12,0)node[fill,circle,inner sep=3pt]{}
	(-11,0)node[fill,circle,inner sep=3pt]{}
	(-10,0)node[fill,circle,inner sep=3pt]{}
	(-9,0)node[fill,circle,inner sep=3pt]{}
	(-8,0)node[fill,circle,inner sep=3pt]{}
	(-12,1)node[fill,circle,inner sep=3pt]{}
	(-11,1)node[fill,circle,inner sep=3pt]{}
	(-10,1)node[fill,circle,inner sep=3pt]{}
	(-9,1)node[fill,circle,inner sep=3pt]{}
	(-8,1)node[fill,circle,inner sep=3pt]{}
	(-7,0)node[fill,circle,inner sep=3pt]{}
	(-6,0)node[fill,circle,inner sep=3pt]{}
	(-5,0)node[fill,circle,inner sep=3pt]{}
	(-4,0)node[fill,circle,inner sep=3pt]{}
	(-3,0)node[fill,circle,inner sep=3pt]{}
	(-7,1)node[fill,circle,inner sep=3pt]{}
	(-6,1)node[fill,circle,inner sep=3pt]{}
	(-5,1)node[fill,circle,inner sep=3pt]{}
	(-4,1)node[fill,circle,inner sep=3pt]{}
	(-3,1)node[fill,circle,inner sep=3pt]{}
		(1,0)node[fill,circle,inner sep=3pt]{}
	(2,0)node[fill,circle,inner sep=3pt]{}
	(3,0)node[fill,circle,inner sep=3pt]{}
	(4,0)node[fill,circle,inner sep=3pt]{}
	(5,1)node[fill,circle,inner sep=3pt]{}
	(6,1)node[fill,circle,inner sep=3pt]{}
	(7,1)node[fill,circle,inner sep=3pt]{}
	(8,1)node[fill,circle,inner sep=3pt]{}
	(9,1)node[fill,circle,inner sep=3pt]{}
	(10,1)node[fill,circle,inner sep=3pt]{}
	(-2,1)node[fill,circle,inner sep=.5pt]{}
	(-2,0)node[fill,circle,inner sep=3pt]{}
	(-1,0)node[fill,circle,inner sep=.5pt]{}
	(0,0)node[fill,circle,inner sep=.5pt]{}
	(-1,1)node[fill,circle,inner sep=3pt]{}
	(0,1)node[fill,circle,inner sep=3pt]{}
	(1,1)node[fill,circle,inner sep=3pt]{}
	(2,1)node[fill,circle,inner sep=3pt]{}
	(3,1)node[fill,circle,inner sep=.5pt]{}
	(4,1)node[fill,circle,inner sep=.5pt]{}
	(5,0)node[fill,circle,inner sep=.5pt]{}
	(6,0)node[fill,circle,inner sep=.5pt]{}
	(7,0)node[fill,circle,inner sep=.5pt]{}
	(8,0)node[fill,circle,inner sep=.5pt]{}
	(9,0)node[fill,circle,inner sep=3pt]{}
	(10,0)node[fill,circle,inner sep=3pt]{}
	(11,1)node[fill,circle,inner sep=.5pt]{}
	(11,0)node[fill,circle,inner sep=3pt]{}
	(12,1)node[fill,circle,inner sep=3pt]{}
	(13,1)node[fill,circle,inner sep=3pt]{}
	(12,0)node[fill,circle,inner sep=.5pt]{}
	(13,0)node[fill,circle,inner sep=.5pt]{}
	(14,0)node[fill,circle,inner sep=.5pt]{}
	(14,1)node[fill,circle,inner sep=.5pt]{}
	(15,0)node[fill,circle,inner sep=.5pt]{}
	(15,1)node[fill,circle,inner sep=.5pt]{}
	;
	\draw[very thick,red] plot [smooth,tension=.1] coordinates{(13,1)(12,1)(11,0)(9,0)};
	\draw[very thick,red] plot [smooth,tension=.1] coordinates{(10,1)(6,1)};
	\draw[very thick,red] plot [smooth,tension=.1] coordinates{(5,1)(4,0)(1,0)};
	\draw[very thick,red] plot [smooth,tension=.1] coordinates{(2,1)(-1,1)(-2,0)};
	\draw[very thick,red] plot [smooth,tension=.1] coordinates{(-3,0)(-7,0)};
	\draw[very thick,red] plot [smooth,tension=.1] coordinates{(-3,1)(-7,1)};
	\draw[very thick,red] plot [smooth,tension=.1] coordinates{(-8,0)(-12,0)};
	\draw[very thick,red] plot [smooth,tension=.1] coordinates{(-8,1)(-12,1)};
}
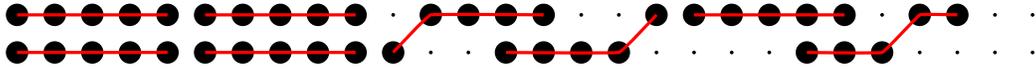
$$

 \begin{remark}
 The device of $e$-periods is very close to the device of yokes that appears in \cite{FodaWelsh}; however the two are distinct, as the $\beta$-numbers appearing in a yoke do not have to be consecutive, see \cite[Figure 3]{FodaWelsh}. In \cite{GerberN} we followed \cite{JL1} where $e$-periods and totally $e$-periodic abaci were introduced and \cite{Gerber2} where they were interpreted in terms of vertical strips. The cylindric multipartitions studied using yokes in \cite{FodaWelsh} and the totally $e$-periodic abaci as above turn out to be related to each other by level-rank duality \cite{GerberN+1}.
 \end{remark}

More formally, following \cite[Def 2.2]{JL1} we define an $e$-period as follows.
\begin{definition}\label{JL e-period def} The abacus $\cA|\bla,\bs\rangle$ of a charged bipartition $|\bla,\bs\rangle$ is said to have an $e$-period if it contains a sequence of beads $(x,j_1),(x-1,j_2),(x-2,j_3)\dots,(x-e+1,j_e)$ such that:
\begin{itemize}
	\item $x$ is the largest $\beta$-number in $\cA|\bla,\bs\rangle$;
	\item if $x-i+1\in\beta^1$ then $j_i=1$ for all $i=1,\dots,e-1$;
	\item if $j_i=1$ then $j_{i+1}=1$ for all $i=1,\dots,e-1$.
\end{itemize}
\end{definition}
The $e$-period of $\cA|\bla,\bs\rangle$, if it exists, represents the largest  vertical strip of length $e$ that can be added along the border of the bipartition $\bla$ in such a way that the charged contents of the added boxes form an interval of length $e$ and the result of adding the boxes is again a bipartition. By ``largest" we mean again, largest with respect to the order that $b>b'$ if $\mathrm{ct}(b)>\mathrm{ct}(b')$ or $\mathrm{ct}(b)=\mathrm{ct}(b')$ and $b\in\lambda^1$, $b'\in\lambda^2$, extended to an order on addable vertical strips in the obvious way. Technically, we should use the extended Young diagram of \cite{Jacon} to make this precise, which translates into a total order on the beads of the abacus: if $(\beta,j),\;(\beta',j')\in\cA|\bla,\bs\rangle$ are distinct beads in rows $j,j'$ and columns $\beta,\beta'$ respectively, then $(\beta,j)>(\beta',j')$ if $\beta>\beta'$ or $\beta=\beta'$ and $j=1$, $j'=2$.

 Define the first $e$-period $P_1$ of $\cA|\bla,\bs\rangle$ to be the $e$-period of $\cA|\bla,\bs\rangle$, if it exists. By induction then define the $k$'th $e$-period $P_k$ of $\cA|\bla,\bs\rangle$ to be the $e$-period of $\cA|\bla,\bs\rangle\setminus(P_1\cup\dots\cup P_{k-1})$ if it exists (thinking of each $e$-period $P_k$ as a subset of the beads of $\cA|\bla,\bs\rangle$ itself). An abacus $\cA|\bla,\bs\rangle$ is called totally $e$-periodic if there exists an $N\in\N$ such that $P_k$ is defined for all $k\leq N$ and removing $P_1,\dots, P_N$ from $\cA|\bla,\bs\rangle$ yields the abacus of the empty bipartition with respect to some charge \cite[Def 5.4]{JL1}; equivalently, if the $k$'th period of $\cA|\bla,\bs\rangle$ exists for any $k\in\N$.  The $e$-periods of a totally $e$-periodic abacus are unique and partition the beads of the abacus.
\subsection{The $\mathfrak{sl}_\infty$-crystal on level $2$ Fock spaces} We focus on the situation that $\cA|\bla,\bs\rangle$ is a totally $e$-periodic abacus of some charged bipartition $|\bla,\bs\rangle$. The description of the edges in the $\slinf$-crystal on $\cF_{e,\bs}$ is a bit more complicated in full generality and we refer to \cite{GerberN} for the general case. The idea however is similar to the totally $e$-periodic case. Moreover, the $\slinf$- and $\sle$-crystals on $\cF_{e,\bs}$ commute, so it is always possible go to a source vertex in the $\sle$-crystal by removing good boxes, do the computations in the $\slinf$-crystal, then retrace one's steps in the $\sle$-crystal by adding good boxes of the same residue sequence as were removed (in reverse order of course).

\begin{definition}\label{slinf tot per}\cite[Corollary 4.21]{GerberN} Let $\cF_{e,\bs}^{\mathrm{per}}$ denote the sub-$\Q$-vector space of $\cF_{e,\bs}$ spanned by those charged bipartitions $|\bla,\bs\rangle$, $\bla\in\cP^{(2)}$, whose abaci are totally $e$-periodic.
	The $\slinf$-crystal on $\cF_{e,\bs}^{\mathrm{per}}$ is the simple directed graph whose vertices consist of all $\bla\in\cP^{(2)}$ such that $\cA|\bla,\bs\rangle$ is totally $e$-periodic, and whose arrows are given by $\bla\rightarrow\bmu$ if and only if $\bmu$ is obtained from $\bla$ by shifting the $k$'th $e$-period $P_k$ of $\cA|\bla,\bs\rangle$ one step to the right and the right shift of $P_k$ coincides with the $k$'th $e$-period $P_k'$ of $\cA|\bmu,\bs\rangle$.
\end{definition}

\begin{example}
	Consider the charged bipartition $|\bla,\bs\rangle=|2^4.2^2,(-2,0)\rangle$ and take $e=3$. Its abacus is totally $e$-periodic, and thus it is a source vertex of the $\sle$-crystal by Theorem \ref{tot per}. We draw its abacus with its first five $3$-periods below, and the two edges emanating from it in the $\slinf$-crystal. 
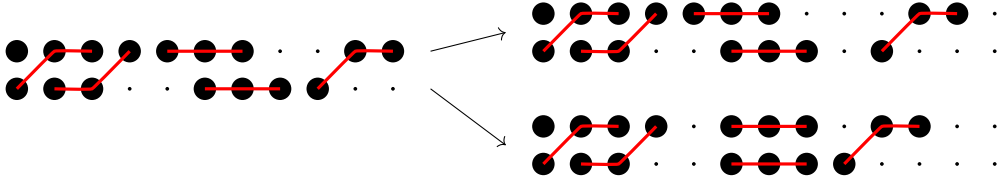
		\[
	\TikZ{[scale=.5]
		\draw
		(-4,0)node[fill,circle,inner sep=3pt]{}
		(-3,0)node[fill,circle,inner sep=3pt]{}
		(-2,0)node[fill,circle,inner sep=3pt]{}
		(-1,0)node[fill,circle,inner sep=3pt]{}
		(0,0)node[fill,circle,inner sep=3pt]{}
		(1,0)node[fill,circle,inner sep=3pt]{}
		(2,0)node[fill,circle,inner sep=3pt]{}
		(-4,-1)node[fill,circle,inner sep=3pt]{}
		(-3,-1)node[fill,circle,inner sep=3pt]{}
		(-2,-1)node[fill,circle,inner sep=3pt]{}
		(1,-1)node[fill,circle,inner sep=3pt]{}
		(2,-1)node[fill,circle,inner sep=3pt]{}
		(3,-1)node[fill,circle,inner sep=3pt]{}
		(4,-1)node[fill,circle,inner sep=3pt]{}
		(5,0)node[fill,circle,inner sep=3pt]{}
		(6,0)node[fill,circle,inner sep=3pt]{}
		(6,-1)node[fill,circle,inner sep=.5pt]{}
		(5,-1)node[fill,circle,inner sep=.5pt]{}
		(4,0)node[fill,circle,inner sep=.5pt]{}
		(3,0)node[fill,circle,inner sep=.5pt]{}
		(0,-1)node[fill,circle,inner sep=.5pt]{}
		(-1,-1)node[fill,circle,inner sep=.5pt]{}
		;
		\draw[very thick,red] plot [smooth,tension=.1] coordinates{(6,0)(5,0)(4,-1)}; 
		\draw[very thick,red] plot [smooth,tension=.1] coordinates{(3,-1)(1,-1)}; 
		\draw[very thick,red] plot [smooth,tension=.1] coordinates{(2,0)(0,0)};
		\draw[very thick,red] plot [smooth,tension=.1] coordinates{(-1,0)(-2,-1)(-3,-1)};
		\draw[very thick,red] plot [smooth,tension=.1] coordinates{(-2,0)(-3,0)(-4,-1)};
		\draw[->] plot [smooth] coordinates{(7,0)(9,.5)};
			\draw
		(10,1)node[fill,circle,inner sep=3pt]{}
		(11,1)node[fill,circle,inner sep=3pt]{}
		(12,1)node[fill,circle,inner sep=3pt]{}
		(13,1)node[fill,circle,inner sep=3pt]{}
		(14,1)node[fill,circle,inner sep=3pt]{}
		(15,1)node[fill,circle,inner sep=3pt]{}
		(16,1)node[fill,circle,inner sep=3pt]{}
		(10,0)node[fill,circle,inner sep=3pt]{}
		(11,0)node[fill,circle,inner sep=3pt]{}
		(12,0)node[fill,circle,inner sep=3pt]{}
		(15,0)node[fill,circle,inner sep=3pt]{}
		(16,0)node[fill,circle,inner sep=3pt]{}
		(17,0)node[fill,circle,inner sep=3pt]{}
		(19,0)node[fill,circle,inner sep=3pt]{}
		(20,1)node[fill,circle,inner sep=3pt]{}
		(21,1)node[fill,circle,inner sep=3pt]{}
		(13,0)node[fill,circle,inner sep=.5pt]{}
		(14,0)node[fill,circle,inner sep=.5pt]{}
		(17,1)node[fill,circle,inner sep=.5pt]{}
		(18,1)node[fill,circle,inner sep=.5pt]{}
		(19,1)node[fill,circle,inner sep=.5pt]{}
		(22,1)node[fill,circle,inner sep=.5pt]{}
		(18,0)node[fill,circle,inner sep=.5pt]{}
		(20,0)node[fill,circle,inner sep=.5pt]{}
		(21,0)node[fill,circle,inner sep=.5pt]{}
		(22,0)node[fill,circle,inner sep=.5pt]{};
		\draw[very thick,red] plot [smooth,tension=.1] coordinates{(21,1)(20,1)(19,0)}; 
		\draw[very thick,red] plot [smooth,tension=.1] coordinates{(17,0)(15,0)}; 
		\draw[very thick,red] plot [smooth,tension=.1] coordinates{(16,1)(14,1)};
		\draw[very thick,red] plot [smooth,tension=.1] coordinates{(13,1)(12,0)(11,0)};
		\draw[very thick,red] plot [smooth,tension=.1] coordinates{(12,1)(11,1)(10,0)};
		\draw[->] plot [smooth] coordinates{(7,-1)(9,-2.5)};
			\draw
	(10,-2)node[fill,circle,inner sep=3pt]{}
(11,-2)node[fill,circle,inner sep=3pt]{}
(12,-2)node[fill,circle,inner sep=3pt]{}
(13,-2)node[fill,circle,inner sep=3pt]{}
(17,-2)node[fill,circle,inner sep=3pt]{}
		(15,-2)node[fill,circle,inner sep=3pt]{}
		(16,-2)node[fill,circle,inner sep=3pt]{}
		(10,-3)node[fill,circle,inner sep=3pt]{}
		(11,-3)node[fill,circle,inner sep=3pt]{}
		(12,-3)node[fill,circle,inner sep=3pt]{}
		(15,-3)node[fill,circle,inner sep=3pt]{}
		(16,-3)node[fill,circle,inner sep=3pt]{}
		(17,-3)node[fill,circle,inner sep=3pt]{}
		(18,-3)node[fill,circle,inner sep=3pt]{}
		(19,-2)node[fill,circle,inner sep=3pt]{}
		(20,-2)node[fill,circle,inner sep=3pt]{}
		(13,-3)node[fill,circle,inner sep=.5pt]{}
		(14,-3)node[fill,circle,inner sep=.5pt]{}
		(14,-2)node[fill,circle,inner sep=.5pt]{}
		(18,-2)node[fill,circle,inner sep=.5pt]{}
		(21,-2)node[fill,circle,inner sep=.5pt]{}
		(22,-2)node[fill,circle,inner sep=.5pt]{}
		(19,-3)node[fill,circle,inner sep=.5pt]{}
		(20,-3)node[fill,circle,inner sep=.5pt]{}
		(21,-3)node[fill,circle,inner sep=.5pt]{}
		(22,-3)node[fill,circle,inner sep=.5pt]{};
		\draw[very thick,red] plot [smooth,tension=.1] coordinates{(20,-2)(19,-2)(18,-3)}; 
		\draw[very thick,red] plot [smooth,tension=.1] coordinates{(17,-3)(15,-3)}; 
		\draw[very thick,red] plot [smooth,tension=.1] coordinates{(17,-2)(15,-2)};
		\draw[very thick,red] plot [smooth,tension=.1] coordinates{(13,-2)(12,-3)(11,-3)};
		\draw[very thick,red] plot [smooth,tension=.1] coordinates{(12,-2)(11,-2)(10,-3)};
	}
	\]

\end{example}
\begin{example}
	Though the edges are described by moving some $P_k$'s to the right, it is not always the case that moving a $P_k$ to the right, even when it's physically possible to do so, describes an edge in the graph. It must additionally hold that the shift of $P_k$ coincides with $k$'th $e$-period $P_k'$ of the resulting abacus. For example, take $P_3$ in the bottom right abacus in the previous example. It can physically slide once to the right. However, doing so, the three beads that move no longer all belong to the same $3$-period $P_3'$ in the resulting abacus, which we draw with its $3$-periods:
	\[\TikZ{[scale=.5]	\draw
	(10,-2)node[fill,circle,inner sep=3pt]{}
	(11,-2)node[fill,circle,inner sep=3pt]{}
	(12,-2)node[fill,circle,inner sep=3pt]{}
	(13,-2)node[fill,circle,inner sep=3pt]{}
	(18,-2)node[fill,circle,inner sep=3pt]{}
	(16,-2)node[fill,circle,inner sep=3pt]{}
	(17,-2)node[fill,circle,inner sep=3pt]{}
	(10,-3)node[fill,circle,inner sep=3pt]{}
	(11,-3)node[fill,circle,inner sep=3pt]{}
	(12,-3)node[fill,circle,inner sep=3pt]{}
	(15,-3)node[fill,circle,inner sep=3pt]{}
	(16,-3)node[fill,circle,inner sep=3pt]{}
	(17,-3)node[fill,circle,inner sep=3pt]{}
	(18,-3)node[fill,circle,inner sep=3pt]{}
	(19,-2)node[fill,circle,inner sep=3pt]{}
	(20,-2)node[fill,circle,inner sep=3pt]{}
	(13,-3)node[fill,circle,inner sep=.5pt]{}
	(14,-3)node[fill,circle,inner sep=.5pt]{}
	(14,-2)node[fill,circle,inner sep=.5pt]{}
	(15,-2)node[fill,circle,inner sep=.5pt]{}
	(21,-2)node[fill,circle,inner sep=.5pt]{}
	(22,-2)node[fill,circle,inner sep=.5pt]{}
	(19,-3)node[fill,circle,inner sep=.5pt]{}
	(20,-3)node[fill,circle,inner sep=.5pt]{}
	(21,-3)node[fill,circle,inner sep=.5pt]{}
	(22,-3)node[fill,circle,inner sep=.5pt]{};
	\draw[very thick,red] plot [smooth,tension=.1] coordinates{(20,-2)(19,-2)(18,-3)}; 
	\draw[very thick,red] plot [smooth,tension=.1] coordinates{(18,-2)(17,-3)(16,-3)}; 
	\draw[very thick,red] plot [smooth,tension=.1] coordinates{(17,-2)(16,-2)(15,-3)};
	\draw[very thick,red] plot [smooth,tension=.1] coordinates{(13,-2)(12,-3)(11,-3)};
	\draw[very thick,red] plot [smooth,tension=.1] coordinates{(12,-2)(11,-2)(10,-3)};
}
\]
	Thus sliding $P_3$ to the right did not describe an edge in the $\slinf$-crystal.
\end{example}

The rule for the edges in the $\slinf$-crystal on abaci which are not totally $e$-periodic involves generalizing the notion of $e$-period into two separate cases, both of which draw from the same repertoire of shapes as $e$-periods: there are $e$-chains called fore periods that are possibly allowed to slide right and satisfy a maximality condition with respect to the total order on beads of the abacus described above; and there are $e$-chains called aft periods that are possibly allowed to slide to the left, satisfying a minimality condition with respect to the clump of beads between adjacent fore periods in which they are situated. The modification of the edge rule in Definition \ref{slinf tot per} consists in specifying that the right shift of the $k$'th fore period should be the $k$'th aft period of the resulting abacus. We refer to \cite[Theorem 4.15]{GerberN}, and simply illustrate the rule here with an example.

\begin{example}
	An edge in the $\slinf$-crystal on $\cF_{e,\bs}$ for $e=5$ and $\bs=(0,4)$, with fore periods in red (solid) and aft  periods in green (dashed). 
		\[
	\TikZ{[scale=.5]
		\draw
		(-11,0)node[fill,circle,inner sep=3pt]{}
		(-10,0)node[fill,circle,inner sep=3pt]{}
		(-9,0)node[fill,circle,inner sep=3pt]{}
		(-11,-1)node[fill,circle,inner sep=3pt]{}
		(-10,-1)node[fill,circle,inner sep=3pt]{}
		(-9,-1)node[fill,circle,inner sep=3pt]{}
		(-8,-1)node[fill,circle,inner sep=3pt]{}
		(-7,-1)node[fill,circle,inner sep=3pt]{}
				(-8,0)node[fill,circle,inner sep=3pt]{}
		(-7,0)node[fill,circle,inner sep=3pt]{}
		(-6,0)node[fill,circle,inner sep=.5pt]{}
		(-6,-1)node[fill,circle,inner sep=.5pt]{}
		(-5,-1)node[fill,circle,inner sep=.5pt]{}
		(-2,-1)node[fill,circle,inner sep=.5pt]{}
		(-1,-1)node[fill,circle,inner sep=.5pt]{}
		(0,-1)node[fill,circle,inner sep=.5pt]{}
		(1,0)node[fill,circle,inner sep=.5pt]{}
		(3,0)node[fill,circle,inner sep=.5pt]{}
		(4,0)node[fill,circle,inner sep=.5pt]{}
		(9,0)node[fill,circle,inner sep=.5pt]{}
		(8,0)node[fill,circle,inner sep=.5pt]{}
		(5,-1)node[fill,circle,inner sep=.5pt]{}
		(6,-1)node[fill,circle,inner sep=.5pt]{}
		(7,-1)node[fill,circle,inner sep=.5pt]{}
		(8,-1)node[fill,circle,inner sep=.5pt]{}
		(9,-1)node[fill,circle,inner sep=.5pt]{}
		(-5,0)node[fill,circle,inner sep=3pt]{}
		(-4,0)node[fill,circle,inner sep=3pt]{}
		(-3,0)node[fill,circle,inner sep=3pt]{}
		(-2,0)node[fill,circle,inner sep=3pt]{}
		(-1,0)node[fill,circle,inner sep=3pt]{}
		(0,0)node[fill,circle,inner sep=3pt]{}
		(2,0)node[fill,circle,inner sep=3pt]{}
		(-4,-1)node[fill,circle,inner sep=3pt]{}
		(-3,-1)node[fill,circle,inner sep=3pt]{}
		(1,-1)node[fill,circle,inner sep=3pt]{}
		(2,-1)node[fill,circle,inner sep=3pt]{}
		(3,-1)node[fill,circle,inner sep=3pt]{}
		(4,-1)node[fill,circle,inner sep=3pt]{}
		(5,0)node[fill,circle,inner sep=3pt]{}
		(6,0)node[fill,circle,inner sep=3pt]{}
		(7,0)node[fill,circle,inner sep=3pt]{};
		\draw[very thick,red] plot [smooth,tension=.1] coordinates{(7,0)(5,0)(4,-1)(3,-1)}; 
		\draw[very thick, green, dashed] plot [smooth,tension=.1] coordinates {(5,.1)(4,-.9)(1,-.9)};
		\draw[very thick,red] plot [smooth,tension=.1] coordinates{(0,0)(-2,0)(-3,-1)(-4,-1)}; 
		\draw[very thick,red] plot [smooth,tension=.1] coordinates{(-7,0)(-11,0)};
		\draw[very thick,red] plot [smooth,tension=.1] coordinates{(-7,-1)(-11,-1)};
				\draw[very thick,green,dashed] plot [smooth,tension=.1] coordinates{(-7,.1)(-11,.1)};
		\draw[very thick,green,dashed] plot [smooth,tension=.1] coordinates{(-7,-.9)(-11,-.9)};
		\draw[very thick,green, dashed] plot [smooth,tension=.1] coordinates{(-1,.1)(-5,.1)};
		\draw[->] plot [smooth] coordinates{(-1,-1.7)(1,-3.3)};
		\draw(-6,-4)node[fill,circle,inner sep=3pt]{}
		(-5,-4)node[fill,circle,inner sep=3pt]{}
		(-4,-4)node[fill,circle,inner sep=3pt]{}
		(-3,-4)node[fill,circle,inner sep=3pt]{}
		(-2,-4)node[fill,circle,inner sep=3pt]{}
		(-6,-5)node[fill,circle,inner sep=3pt]{}
		(-5,-5)node[fill,circle,inner sep=3pt]{}
		(-4,-5)node[fill,circle,inner sep=3pt]{}
		(-3,-5)node[fill,circle,inner sep=3pt]{}
		(-2,-5)node[fill,circle,inner sep=3pt]{}
		(-1,-4)node[fill,circle,inner sep=.5pt]{}
		(-1,-5)node[fill,circle,inner sep=.5pt]{}
		(0,-5)node[fill,circle,inner sep=.5pt]{}
		(1,-5)node[fill,circle,inner sep=.5pt]{}
		(3,-4)node[fill,circle,inner sep=.5pt]{}
		(4,-5)node[fill,circle,inner sep=.5pt]{}
		(5,-5)node[fill,circle,inner sep=.5pt]{}
		(8,-4)node[fill,circle,inner sep=.5pt]{}
		(9,-4)node[fill,circle,inner sep=.5pt]{}
		(10,-5)node[fill,circle,inner sep=.5pt]{}
		(11,-5)node[fill,circle,inner sep=.5pt]{}
		(12,-5)node[fill,circle,inner sep=.5pt]{}
		(13,-5)node[fill,circle,inner sep=.5pt]{}
		(14,-5)node[fill,circle,inner sep=.5pt]{}
		(13,-4)node[fill,circle,inner sep=.5pt]{}
		(14,-4)node[fill,circle,inner sep=.5pt]{}
		(0,-4)node[fill,circle,inner sep=3pt]{}
		(1,-4)node[fill,circle,inner sep=3pt]{}
		(2,-4)node[fill,circle,inner sep=3pt]{}
		(2,-5)node[fill,circle,inner sep=3pt]{}
		(3,-5)node[fill,circle,inner sep=3pt]{}
		(4,-4)node[fill,circle,inner sep=3pt]{}
		(5,-4)node[fill,circle,inner sep=3pt]{}
		(6,-4)node[fill,circle,inner sep=3pt]{}
		(7,-4)node[fill,circle,inner sep=3pt]{}
		(10,-4)node[fill,circle,inner sep=3pt]{}
		(11,-4)node[fill,circle,inner sep=3pt]{}
		(12,-4)node[fill,circle,inner sep=3pt]{}
		(6,-5)node[fill,circle,inner sep=3pt]{}
		(7,-5)node[fill,circle,inner sep=3pt]{}
		(8,-5)node[fill,circle,inner sep=3pt]{}
		(9,-5)node[fill,circle,inner sep=3pt]{};
			\draw[very thick,red] plot [smooth,tension=.1] coordinates{(12,-4)(10,-4)(9,-5)(8,-5)}; 
			\draw[very thick,red] plot [smooth,tension=.1] coordinates{(7,-4)(4,-4)(3,-5)}; 
			\draw[very thick,red] plot [smooth,tension=.1] coordinates{(-2,-4)(-6,-4)};
		\draw[very thick,red] plot [smooth,tension=.1] coordinates{(-2,-5)(-6,-5)};
				\draw[very thick,green,dashed] plot [smooth,tension=.1] coordinates{(10,-3.9)(9,-4.9)(6,-4.9)};
		\draw[very thick,green,dashed] plot [smooth,tension=.1] coordinates{(6,-3.9)(4,-3.9)(3,-4.9)(2,-4.9)};
			\draw[very thick,green,dashed] plot [smooth,tension=.1] coordinates{(-2,-3.9)(-6,-3.9)};
		\draw[very thick,green,dashed] plot [smooth,tension=.1] coordinates{(-2,-4.9)(-6,-4.9)};
	}
\]
\end{example}

We note the following facts about the $\slinf$-crystal on $\cF_{e,\bs}$:
\begin{itemize}
	\item The $\sle$-crystal and the $\slinf$-crystal on $\cF_{e,\bs}$ commute  \cite{Losev}, \cite{Gerber1}.
	\item Each connected component of the $\slinf$-crystal has a unique source vertex, and is isomorphic as a graph to the branching graph of the symmetric groups in characteristic $0$, also known as the Young graph \cite{Gerber2}.
	\item 
In more detail: let $\bla_0$ be a source vertex of both the $\slinf$- and $\sle$-crystals on $\cF_{e,\bs}$. Let $P_1,P_2,\dots,P_k,\dots $ be its $e$-periods. Then each $\bla$ in the connected component of the $\slinf$-crystal with source vertex $\bla_0$ can be described as $\bla=\tilde{a}_\sigma(\bla_0)$, $\sigma\in\cP$, where $\tilde{a}_\sigma$ is the operator that moves $P_1$ to the right $\sigma_1$ times, $P_2$ to the right $\sigma_2$ times, $\dots $, $P_k$ to the right $\sigma_k$ times, $\dots$ \cite{Gerber2}.
	
\end{itemize}

There is a combinatorial criterion for the abacus $\cA|\bla,\bs\rangle$ of a charged bipartition to be a source vertex of the $\slinf$-crystal on $\cF_{e,\bs}$:

\begin{theorem}\label{slinf source}\cite[Theorem 7.13]{GerberN}
The charged bipartition $|\bla,\bs\rangle$ is a source vertex of the $\slinf$-crystal on $\cF_{e,\bs}$ if and only if its abacus $\cA|\bla,\bs\rangle$ avoids the following $e+1$ patterns in the semiinfinite region of the abacus bounded to the right, inclusive, by the column containing the rightmost bead of the first fore period $P_1$: 
$$(1)\;\TikZ{[scale=.5]
	\draw
	(0,1)node[fill,circle,inner sep=.5pt]{}
	(0,0)node[fill,circle,inner sep=.5pt]{}
	;}\quad
(2)\;\TikZ{[scale=.5]\draw
	(1,1)node[fill,circle,inner sep=.5pt]{}
	(0,0)node[fill,circle,inner sep=.5pt]{}
	(1,0)node[fill,circle,inner sep=2pt]{}
	(0,1)node[fill,circle,inner sep=2pt]{}
	;}\quad
(3)\;\TikZ{[scale=.5]\draw
	(-1,0)node[fill,circle,inner sep=.5pt]{}
	(1,1)node[fill,circle,inner sep=.5pt]{}
	(0,0)node[fill,circle,inner sep=2pt]{}
	(-1,1)node[fill,circle,inner sep=2pt]{}
	(1,0)node[fill,circle,inner sep=2pt]{}
	(0,1)node[fill,circle,inner sep=2pt]{}
	;}\quad
(4)\;\TikZ{[scale=.5]\draw
	(-1,0)node[fill,circle,inner sep=.5pt]{}
	(2,1)node[fill,circle,inner sep=.5pt]{}
	(0,0)node[fill,circle,inner sep=2pt]{}
	(-1,1)node[fill,circle,inner sep=2pt]{}
	(1,0)node[fill,circle,inner sep=2pt]{}
	(0,1)node[fill,circle,inner sep=2pt]{}
	(2,0)node[fill,circle,inner sep=2pt]{}
	(1,1)node[fill,circle,inner sep=2pt]{}
	;}\quad\dots\quad
(e+1)\;\TikZ{[scale=.5]\draw
	(-1,0)node[fill,circle,inner sep=.5pt]{}
	(0,0)node[fill,circle,inner sep=2pt]{}
	(-1,1)node[fill,circle,inner sep=2pt]{}
	(0,1)node[fill,circle,inner sep=2pt]{}
(1,0)node[fill,circle,inner sep=2pt]{}
(1,1)node[fill,circle,inner sep=2pt]{};}
\dots\TikZ{[scale=.5]\draw
	(0,1)node[fill,circle,inner sep=2pt]{}
	(0,0)node[fill,circle,inner sep=2pt]{}
	(1,0)node[fill,circle,inner sep=2pt]{}
	(2,0)node[fill,circle,inner sep=2pt]{}
	(1,1)node[fill,circle,inner sep=2pt]{}
	(2,1)node[fill,circle,inner sep=.5pt]{}
	;}
$$
\end{theorem}
The first fore period $P_1$ is the biggest chain of $e$ beads in $\cA|\bla,\bs\rangle$ from among the possible shapes described in the paragraph following Theorem \ref{tot per}. 
\begin{example}
	Let $e=3$. The abacus below is a source vertex of the $\slinf$-crystal but not of the $\widehat{\mathfrak{sl}}_3$-crystal (since it is not totally $3$-periodic). The first fore period $P_1$ is marked with the red line.
		\[\TikZ{[scale=.5]	\draw
		(10,-2)node[fill,circle,inner sep=3pt]{}
		(11,-2)node[fill,circle,inner sep=3pt]{}
		(12,-2)node[fill,circle,inner sep=3pt]{}
		(13,-2)node[fill,circle,inner sep=3pt]{}
		(18,-2)node[fill,circle,inner sep=3pt]{}
		(16,-2)node[fill,circle,inner sep=3pt]{}
		(17,-2)node[fill,circle,inner sep=3pt]{}
		(10,-3)node[fill,circle,inner sep=3pt]{}
		(11,-3)node[fill,circle,inner sep=3pt]{}
		(12,-3)node[fill,circle,inner sep=.5pt]{}
		(15,-3)node[fill,circle,inner sep=3pt]{}
		(16,-3)node[fill,circle,inner sep=3pt]{}
		(17,-3)node[fill,circle,inner sep=3pt]{}
		(18,-3)node[fill,circle,inner sep=3pt]{}
		(19,-2)node[fill,circle,inner sep=3pt]{}
		(20,-2)node[fill,circle,inner sep=3pt]{}
		(13,-3)node[fill,circle,inner sep=.5pt]{}
		(14,-3)node[fill,circle,inner sep=3pt]{}
		(14,-2)node[fill,circle,inner sep=3pt]{}
		(15,-2)node[fill,circle,inner sep=3pt]{}
		(21,-2)node[fill,circle,inner sep=.5pt]{}
		(22,-2)node[fill,circle,inner sep=3pt]{}
		(19,-3)node[fill,circle,inner sep=.5pt]{}
		(20,-3)node[fill,circle,inner sep=3pt]{}
		(21,-3)node[fill,circle,inner sep=.5pt]{}
		(22,-3)node[fill,circle,inner sep=.5pt]{}
		(23,-2)node[fill,circle,inner sep=.5pt]{}
		(23,-3)node[fill,circle,inner sep=.5pt]{};
		\draw[very thick,red] plot [smooth,tension=.1] coordinates{(20,-2)(19,-2)(18,-3)}; 
	}
	\]
\end{example}

\section{Finite unitary groups and their unipotent representations in characteristic $\ell$}\label{unipunit}
\subsection{Harish-Chandra theory} 

Harish-Chandra theory was introduced by Harish-Chandra \cite{HC}, studied in characteristic $0$ by 
Howlett-Lehrer \cite{HowlettLehrer}, and in characteristic $\ell$ coprime to $q$ by Geck, Hiss, and Malle \cite{Hiss}, \cite{GHM1}, \cite{GHM2}. We refer to \cite{DigneMichel} for an introduction to Harish-Chandra theory. For $\rL\leq \rP\leq \rG$ a rational Levi subgroup of a rational parabolic subgroup $\rP$ of an algebraic group $\rG$ defined over $\mathbb{F}_q$ with respect to a Frobenius endomorphism $F$, there is a parabolic induction functor $R_{\rL^F}^{\rG^F}$ from $\rL^F$-representations to $\rG^F$-representations  given by inflating a representation from $\rL^F$ to $\rP^F$ by having the unipotent radical $\rU^F$ of $\rP^F$ act by $1$, and then inducing the inflated module to $\rG^F$. The adjoint functor $^*R_{\rL^F}^{\rG^F}$ is called Harish-Chandra restriction. In characteristic $0$ or characteristic $\ell$ coprime to $q$, the functors $R_{\rL^F}^{\rG^F}$ and $^*R_{\rL^F}^{\rG^F}$  are independent of the choice of $\rP$ \cite[Theorem 1.1]{HowlettLehrer2},\cite{DipperDu1}, and they are exact, biadjoint, and transitive with respect to inclusion of Levi subgroups. 
 
A cuspidal representation is a simple representation which does not appear in the head of any Harish-Chandra induced simple representation from a proper Levi subgroup. Equivalently (by adjointness of $R_{\rL^F}^{\rG^F}$ and $^*R_{\rL^F}^{\rG^F}$), a $\mathbbm{k}\rG$-module $Y$ is cuspidal if and only if $^*R_{\rL^F}^{\rG^F}Y=0$ for all proper rational Levi subgroups $\rL<\rG$. In characteristic $\ell$, the Harish-Chandra induction of a simple module $X$, $R_{\rL^F}^{\rG^F} X$, will not be semisimple in general. However, there is an isomorphism of $\mathbbm{k}\rG$-modules $R_{\rL^F}^{\rG^F} X\cong Z_1\oplus Z_2\oplus \dots \oplus Z_r$ for some $r>0$ such that each direct summand $Z_i$ is indecomposable and has a simple head $Y_i$ \cite[Theorem 2.4]{GHM2}. In this case, we say that $Y_i$ is in the same Harish-Chandra series as $X$. If $X$ is cuspidal, then $(L,X)$ is called a cuspidal pair and is said to be the cuspidal support of $Y_i$. Each simple module has a unique cuspidal support up to conjugacy \cite{Hiss}, and so the Harish-Chandra series partition the irreducible representations of $\mathbbm{k}\rG$-mod. 

\subsection{Finite unitary groups}
Let $q=p^r$ for some prime $p$ and some integer $r\geq 1$. Let $F_q:\GL_n(\overline{\mathbf{F}_q})\rightarrow \GL_n(\overline{\mathbf{F}_q})$ be the standard Frobenius endomorphism raising the entries of a matrix to the $q$'th power: $F_q(a_{ij})=(a_{ij}^q)$. Then define $F:\GL_n(\overline{\mathbf{F}_q})\rightarrow \GL_n(\overline{\mathbf{F}_q})$ by $F(a_{ij})=((a_{ij}^q)^{-1})^{\mathrm{tr}}$. The finite unitary group is given by the fixed points of $F$: $$\GU_n(q):=\GL_n(\overline{\mathbf{F}_q})^F$$
and $\GU_n(q)$ is naturally a subgroup of $\GL_n(q^2):=\GL_n(\overline{\mathbf{F}_q})^{F_{q^2}}$  since $F^2=F_{q^2}$ \cite[Example 21.2]{MalleTesterman}.

The finite unitary group has a split $BN$-pair; the Frobenius fixed points of the Weyl group of $\GL_n(\overline{\mathbb{F}_q})$ is the Coxeter group $B_m$ where $n=2m+\iota$, $\iota\in\{0,1\}$.
The conjugacy classes of rational Levi subgroups of $\GU_n(q)$ are in bijection with conjugacy classes of parabolic subgroups of $B_m$, which in turn are parametrized by subsets $I\subseteq S$ of the set of simple reflections $S\subset B_m$. These conjugacy classes of parabolic subgroups of $B_m$ are of the form $B_{m'}\times S_{\lambda_1}\times S_{\lambda_2}\times \dots \times S_{\lambda_s}$ where $m'\leq m$ and $\lambda=(\lambda_1,\lambda_2,\dots,\lambda_s)$ is a partition of $m-m'$. The corresponding conjugacy classes of Levi subgroups $L\leq \GU_{2m+\iota}(q)$ are of the form $$L=\GU_{2m'+\iota}(q)\times \GL_{\lambda_1}(q^2)\times\GL_{\lambda_2}(q^2)\dots\times \GL_{\lambda_s}(q^2).$$
Because of the ``twisted" nature of $\GU_n(q)$, there are not inclusions of Levi subgroups of rational parabolics between $\GU_n(q)$ and $\GU_{n+1}(q)$. Rather, $\GU_n(q)\subset \GU_{n+2}(q)$ with $\GU_n(q)\times \GL_1(q^2)$ being a Levi subgroup in $\GU_{n+2}(q)$. Thus there are two ``towers of algebras" of finite unitary groups, one for even $n$ and one for odd $n$: $$\GU_\iota(q)\subset \GU_{2+\iota}(q)\subset\dots\subset \GU_{2m+\iota}(q)\subset\GU_{2m+2+\iota}(q)\subset\dots $$
for $\iota\in\{0,1\}$.
  We will write $R_{n}^{n+2}$ as shorthand for the functor $R_{\GU_n(q)\times \GL_1(q^2)}^{\GU_{n+2}(q)}(- \otimes \st_1)$ where $\st_1$ is the trivial rep of $\GL_1(q^2)$.

\subsection{Unipotent representations of finite unitary groups and their combinatorics}\label{unip unitary combinat}
Let $\ell$ be a prime not dividing $q$. Fix an $\ell$-modular system $(K,\mathcal{O},\mathbbm{k})$ such that $K$ and $\mathbbm{k}$ are splitting fields for all subgroups of $\GU_n(q)$. A simple $K\GU_n(q)$-representation is called unipotent if its character is a constituent of a virtual Deligne-Lusztig character of the form $R_{\rT}^{\rG}(\mathrm{Id}_{\rT})$ where $\rT$ is a rational maximal torus \cite[Definition 13.19]{DigneMichel}. 
As in \cite{GHJ} we say that a simple $\mathbbm{k}\GU_n(q)$-representation is unipotent if it belongs to a block containing (the mod $\ell$ reduction of) a unipotent character; such a block is called a unipotent block. Let $\mathbbm{k}\GU_n(q)\mathrm{-mod}^{\mathrm{unip}}$ denote the direct sum of all unipotent blocks of $\mathbbm{k}\GU_n(q)$-mod. We remark that the word ``unipotent" is ambiguous, referring to ordinary (i.e. characteristic $0$) characters in some contexts and to irreducible Brauer characters or the modules with those characters in other contexts. 
 On the flip side, if we call a module unipotent then in particular it is a simple module in one of the two worlds, characteristic $0$ or characteristic $\ell>0$. 
 Hopefully it will be clear from the context which usage of unipotent is meant. In particular when we write ``$[-]$ is the unipotent part of a projective character" we always mean that $[-]$ is a positive linear combination of ordinary unipotent characters $\lambda$, given by summing relevant columns of the square decomposition matrix. This is simply working with the square decomposition matrix as if it were the decomposition matrix of a highest weight category, ignoring the non-unipotent characters.

The unipotent representations of $\GU_n(q)$ in characteristic $0$ and in characteristic $\ell$ are both labeled by partitions of $n$ \cite{LusztigSrinivasan},\cite{Geck}.  We will often drop the parentheses around partitions and use exponential notation, writing e.g. $21^3=(2,1,1,1)$, and we will freely identify partitions with unipotent $K\GU_n(q)$-representations. If $\lambda$ is a partition of $n$, let $X_\lambda$ be the unipotent $\mathbbm{k}\GU_n(q)$-representation labeled by $\lambda$, let $P_\lambda$ be the projective cover of $X_\lambda$, and let $Y_\lambda$ be the reduction modulo $\ell$ of an $\cO$-lattice in $\lambda\in K\GU_n(q)\mathrm{-mod}^{\mathrm{unip}}$. 
 When we work at the level of characters, the trio $X_\lambda$, $P_\lambda$, and $Y_\lambda$ play similar roles to the simple, projective indecomposable, and standard modules, respectively, in a highest weight category. In particular, the decomposition matrix is lower triangular with $1$'s on the diagonal when the partitions of $n$ are ordered lexicographically with $(n)$ being biggest and labeling the top (leftmost) row (column) \cite{Geck}. The decomposition numbers $d_{\lambda,\mu}$, which are the entries of this matrix, satisfy Brauer-Humphreys reciprocity: for any partitions $\lambda$ and $\mu$ of the same size, $d_{\lambda,\mu}:=[Y_\lambda:X_\mu]=[P_\mu:Y_\lambda]$. The unipotent blocks of $\mathbbm{k}\GU_n(q)$ are labeled by $e$-cores: $X_\lambda$ and $X_\mu$ are in the same block if and only if $\lambda$ and $\mu$ have the same $e$-core \cite{FongSrinivasan}.

In characteristic $0$, the unipotent $K\GU_n(q)$-representation $\lambda$ is cuspidal if and only if $\lambda=\Delta_t:=(t,t-1,t-2,\dots,2,1)$ for some $t\in\N_0$ \cite{Lusztig}. The $\ell$-reduction of $\Delta_t$ remains simple \cite{DMCusp}: $[\Delta_t:X_\mu]=\delta_{\mu,\Delta_t}$. 
Recall that the $2$-core of a partition $\lambda$ is obtained by successively removing vertical or horizontal dominoes from the rim of $\lambda$ so long as the boxes remaining still form the Young diagram of a partition, and doing this until no more dominoes can be removed. The $2$-core partitions are exactly the triangular partitions $\Delta_t$ for $t=0,1,2,\dots$ and thus the $2$-core partitions $\Delta_t$ such that $\frac{t(t+1)}{2}\leq n$ and  $\frac{t(t+1)}{2}=n\mod 2$ label the Harish-Chandra series of unipotent representations of $K\GU_n(q)$. The $2$-quotient of $\lambda$ is a bipartition $\lambda^1.\lambda^2$, which we define according to the following convention ($d$-quotients only being uniquely defined up to a cyclical permutation of the $\lambda^j$'s for every $\lambda\in\cP(n)$). First, write the $\beta$-numbers of $\lambda$:
$$\beta(\lambda)=\{\lambda_k-k+1\mid k\geq 1\}$$
and then write $\beta(\lambda)$ as a disjoint set of even and odd integers: $\beta(\lambda)=\beta^{\mathrm{even}}(\lambda)\sqcup\beta^{\mathrm{odd}}(\lambda)$. 
Writing $\beta^{\mathrm{even}}(\lambda)=\{2\beta(\lambda^2)\}$ and $\beta^{\mathrm{odd}}(\lambda)=\{2\beta(\lambda^1)-1\}$, the $2$-quotient of $\lambda$ is then $\lambda^1.\lambda^2$. This formula agrees with the implementation of $2$-quotients in Sage \cite{sage}. Since $|\lambda|=2|2\hbox{-quotient}(\lambda)|+|2\hbox{-core}(\lambda)|$ and there is at most one $2$-core of size $n$ for any $n\in\N_0$, the $2$-core of a partition is determined once its $2$-quotient is known.

\begin{example}
Let $\lambda=21^4$. Then $\beta(\lambda)=\{2,0,-1,-2,-3,-5,-6,-7,-8,-9,\dots\}$. To obtain $\lambda^2$, we write $\beta^{\mathrm{even}}(\lambda)=\{2,0,-2,-6,-8,-10,\dots\}\rightsquigarrow\{1,0,-1,-3,-4,-5,\dots\}\rightsquigarrow \lambda^2=1^3$. To obtain $\lambda^1$, we write $\beta^{\mathrm{odd}}(\lambda)=\{-1,-3,-5,-7,-9,\dots\}\rightsquigarrow\{0,-1,-2,-3,-4,\dots\}\rightsquigarrow\lambda^1=\emptyset$. The $2$-core of $\lambda$ is $\emptyset$.
\end{example}

Let $e$ be the order of $-q$ mod $\ell$. When $e$ is even, a complete description of the Harish-Chandra series of simple unipotent $\mathbbm{k}\GU_n(q)$-modules was conjectured in \cite[Conjecture 9.2]{GHM1}, verified for blocks with cyclic defect groups in \cite[Proposition 9.3]{GHM1}, and verified in full generality in \cite{GruberHiss}.
The description can be framed in terms of the combinatorics of a product of level $1$ Fock spaces. The case that is not fully understood yet is when $e$ is odd and $e\geq 3$. Each $\Delta_t$ yields a level $2$ Fock space $\mathcal{F}_{e,t}$ of rank $e$ and charge $(s_1,s_2)$ determined by $t$ and $e$, with the bipartitions labeling the basis of the Fock space given by the ``twisted $2$-quotients" of all partitions with $2$-core $\Delta_t$. 
Let $\lambda$ be a partition with $2$-core $\Delta_t$ and $2$-quotient $\lambda^1.\lambda^2$. Then the bijection $\tau$ between partitions $\lambda$ with $2$-core $\Delta_t$ and the charged bipartitions in $\mathcal{F}_{e,t}$ is given by:

$$\tau:\;\lambda\mapsto \begin{cases}
|\lambda^1.\lambda^2,(-(t+1+e)/2,\;t/2)\rangle \quad \hbox{ if }t\hbox{ is even}\\
|\lambda^2.\lambda^1,(-(t+1)/2,\;(t+e)/2)\rangle \quad \hbox{ if }t\hbox{ is odd}
\end{cases}
$$
In particular, in this convention using the twisted $2$-quotient we have $$s_2-s_1=\frac{e+1}{2}+t$$ and for computing crystals on the Fock space $\mathcal{F}_{e,t}$, the value of $s_2-s_1$ is all that matters up to a relabeling of the arrows. In our formula, we have taken the formula from \cite{DVV1} and in the case of $t$ odd, we have  applied the crystal-preserving isomorphism $|\lambda^1.\lambda^2,(s_1,s_2)\rangle\mapsto|\lambda^2.\lambda^1,(s_2,s_1+e)\rangle$. We call this the twisted $2$-quotient, following \cite{GHJ}. We remark that the charge for the Fock space looks different in \cite{GHJ} due to opposite conventions for taking the untwisted $2$-quotient, but once this is taken into account it is equivalent. We write $\bs_t$ for the charge $(-(t+1+e)/2,t/2)$ if $t$ is even, and $(-(t+1)/2,(t+e)/2)$ if $t$ is odd.

Given an arbitrary charged bipartition $|\bmu,\bs_t\rangle=|\mu^1.\mu^2,\bs_t\rangle$ with charge corresponding to some $\Delta_t$,
the partition $\lambda$ such that $\tau(\lambda)=|\bmu,\bs_t\rangle$ can be recovered as follows \cite[Lemma 7.2]{GHJ}: let $\mathfrak{A}_{e,t}(\bmu,\bs_t)=\{2j\mid j\in\beta^1\}\cup\{2j-e\mid j\in\beta^2\}$. Then $\lambda$ is the partition whose $1$-line abacus is given by $\mathfrak{A}_{e,t}(\bmu,\bs_t)$, up to a global shift of the $\beta$-numbers.

Write $n=2m+\iota$ where $\iota\in\{0,1\}$. If $\iota=0$, the principal series representations of $K\GU_n(q)$ are labeled by partitions $\lambda$ of $n$ whose $2$-core is $\emptyset=\Delta_0$. If $\iota=1$, the principal series representations of $K\GU_n(q)$ are labeled by $\lambda$ with $2$-core $(1)=\Delta_1$. We write $X_\iota$ for the unique unipotent module of $\mathbbm{k}\GU_\iota(q)$, where $X_0$ is simply the trivial representation of the trivial group, an honorary unipotent module, and $X_1$ is the trivial representation of $\GU_1(q)$. The subgroup $\GL_r(q^2)\subset\GU_n(q)$ for $r\leq m$ also has unipotent representations labeled by partitions (in this case, partitions of $r$), and all unipotent $K\GL_r(q^2)$-representations belong to the principal series. Using $\tau$, we can do computations involving Harish-Chandra induction by working with type $B$ Coxeter groups and their parabolic subgroups. Let's say $\rL=\GU_{2m+\iota}(q)\times \GL_{r}(q^2)$ and $\rG=\GU_{2m+\iota+2r}(q).$ Harish-Chandra induction preserves the principal series from characteristic $0$, i.e. it preserves the set of ordinary unipotent characters $\lambda$ with $2$-core $\Delta_0$ if $\iota=0$ and $2$-core $\Delta_1$ if $\iota=1$, so to find $R_{\rL}^{\rG}\mu\otimes\lambda$ at the level of the Grothendieck group we just compute $\Ind_{B_m\times S_r}^{B_{m+r}}\tau(\mu)\otimes\lambda$ and then apply $\tau^{-1}$ with $t=0$ if $\iota=0$ and $t=1$ if $\iota=1$.

We will refer to $(1^n)\in K\GL_n(q^2)\hbox{-mod}^{\mathrm{unip}} $ as the Steinberg representation and we will denote it by $\St_n$. Hopefully this will avoid confusion over what group we are working in, since unipotent representations of $\GL_n(q^2)$ and $\GU_n(q)$ are both labeled by partitions of $n$. We will denote the unipotent $\mathbbm{k}\GL_n(q^2)$-representation $X_{(1^n)}$ by $\st_n$. We always have $[P_{\st_n}: \lambda]=\delta_{(1^n),\lambda}$. Thus $\St_n$ is the unipotent part of a projective character. This means it plays a similar role to a standard module that is also projective in a highest weight category.

\subsection{The $\slinf$-crystal versus blocks} Let $\lambda\in\cP$ be a partition with $2$-core $\Delta_t$ for some $t\in\N_0$. Write $\tilde{a}_\sigma(\lambda)$ as shorthand for $\tau^{-1}\left(\tilde{a}_\sigma\left(\tau(\lambda)\right)\right)$, i.e. $\tilde{a}_\sigma(\lambda)$ is the partition with $2$-core $\Delta_t$ whose twisted $2$-quotient is $\tilde{a}_\sigma(\tau(\lambda))$.
\begin{lemma}\label{blocks slinf} 
	The $e$-core of $\tilde{a}_\sigma(\lambda)$ is the same as the $e$-core of $\lambda$.
\end{lemma}
\begin{proof}	
	Write $\tau(\lambda)=|\lambda^1.\lambda^2,\bs_t\rangle$ and let $\cA$ be the abacus of $\tau(\lambda)$. Let $P\subset\cA$, $P=\{(x_1,j_1),(x_2,j_2),\dots, (x_e,j_e)\}$, be a collection of $e$ beads of $\cA$ satisfying the second and third conditions of Definition \ref{JL e-period def}.
	 We will show that moving $P$ one place to the right does not change the $e$-core of the corresponding partition. Suppose $P$ can move to the right, so if $j_1=\dots= j_i=2$ and $j_{i+1}=\dots= j_e=1$ for some $i\in\{0,1,\dots,e\}$, then $(x_1+1,j_1)\notin\cA$ and $(x_{i+1}+1,1)\notin\cA$. Moving $P$ one place to the right put more pedantically consists of two steps: first, removing all $e$ beads $(x_b,j_b)\in P$ from $\cA$ for $b\in\{1,\dots,e\}$; second adding the beads $(x_{b+1},j_b)$ to $\cA\setminus P$. This is the same as just removing the bead $(x_e,j_e)$ from $\cA$ and adding the bead $(x_1+1,j_1)$ to $\cA$, and if $i\neq 0,e$, also removing $(x_i,j_2)$ from $\cA$ and adding the bead $(x_i,j_1)$ to $\cA$. Or in other words, if $P$ is not a horizontal strip so $i\neq 0,e$ then moving $P$ one place to the right is the same as moving its leftmost bead in the bottom row $e$ spaces to the right and up to the top row, and moving its leftmost bead in the top row down to the bottom row. If $P$ is a horizontal strip so $i=0$ or $e$, then moving $P$ one place to the right is the same as moving its leftmost bead $e$ places to the right in the same row. This is again composite of the same two types of moves: (i) moving a bead from row $2$ to row $1$ in the same column, or (ii) moving a bead from row $1$ to the column $e$ spaces to the right and up to row $2$. In \cite[Section 7.3]{GHJ}, the inverse moves to the moves of types (i) and (ii) are called elementary operations. It is shown in \cite[Proposition 7.3]{GHJ} that elementary operations preserve $e$-cores. Therefore moving $P$ one place to the right produces the twisted $2$-quotient of a partition with the same $e$-core as $\lambda$. Any edge in the $\slinf$-crystal is given by moving some subset $P\subset\cA$ as above one place to the right. Therefore $\tilde{a}_\sigma(\lambda)$, which is obtained by following a path of edges in the $\slinf$-crystal, is a partition with the same $e$-core as $\lambda$.
\end{proof}

\subsection{Projective covers and cuspidal supports}
We recall some classical statements about the relationship between the cuspidal support of a simple module and the behavior of its projective cover under Harish-Chandra induction. The proofs are straightforward and can be found for example in \cite[Section 10]{DudasMichel}. Given  $\mathbbm{k}\GU_n(q)$-modules $M$ and $N$, we write $M\mid N$ if $M$ is a direct summand of $N$. Let $G$ be a finite group of Lie type and let $L$ be a standard Levi subgroup of $G$.

\begin{lemma}\label{proj lemma 0}
	Let $X$ be a cuspidal simple $\mathbbm{k}L$-module and let $Y$ be a simple $\mathbbm{k}G$-module such that $R_L^G X\twoheadrightarrow Y$. Let $P_X$ and $P_Y$ be the projective covers of $X$ and $Y$, respectively. Then $P_Y\mid R_L^G P_X$.
\end{lemma}

\begin{lemma}\label{proj lemma 1}
	Suppose $X$ is a simple $\mathbbm{k}\rL$-module and let $P$ be its projective cover. If $Q$ is a projective indecomposable $\mathbbm{k}G$-module such that $Q\mid R_{\rL}^{\rG} P$, then the simple head $Y$ of $Q$ has cuspidal support $(\rM, Z)$ for some standard Levi subgroup $\rM \subseteq \rL$ and some cuspidal simple $\mathbbm{k}M$-module $Z$.
\end{lemma}
\begin{lemma}\label{proj lemma 1.5}
 Suppose $X_1,X_2\in\mathbbm{k}\rL$-mod are cuspidal simple modules with projective covers $P_1,P_2$. Suppose 
 $Y\in\mathbbm{k}\rG$-mod has cuspidal support $(\rL,X_1)$ and let $Q$ be the projective cover of $Y$.
 If $Q\mid R_{\rL}^{\rG}P_2$ then $X_1$ is $G$-conjugate to  $X_2$.
\end{lemma} 

\subsection{Harish-Chandra theory for $\GU_n(q)$ in characteristic $\ell$ and the weak branching graph} In this section, we summarize the work of Gerber-Hiss-Jacon \cite{GHJ}, Gerber-Hiss \cite{GH}, and Dudas-Varagnolo-Vasserot \cite{DVV1} on the weak Harish-Chandra branching graph of the unipotent category of the finite unitary groups in characteristic $\ell$ for a unitary prime $\ell$. We assume $\ell$ does not divide $q^2-1$. The case that $e$ is odd and the order of $-q \mod \ell$ is equal to $e$ is called the unitary prime case, and is the case that has presented the most difficulty. The case of a linear prime $\ell$, which is the case that $e$ is even, is described in terms of representation theory of $\GL_n(q)$, the combinatorics breaking into a product of level $1$ Fock spaces \cite{DipperDu},\cite{GruberHiss},\cite{GHM1}. 
\subsubsection{Possible cuspidal supports of $\mathbbm{k}\GU_n(q)$-modules}
The standard Levi subgroups of $\GU_n(q)$ are of the form 
$L=\GU_{n'}(q)\times \GL_{\alpha_1}(q^2)\times \GL_{\alpha_2}(q^2)\times\dots\times \GL_{\alpha_s}(q^2)$ such that  $n \equiv n'\mod 2$, $0\leq n'\leq n$, $s\geq 0$, and $(\alpha_1,\alpha_2,\dots,\alpha_s)$ is a partition of $\frac{n-n'}{2}$. The results of Dipper-Du on cuspidal supports of unipotent representations of finite general linear groups state that the only Levi subgroups as above that can possibly afford a cuspidal representation when $\ell>n/e$ are those such that $\alpha_i\in\{1,e\}$ for all $i=1,\dots, s$ \cite{DipperDu}. The only cuspidal unipotent representation of $\GL_1(q^2)$ is $\st_1$ (it is the only unipotent representation of $\GL_1(q^2)$ and equal to the trivial representation), and the only cuspidal unipotent representation of $\GL_e(q^2)$ is $\st_e$ \cite{DipperDu}. Thus the cuspidal support of an arbitrary $\mathbbm{k}\GU_n(q)$-module is of the form $\left(\GU_{n'}(q)\times\GL_e(q^2)^{\times k}\times \GL_1(q^2)^{\times r}, X_{\lambda_0}\otimes \st_e^{\otimes k}\otimes \st_1^{\otimes r}\right)$. The problem is then to identify $n'$, $k$, $r$ and $\lambda_0$ given $\lambda$.

 In the case that $G=\GU_n(q)$, Lemma \ref{proj lemma 1.5} may be strengthened.
\begin{lemma}\label{proj lemma 2} Let $L=\GU_{n'}(q)\times \GL_e(q^2)\times \dots \times \GL_e(q^2)\leq G$ be a standard Levi subgroup of $G=\GU_n(q)$. 
	Suppose $X_1,X_2\in\mathbbm{k}\rL$-mod are cuspidal unipotent modules with projective covers $P_1,P_2$. Suppose 
	$Y\in\mathbbm{k}\rG$-mod has cuspidal support $(\rL,X_1)$ and let $Q$ be the projective cover of $Y$.
	If $Q\mid R_{\rL}^{\rG}P_2$ then $X_1\cong X_2$.
\end{lemma}
\begin{proof}
	By Lemma \ref{proj lemma 1.5}, $X_1$ is $G$-conjugate to $X_2$. The group $\GL_e(q^2)$ has a unique cuspidal unipotent module over $\mathbbm{k}$, namely $\st_e$. Write $X_i=X_{\mu_i}\otimes \st_e\otimes \dots \otimes \st_e$ for $\mu_i\vdash n'$, $i=1,2$. Thus $X_1\cong X_2$ if and only if $X_{\mu_1}\cong X_{\mu_2}$. 
	 This follows by uniqueness of the weight spaces of cuspidals under the categorical $\sle$-action established in \cite{DVV1}.
\end{proof}

\subsubsection{Weak branching} Fix $\iota\in\{0,1\}$ and an odd integer $e\geq 3$. The branching along the tower of groups $\GU_{2m+\iota}$, $m\in\N_0$, categorifies the $\sle$-crystal on the sum of Fock spaces $\cF_{e,\bs_t}$ corresponding to the different $2$-cores $\Delta_t$. Each Fock space corresponds to a Harish-Chandra series in characteristic $0$ and the $\sle$-crystal preserves each Fock space $\cF_{e,\bs_t}$. More precisely, consider the Levi subgroup $\GU_n(q)\times \GL_1(q^2)\subset \GU_{n+2}(q)$ and let $X_\lambda$ be a unipotent representation of $\mathbbm{k}\GU_n(q)$, $\lambda\in\cP(n)$. Then $R_{n}^{n+2}X_\lambda=\bigoplus Z_\mu$ for some indecomposable modules $Z_\mu\in\mathbbm{k}\GU_{n+2}$-mod and partitions $\mu\in\cP(n+2)$, such that $Z_\mu$ has simple head $X_\mu$ for each $Z_\mu$ appearing as a summand. Now write $\cP^\iota\subset\cP$ for the subset of all partitions of size congruent to $\iota$ mod $2$. The weak Harish-Chandra branching graph is the simple directed graph with vertices $\{\lambda\in\cP^\iota\}$ and arrows given by $\lambda\rightarrow \mu$ if and only if $|\mu|=|\lambda|+2$ and $R_{n}^{n+2}X_\lambda\twoheadrightarrow X_\mu$ \cite[Section 4.2]{GHJ}. By applying the map $\tau$ to all partitions in $\cP^{\iota}$, the weak Harish-Chandra branching graph yields a simple directed graph on $\bigoplus_{t\geq 0}\mathcal{F}_{e,\bs_t}$.
\begin{theorem}\label{weakHC crystal} \cite[Conjecture 5.7]{GHJ},\cite[Theorem 5.5]{GH},\cite[Theorem B]{DVV1}
	The weak Harish-Chandra branching graph is isomorphic to the $\sle$-crystal on $\bigoplus_{t\geq 0}\mathcal{F}_{e,\bs_t}$.
\end{theorem}
The proof of Theorem \ref{weakHC crystal} in full generality was given by Dudas-Varagnolo-Vasserot \cite[Theorem B]{DVV1}. The proof proceeds by constructing a categorical action of $\sle$ on the category $\cU^\iota:=\bigoplus \mathbbm{k}\GU_{2m+\iota}(q)\mathrm{-mod}^{\mathrm{unip}}$ \cite[Theorem A]{DVV1}. 
This $\sle$-categorification breaks the Harish-Chandra induction and restriction functors into a direct sum of functors $F_i$ and $E_i$ respectively, called $i$-induction and $i$-restriction functors:
\[ R_{2n+\iota}^{2n+2+\iota}\simeq \bigoplus_{i\in\Z/e\Z}F_i,\qquad\qquad {^*R}_{2n+\iota}^{2n+2+\iota}\simeq \bigoplus_{i\in\Z/e\Z}E_i.\]
Given a unipotent $\mathbbm{k}\GU_n(q)$-module $X_\lambda$, Chuang-Rouquier's theory of categorical $\sle$-actions then implies that for each $i\in\Z/e\Z$, $F_i(X_\lambda)$ has simple head $X_{\tilde{f_i}(\lambda)}$ if $\tilde{f}_i(\lambda)\neq 0$, and otherwise $F_i(X_\lambda)=0$ \cite{ChuangRouquier}. Here, $\tilde{f}_i$ is the $\sle$-crystal operator that adds a good $i$-box. Thus the Harish-Chandra branching of simple representations along the tower  $\GU_\iota(q)\subset \GU_{2+\iota}(q)\subset\dots\subset\GU_{2m+\iota}(q)\subset\GU_{2m+2+\iota}(q)\subset\dots$ can be studied by doing combinatorics with the $\sle$-crystal.

The $\sle$-crystal can answer some questions such as when the Harish-Chandra induction of a simple representation is indecomposable in terms of combinatorics. We illustrate with some elementary statements about indecomposability of Harish-Chandra induced representations that follow immediately from Theorem \ref{weakHC crystal}:

\begin{lemma}
Suppose $X_\lambda$ is a unipotent $\mathbbm{k}\GU_n(q)$-module, $\lambda$ a partition of $n$. The following statements are true:
\begin{enumerate}
\item $R_{n}^{n+2}X_\lambda$ is indecomposable if and only if $\tau(\lambda)$ has only one good addable box. 
\item Suppose that $X_\lambda$ is weakly cuspidal. Then $R_{n}^{n+2}X_\lambda$  is semisimple with two nonisomorphic simple summands if and only if $\tau(\lambda)$ has good addable boxes of different residues.
\item 
Suppose that $X_\lambda$ is weakly cuspidal. Then $R_{n}^{n+4}X_\lambda$ is always decomposable.
\end{enumerate}
\end{lemma}\begin{proof} (1) Immediate since $R_{n}^{n+2}\simeq\bigoplus_{i\in \Z/e\Z}\mathrm{F}_i$ and $\mathrm{F}_i(X_\lambda)$ has a simple head or is $0$ for each $i\in \Z/e\Z$.
(2) By \cite[Theorem 3.2]{GHJ} we know that $\End R_n^{n+2m}X_\lambda$ is a Hecke algebra of type $B_m$. Taking $m=1$ yields the analogue of \cite[Lemma 3.15]{GHM2}: either $R_{n}^{n+2}X_\lambda$ is indecomposable (and not simple) or it is semisimple with two nonisomorphic simple summands. We have just rephrased this statement in terms of the charged bipartition $\tau(\lambda)$, applying part (1).
 (3) Harish-Chandra induction is transitive, so $R_n^{n+4}X_\lambda=R_{n+2}^{n+4}R_n^{n+2}X_\lambda$ so for $R_{n}^{n+4}X_\lambda$ to be indecomposable, $R_n^{n+2}X_\lambda$ needs to be indecomposable. Now observe that the number of addable boxes minus the number of removable boxes of any partition is equal to $1$. In a bipartition there are therefore two more addable boxes than removable boxes. Since $\tau(\lambda)$ has only one good addable box and $X_\lambda$ is weakly cuspidal, every removable box in $\tau(\lambda)$ cancels some addable box in the Kashiwara $i$-words of $\tau(\lambda)$, and the two leftover addable boxes, call them $b_1$ and $b_2$, have the same residue $i$ mod $e$. The good addable $i$-box is the box among $b_1$, $b_2$ which is larger in the order we have put on charged bipartitions; say it is $b_1$. Now we add $b_1$ to $\lambda$. Now $b_2$ is a good addable $i$-box for $\tau(\lambda)\cup b_1$. On the other hand, $\tau(\lambda)\cup b_1$ has a new addable $(i-1)$-box, right below $b_1$. This forces $\tau(\lambda)\cup b$ to have a good addable $(i-1)$-box, since the number of $+$'s in the $(i-1)$-word must now outstrip the number of $-$'s. Thus $\tau(\lambda)\cup b_1$ has good addable boxes of different residues, so by part (1), $R_n^{n+4} X_\lambda$ is decomposable. 
\end{proof}
With much more work and attention to the details of the combinatorics, a better version of statement (1) can be proved: in fact, if $R_n^{n+2}X_\lambda$ is decomposable, then any two non-isomorphic simple submodules of it lie in different blocks \cite[Corollary 7.9]{GHJ}.

\subsection{The combinatorial classification of cuspidals} We continue with the same assumptions about $e$ and $\ell$, in particular $e$ is odd and at least $3$. 
\begin{theorem}\label{dvv cusp} \cite[Theorem 5.10]{DVV2} Let $\lambda\in\cP$ and let $\Delta_t$ be its $2$-core. The unipotent $\mathbbm{k}\GU_n(q)$-module $X_\lambda$ is cuspidal if and only if the charged bipartition $\tau(\lambda)$ is a source vertex of both the $\sle$- and $\slinf$-crystals on $\mathcal{F}_{e,\bs_t}$.
\end{theorem}

Using Theorems \ref{tot per} and \ref{slinf source}, we can check if $X_\lambda$ is cuspidal by checking (i) if the abacus $\cA|\tau(\lambda),\bs_t\rangle$ is totally $e$-periodic, making it a source vertex of the $\sle$-crystal and (ii) if $\cA|\tau(\lambda),\bs_t\rangle$ satisfies the pattern avoidance condition that makes it a source vertex of the $\slinf$-crystal. In practice, we usually eyeball cuspidality by drawing the $e$-periods; if no beads are left out and all the $e$-periods are jammed together so that none of them can even slide left, then Theorems \ref{tot per} and \ref{slinf source} guarantee that the abacus labels a cuspidal. It is only the completely horizontal $e$-periods that have to be dealt with carefully. 
\begin{remark}Level-rank duality offers an elegant alternative way to check cuspidality: $|\bla,\bs\rangle$ is a source vertex of both the $\sle$- and $\slinf$-crystals if and only if its level-rank dual is a FLOTW $e$-partition \cite[Theorem 7.7]{Gerber1}.\end{remark}

Theorem \ref{dvv cusp} and the results of \cite{GerberN} allow us to upgrade some of the results of \cite{GHJ} on weak cuspidals as statements about actual cuspidals. For example,  \cite[Proposition 7.5]{GHJ} determines when $1^n$ labels a weak cuspidal, and it turns out by comparing with \cite[Theorem 8.3]{GHM1} that $X_{1^n}$ is weakly cuspidal if and only if it is cuspidal. We may now generalize this result to determine when $\lambda$ labels a cuspidal for any $\lambda\in\cP$ such that $\tau(\lambda)=|\emp.1^m,\bs_t\rangle$ or $\tau(\lambda)=|1^m.\emp,\bs_t\rangle$, $t\in\N_0$. Let us introduce the following notation for concatenating partitions: if $\lambda=(\lambda_1,\lambda_2,\dots,\lambda_r),\;\mu=(\mu_1,\mu_2,\dots,\mu_s)\in\cP$ then by $\lambda\sqcup\mu$ we denote the partition with parts $(\lambda_1,\lambda_2,\dots,\lambda_r,\mu_1,\mu_2,\dots,\mu_s)$, rearranged if necessary so that the parts are non-increasing when read from left to right. For example, if $\lambda=\Delta_3=(3,2,1)$ and $\mu=(1^4)$ then $\lambda\sqcup\mu=(3,2,1^5)$ is just given by stacking $\mu$ below $\lambda$.

\begin{theorem}\label{cusp column}
	\begin{enumerate}
		\item Let $\nu=\Delta_t\sqcup(1^{2m})$ for $m\geq 0$. Then $X_\nu$ is cuspidal if and only if $e\mid m$ or $e\mid 2(t+m)-1$.
		\item Let $t\geq 2$ and $\nu=\Delta_t\sqcup(1^{2m+1})$ for $m\geq0$. Then $X_\nu$ is cuspidal if and only if (i) $e\mid 2m+1$, or (ii) $e\mid m+t$ and $2m+1\geq e$.
	\end{enumerate}
\end{theorem}
\begin{proof}
	By \cite[Corollary 7.5]{GerberN}, if $\bs=(s_1,s_2)\in\Z^2$ is any charge, it holds that $|1^m.\emp,\bs\rangle$ is a source vertex for both the $\sle$- and $\slinf$-crystals if and only if (i) $s_2-s_1=ke-m+1$ for some $k\in\N$, or (ii) $e\mid m$ and $s_2-s_1\geq e-m+1$. If $\nu=\Delta_t\sqcup (1^{2m})$ for some $m\geq 0$ then $\tau(\nu)=|1^m.\emp,\bs_t\rangle$, where $\bs_t=(s_1,s_2)$ satisfies $s_2-s_1=\frac{e+1}{2}+t$. Then case (i) holds if and only if $e\mid\frac{e+1}{2}+t+m-1$, which is the case if and only if $e\mid 2(t+m)-1$ since $e$ is odd. Case (ii) holds if and only if $e\mid m $ and $2(t+m)\geq e+1$; the latter condition is superfluous since $e\geq 3$. This proves part (1).
	
	To prove part (2), we use the crystal isomorphism $|\lambda^1.\lambda^2,(s_1,s_2)\rangle \simeq |\lambda^2.\lambda^1,(s_2,s_1+e)\rangle$.\footnote{This isomorphism is induced by the ``twist by a character" isomorphism between type $B$ rational Cherednik algebras with parameters $(\frac{1}{e}, -\frac{1}{2}+\frac{s_2-s_1}{e})$ and $(\frac{1}{e},-\left(-\frac{1}{2}+\frac{s_2-s_1}{e}\right))$, which exchanges the simple modules labeled by $\lambda^1.\lambda^2$ and $\lambda^2.\lambda^1$ \cite[Section 2.3.4]{Losev}.}
	 If $t\geq 2$ and $\nu=\Delta_t\sqcup(1^{2m+1})$, then the $2$-core of $\nu$ is $\Delta_{t-2}$ and \[\tau(\nu)=|\emp.1^{t+m},\bs_{t-2}\rangle\simeq |\emp.1^{t+m},(0,\frac{e+1}{2}+t-2)\rangle\simeq
	|1^{t+m}.\emp, (\frac{e+1}{2}+t-2,e)\rangle.
	\]
	Write $m'=m+t$ and $t'=t-2$, so that (normalizing the charge) $$\tau(\nu)\simeq|1^{m'}.\emp,(0,\frac{e-1}{2}-t')\rangle.$$
	Then applying \cite[Cor. 7.5]{GerberN} yields that $\tau(\nu)$ is a source vertex of the $\sle$- and $\slinf$-crystals if and only if (i) $e\mid 2(m'-t')-3$ or (ii) $e\mid m'$ and $2m'-2t'-3\geq e$.
\end{proof}
\begin{remark}
	Theorem \ref{cusp column} is consistent with \cite[Proposition 7.5]{GHJ} and \cite[Theorem 8.3]{GHM1} which says that $X_{1^n}$ is cuspidal if and only if $e$ is odd and divides $n$ or $n-1$. In our conventions, $\tau(1^n)=\tau(1^{2m+\iota})=|1^m.\emp, \bs_{\iota}\rangle$ for $\iota\in\{0,1\}$. This falls under the purview of part (1) of Theorem \ref{cusp column} with $t=\iota$. It is trivial to check that when $t=0$ or $1$ then part (1) of Theorem \ref{cusp column} is equivalent to $e\mid n$ or $e\mid n-1$.
\end{remark}

In \cite[Conjecture 5.5]{GHJ}, which is shown in \cite[Theorem 7.6]{GHJ} to follow from \cite[Conjecture 5.7]{GHJ} (now \cite[Theorem B]{DVV1}), it is stated that if a unipotent block of $\mathbbm{k}\GU_n(q)$ contains a weak cuspidal, then it contains an ordinary cuspidal character (the $\ell$-reduction of a cuspidal irreducible character in characteristic $0$, not necessarily unipotent). The $\ell$-reduction of the latter will still be cuspidal and therefore must have simple constituents that are cuspidal unipotent $\mathbbm{k}\GU_n(q)$-representations. This shows that a unipotent block of $\mathbbm{k}\GU_n$ containing a weak cuspidal simple module contains a cuspidal simple module.  
Here we give a different proof of the latter statement using only facts about the $\sle$- and $\slinf$-crystals on level $2$ Fock spaces. Because the classification given by Theorem \ref{dvv cusp}  of cuspidals for finite unitary groups in a given characteristic $0$ Harish-Chandra series $\Delta_t$is identical to the classification of cuspidals for a cyclotomic rational Cherednik algebra with parameters given by the same Fock space $\cF_{e,\bs_t}$ \cite[Proposition 6.2]{ShanVasserot}, we can use an argument coming from derived equivalences of a special kind, namely perverse equivalences, between Cherednik algebras with different parameters \cite{Losev}. The underlying combinatorics then gives results about cuspidal unipotent representations of finite unitary groups.
\begin{theorem}\label{blocks with cuspidals}
Suppose $X_\lambda$ is a weakly cuspidal unipotent  $\mathbbm{k}\GU_{n}(q)$-module. Then there exists a cuspidal unipotent $\mathbbm{k}\GU_{n}(q)$-module $X_\mu$ in the same block as $X_\lambda$ such that $\lambda$ and $\mu$ belong to the same Harish-Chandra series in characteristic $0$, that is, $\lambda$ and $\mu$ have the same $2$-core.
\end{theorem}
\begin{proof}
Write $n=2m+\iota$, $\iota\in\{0,1\}$ and $m\in\N_0$. 
If $X_\lambda$ is weakly cuspidal, then $\tau(\lambda)=|\bla,\bs_t\rangle$ is a source vertex of the $\sle$-crystal. We may write $\bla=\tilde{a}_\sigma(\bla_0)$ for some bipartition $\bla_0$ such that $|\bla_0,\bs_t\rangle$ is a source vertex of the $\slinf$-crystal as well as the $\sle$-crystal, and $\sigma\in\cP(k)$ where $k$ is the depth of $\bla$ in the $\slinf$-crystal. Thus $|\bla|=m=|\bla_0|+ke$. Now we apply a succession of combinatorial wall-crossing bijections as defined in \cite{Losev}, \cite{JL2} to $|\bla_0,\bs_t\rangle$ to reach an ``asymptotic chamber" where the components $s_1$ and $s_2$ of the charge are sufficiently far apart. The important property of the combinatorial wall-crossing that we need is that it is an isomorphism of the $\sle$- and $\slinf$-crystals that permutes the set of bipartitions of size $m$ \cite[Proposition 1.2]{Losev}. We may write $\Phi$ for the combinatorial wall-crossing across a single wall such that $\Phi|\bmu,(s_1,s_2)\rangle=|\phi(\bmu),(s_1,s_2+e)\rangle$, where $\phi:\cP^{(2)}(m)\rightarrow\cP^{(2)}(m)$ is the bijective map given by the combinatorial wall-crossing. Write $\Phi^N:=\Phi\circ\Phi\circ\dots\circ\Phi$.
We have  $\Phi^N|\bla_0,\bs\rangle=|\phi_N(\bla_0),(s_1,s_2+Ne)\rangle$ where we choose $N$ such that $s_2+Ne-s_1>m$ and $\phi_N:\cP^{(2)}(m)\rightarrow\cP^{(2)}(m)$ is the bijective map on bipartitions given by $\Phi^N$. 
 Then $\Phi^N|\bla_0,(s_1,s_2+Ne)\rangle$ is a source vertex of both the $\sle$- and $\slinf$-crystals.

 A charge $\bt=(t_1,t_2)$ is asymptotic for bipartitions of size $m$ if $t_2-t_1>m$, and in this case the $\slinf$-crystal acts on $\bmu\in\cP^{(2)}(m)$ by adding a vertical strip of $e$ boxes to $\mu^2$ only \cite[Proposition 1.1]{Losev}. Then $\mu^1.\emp$ is always a source vertex of the $\slinf$-crystal, and if $|\bmu,\bt\rangle$ is a source vertex of both the $\slinf$- and $\sle$-crystals we have $\mu^2=\emp$. Thus when the charge is asymptotic it is enough to check that $\mu^1.\emp$ is a source vertex of the $\sle$-crystal to check that it is a source vertex of both the $\slinf$- and $\sle$-crystals.  This will be the case if and only if the partition $\mu^1$ has at most one good removable box, and in the case that such a box $b$ exists, then $\mathrm{ct}(b)+t_1=t_2$ mod $e$.

Now we apply these remarks to $\Phi_N|\bla_0,\bs_t\rangle$: the bipartition $\phi_N(\bla_0)$ is of the form $\nu.\emp$ for some partition $\nu$, and $\nu$ has at most one good removable box of content $s_2-s_1\mod e$, and no good removable boxes of any other content mod $e$. Let $\tilde{\nu}$ be the partition obtained from $\nu$ by appending a vertical strip of length $ek$ at the bottom of $\nu$, so if $\nu=(\nu_1,\dots,\nu_s)$ then $\tilde{\nu}=(\nu_1,\dots,\nu_s,1^{ek})$, and $|\tilde{\nu}|=m$. It is clear that $\tilde{\nu}$ again has at most one good removable box, which is subject to the same condition on its content mod $e$. Then $|\tilde{\nu}.\emp,(s_1,s_2+Ne)\rangle$ is a source vertex of both the $\sle$- and $\slinf$-crystal since the charge $(s_1,s_2+Ne)$ is asymptotic for bipartitions of size $m$. Applying the inverse wall-crossing bijection $\Phi_N^{-1}$, we set $\bmu=\phi_N^{-1}(\tilde{\nu}.\emp)$ so that $|\bmu,\bs_t\rangle=\Phi_N^{-1}|\tilde{\nu}.\emp,(s_1,s_2+Ne)\rangle$. Since $\Phi_N^{-1}$ is an isomorphism of the $\sle$- and $\slinf$-crystals, it follows that $|\bmu,\bs_t\rangle$ is a source vertex of the $\sle$- and $\slinf$-crystals on $\mathcal{F}_{e,\bs_t}$. Therefore $\tau^{-1}|\bmu,\bs_t\rangle=:\mu$ labels a simple cuspidal unipotent $\mathbbm{k}\GU_{2m+\iota}(q)$-module $X_\mu$.

Now we check that $X_\mu$ is in the same block as $X_\lambda$. If $|\bla,\bs\rangle$ and $|\bmu,\bs\rangle$ are two charged bipartitions with $|\bla|=|\bmu|$ and the same charge $\bs$, we will say that they are in the same combinatorial block (with respect to the charge $\bs$) if $\{\mathrm{ct}(b)\mod e\mid b\in|\bla,\bs\rangle\}=\{\mathrm{ct}(b)\mod e\mid b\in|\bmu,\bs\rangle\}$ is an equality of multisets. For example, $|3.\emp,(0,2)\rangle$ and $|2.1,(0,2)\rangle$ are in the same combinatorial block. Now, the two content multisets are the same means that $\bla$ can be transformed into $\bmu$ by a sequence of moves where a removable box of content congruent to $i$ mod $e$ is removed and put back again in the spot of an addable box of content congruent to $i$ mod $e$, for various $i$ in $\Z/e\Z$. Doing this corresponds to combinations of moves on the abacus of $\tau(\lambda)$ that are called ``elementary operations" in \cite[Section 7.3]{GHJ}, together with their inverses, so that the charge remains unchanged (one bead moves from row $2$ to row $1$ while another bead moves from row $1$ to row $2$; or a bead moves to a new spot in the same row). The elementary operations and their inverses on $\tau(\lambda)$, in turn, correspond to removing and adding $e$-rimhooks from the partition $\lambda$. It follows 
 that if $\tau(\lambda)$ and $\tau(\mu)$ are in the same combinatorial block then they are in the same block of $\mathbbm{k}\GU_{2m+\iota}(q)$. Both the $\slinf$-crystal and our procedure for constructing $\tilde{\mu}.\emp$ added $k$ boxes of content $i$ for every $i\in\Z/e\Z$ to $\nu.\emp$, so $\Phi|\bla,\bs_t\rangle$ is in the same combinatorial block as $|\tilde{\nu}.\emp,\bs_t+(0,Ne)\rangle$. The combinatorial wall-crossing respects combinatorial blocks, therefore $|\bla,\bs_t\rangle$ is in the same combinatorial block as $|\bmu,\bs_t\rangle$. It follows that $X_\lambda$ is in the same block as $X_\mu$. Since $\tau(\mu)\in\cF_{e,\bs_t}$ by construction, the $2$-core of $\mu$ is $\Delta_t=2\hbox{-core}(\lambda)$.
\end{proof}

\section{The Harish-Chandra branching rule for finite unitary groups}\label{heart of the artichoke}

If $\lambda\in\cP$ and $\tau(\lambda)=|\bla,\bs_t\rangle $ where $t\in\N_0$ is such that $\Delta_t$ is the $2$-core of $\lambda$, then denote by 
$|\bla_0,\bs_t\rangle$ the source vertex of both the $\slinf$- and $\sle$-crystals in the connected component containing $\tau(\lambda)$. Set $\lambda_0=\tau^{-1}|\bla_0,\bs_t\rangle\in\cP$.
Let $k$ be the depth of $\tau(\lambda)$ in the $\slinf$-crystal, let $r$ be the depth of $\tau(\lambda)$ in the $\sle$-crystal, and set $n'=n-2ek-2r$.

\begin{theorem}\label{conj} Let $q$ be a power of a prime $p$, let $\ell$ be a prime not dividing $q$, and suppose $\mathbbm{k}$ is a field of characteristic $\ell$ and a splitting field for all subgroups of $\GU_n(q)$. Setting $e$ equal to the order of $-q$ mod $\ell$, assume that $e$ is odd and at least $3$, and assume that $\ell>n/e$. 
Let $\lambda$ be a partition of $n$ and let $X_\lambda$ be the unipotent $\mathbbm{k}\GU_{n}(q)$-representation labeled by $\lambda$.  Then the cuspidal support of $X_\lambda$ is given by $$(\GU_{n'}(q)\times \GL_e(q^2)^{\times k}\times \GL_1(q^2)^{\times r}, X_{\lambda_0}\otimes \st_e^{\otimes k}\otimes \st_1^{\otimes r} ).$$
\end{theorem}

\noindent Theorem \ref{conj} is an upgrading of \cite[Conjecture 5.7]{GHJ}, \cite[Theorem B]{DVV1} to account for the $\slinf$-crystal on the Fock space as well as the $\sle$-crystal, all Levi subgroups affording cuspidal representations instead of just ``pure Levi subgroups." In the case that $\lambda$ has depth $0$ in the $\slinf$-crystal then the statement of Theorem \ref{conj} reduces to \cite[Conjecture 5.7]{GHJ}, \cite[Theorem B]{DVV1}. Theorem \ref{dvv cusp} (\cite[Theorem 5.10]{DVV2}) says that Theorem \ref{conj} is true if $X_\lambda$ itself is cuspidal. Motivation for Theorem \ref{conj} comes from the analogous statement for cyclotomic rational Cherednik algebras, whose category $\cO$ categorifies the Fock space and whose Harish-Chandra series categorify the $\sle$- and $\slinf$-crystals \cite{ShanVasserot},\cite{Shan}. Since we already know the weak branching rule and the classification of cuspidals is the same for the finite unitary groups as for the Cherednik algebras categorifying the same Fock spaces, it's reasonable to expect that the branching involving copies of $\GL_e(q^2)$ in the cuspidal support of a unipotent $\mathbbm{k}\GU_n$-module $X_\lambda$ is also described by tracing the $\slinf$-crystal component of $\tau(\lambda)$ to its source, just as for Cherednik algebras.

\begin{remark}Theorem \ref{conj} should imply that the involution on unipotent $\mathbbm{k}\GU_n(q)$-modules given by Alvis-Curtis duality coincides with the generalized Mullineux involution defined in \cite[Theorem 2.9]{GJN} on the twisted two-quotients of the partitions labeling the simple modules, generalizing the result of \cite{DJ} for finite general linear groups to finite unitary groups in the case of large $\ell$.\end{remark}

For illustrations of Theorem \ref{conj} when $e=3$, see Tables \ref{tab:Unitary3,5}, \ref{tab:Unitary3,n=12}, \ref{tab:Unitary3,n=13}, \ref{tab:Unitary3,n=15}, \ref{tab:Unitary3,n=16}, \ref{tab:Unitary3,n=18, Cuspidals}, \ref{tab:Unitary3,n=18 depth 1,2,3}, \ref{tab:Unitary3,n=19, Cuspidals}, and \ref{tab:Unitary3,n=19 depth 1,2,3}. For a computational verification for $n\leq 17$ when $e=3$, see Theorem \ref{up to 15}.

\subsection{Proof of Theorem \ref{conj}} \subsubsection{Remarks and motivation}
The proof of the analogue of Theorem \ref{conj} for cyclotomic rational Cherednik algebras proceeds by constructing a categorical action of an infinite-dimensional Heisenberg algebra (over $\mathbb{C}$) on the cyclotomic category $\cO$ \cite{ShanVasserot}. There appear to be problems with doing this for the unipotent category of finite classical groups. Despite this, enough of the construction of \cite{ShanVasserot} was adapted in \cite{DVV2} to yield the same classification of cuspidals as for a cyclotomic rational Cherednik algebra. We are going to take a low-tech approach and prove Theorem \ref{conj} using elementary combinatorics, the combinatorial classification of cuspidals, unitriangularity of the decomposition matrix with respect to dominance order, standard facts about Harish-Chandra series appearing in the Harish-Chandra induction of simple modules and their projective covers, Dipper-Du's results about $\GL_n(q^2)$, and an induction argument. It is natural to use induction since we are dealing with an infinite series of groups indexed by natural numbers. Remarkably, we can circumvent categorification of the Heisenberg algebra to deduce the full branching rule from the extant results about the $\sle$-crystal and cuspidals.

We will abuse notation and write $\tilde{a}_\sigma(\lambda)$ when we mean $\tilde{a}_\sigma(\tau(\lambda))$, etc, by transferring the $\slinf$-crystal on $\sum_{t\geq 0} \cF_{e,\bs_t}$ to the corresponding simple directed graph on $\cP$ along the map $\tau$. Before we sketch the steps of the proof, recall that the vertices of a connected component of the $\slinf$-crystal are given by the partitions $\tilde{a}_\sigma(\lambda_0)$, $\sigma\in\cP$, where $\lambda_0$ is the source vertex of that connected component. We know that the weak cuspidals $X_\lambda$ are labeled by $\lambda$ which have depth $0$ in the $\sle$-crystal; thus they are labeled by those $\lambda$ which lie in a connected component of the $\slinf$-crystal whose source vertex $\lambda_0$ labels a cuspidal unipotent module $X_{\lambda_0}$. It follows that any $\lambda\in\cP$ such that $X_\lambda\in\mathbbm{k}\GU_{|\lambda|}(q)$-mod is weakly cuspidal can be written as $\tilde{a}_\sigma(\lambda_0)$ for some $\sigma,\lambda_0\in\cP$ such that $X_{\lambda_0}\in\mathbbm{k}\GU_{|\lambda_0|}$-mod is cuspidal. Moreover, $\sigma$ and $\lambda_0$ are unique as each $\lambda$ belongs to a unique connected component of the $\slinf$-crystal, and each such connected component is isomorphic as a graph to Young's lattice.

\subsubsection{Outline of the proof}
The proof of Theorem \ref{conj} proceeds by the following steps. Let $\lambda_0$ be an arbitrary partition such that $X_{\lambda_0}$ is a cuspidal $\mathbbm{k}\GU_{|\lambda_0|}(q)$-module. \\
\textbf{Step 1.} Theorem \ref{projectives thm} establishes that the projective indecomposable module $P_{\tilde{a}_\sigma(\lambda_0)}$, for $\sigma$ any partition, is a direct summand of $R_{\GU_{|\lambda_0|}(q)\times \GL_{e}(q^2)^{\times |\sigma|}}^{\GU_{|\lambda_0|+2e|\sigma|}(q)} \left(P_{\lambda_0}\otimes P_{\st_e}^{\otimes|\sigma|} \right)$. This is important because the projective indecomposable module $P_{\tilde{a}_\sigma(\lambda_0)}$ is the projective cover of a weakly cuspidal unipotent module $X_{\tilde{a}_\sigma(\lambda_0)}$, and its appearance as a summand of a Harish-Chandra induced projective has consequences for the cuspidal support of $X_{\tilde{a}_\sigma(\lambda_0)}$.\\
\textbf{Step 2.} Corollary \ref{cusp depth bound} observes that Theorem \ref{projectives thm} bounds the cuspidal depth of $X_{\tilde{a}_\sigma(\lambda_0)}$ from below by $|\sigma|$, so by the depth of $\tilde{a}_\sigma(\lambda_0)$ in the $\slinf$-crystal. (More specifically, it implies that the Dynkin diagram of the standard Levi subgroup in the cuspidal support of $X_{\tilde{a}_\sigma(\lambda_0)}$ is obtained from the Dynkin diagram of $\GU_{|\lambda_0|}(q)\times\GL_{e}(q^2)^{\times |\sigma|}$ by deleting some number $d\geq 0$ of vertices and whatever edges are connected to the deleted vertices.) This also cinches the proof of Theorem \ref{conj} when $|\lambda_0|<2e$, providing the base case for induction on $|\lambda_0|$. \\
\textbf{Step 3.} In Theorem \ref{cusp supp thm} we induct on $m$, where $\GU_{m}(q)$ is a factor of the Levi subgroup in the cuspidal support of $X_{\tilde{a}_\sigma(\lambda_0)}$, to argue that 
$m$ cannot be smaller than $|\lambda_0|$. Comparing with the result of Step 2 we conclude that the cuspidal support of $X_{\tilde{a}_\sigma(\lambda_0)}$ is  $\left(\GU_{|\lambda_0|}(q)\times \GL_{e}(q^2)^{\times |\sigma|},X_{\lambda_0}\otimes \st_e^{\otimes |\sigma|}\right)$. This proves Theorem \ref{conj} for weak cuspidals.\\
\textbf{Step 4.} We observe in Corollary \ref{cusp supp cor} that Theorem \ref{conj} is true for all unipotent $X_\lambda$ if it is true for those $X_\lambda$ which are weakly cuspidal. From the result of the previous step for weak cuspidals we can now conclude that Theorem \ref{conj} is true.

\subsubsection{The proof}
\begin{theorem}\label{projectives thm}
Let $\lambda$ be a partition of $n$ such that $X_\lambda\in\mathbbm{k}\GU_{n}(q)$ is weakly cuspidal and assume that $\ell=\charac(\mathbbm{k})>n/e$. Write $\lambda=\tilde{a}_\sigma(\lambda_0)$ for the unique partitions $\sigma,\lambda_0\in\cP$ such that $X_{\lambda_0}\in\mathbbm{k}\GU_{n}(q)$ is cuspidal. Let $P_\lambda$ be the projective cover of $X_\lambda$, let $P_{\lambda_0}$ be the projective cover of $X_{\lambda_0}$, and let $P_{\st_e}$ be the projective cover of the Steinberg representation $\st_e=X_{1^e}\in\mathbbm{k}\GL_{e}(q^2)$. Then $P_{\lambda}$ is a direct summand of $R_{\GU_{|\lambda_0|}(q)\times\GL_{e}(q^2)^{\times|\sigma|}}^{\GU_{n}(q)}\left(P_{\lambda_0}\otimes P_{\st_e}^{\otimes|\sigma|}\right)$.
\end{theorem}
\begin{proof}
	
Let $\sigma^t=:(c_1,c_2,\dots,c_r) $ be the transpose partition of $\sigma$, so $c_1$ is the length of the first column of $\sigma$, $c_2$ is the length of the second column of $\sigma$, ..., and $r=\sigma_1$. We have 
\[\tilde{a}_\sigma(\lambda_0)=\tilde{a}_{1^{c_r}}\left(\dots\tilde{a}_{1^{c_2}}\left( \tilde{a}_{1^{c_1}}\left(\lambda_0\right)\right)\dots\right).\]
By our assumption that $\ell>n/e$, the cuspidal support of $\st_{ek}\in\mathbbm{k}\GL_{ek}(q^2)$ is $\left( \GL_e(q^2)^{\times k},\st_e^{\otimes k} \right)$ for any $k\leq n/e$ \cite{DipperDu}. This implies that $P_{\st_{ek}}$ is a direct summand of $R_{\GL_e(q^2)^{\times k}}^{\GL_{ek}(q^2)}P_e^{\otimes k} $ for any $k\leq n/e$.
To prove the theorem it thus suffices to show that if $\mu$ is weakly cuspidal, then 
 $P_{\tilde{a}_{1^k}(\mu)}$ is a direct summand of $R_{\rL}^{\rG}\left( P_{\mu}\otimes P_{\st_{ek}} \right)$ where $\rG=\GU_{|\mu|+2ek}(q)$ and $\rL=\GU_{|\mu|}(q)\times\GL_{ek}(q^2)$.

Recall that the unipotent part of $[P_{ek}]\in[\mathbbm{k}\GL_{ek}(q^2)\mathrm{-mod}]$ is given by the ordinary Steinberg character $\St_{ek}=(1^{ek})$. Write $[P_\mu]=\mu+\sum\limits_{\nu\vartriangleleft\mu}a_\nu \nu$ for some $a_\nu\in\N_0$ and such that the $\nu$ in the sum have the same $e$-core as $\mu$. Then we can calculate the unipotent part of the projective character $[R_{\rL}^{\rG}\left(P_\mu\otimes P_{\st_{ek}}\right)]$ by calculating $\Ind_{B_m\times S_{ek}}^{B_{m+ek}}\mu^1.\mu^2\otimes 1^{ek}$ where $\tau(\mu)=|\mu^1.\mu^2,\bs_t\rangle$ and $m=|\mu^1|+|\mu^2|$, and then applying $\tau^{-1}$ to the resulting positive linear combination of bipartitions. From the Pieri rule we know that the bipartitions occurring are obtained by adding a (possibly broken) vertical strip of length $ek$ to the bipartitions $\tau(\mu),\tau(\nu)$ etc. in the formula for $P_\mu$. Moreover, we may pick out a direct summand $Q$ of the projective module $R_{\rL}^{\rG}\left(P_\mu\otimes P_{\st_{ek}}\right)$ by cutting to the block determined by adding only those vertical strips in which exactly $k$ boxes of each distinct residue mod $e$ are added.

Let $\nu\vartriangleleft\mu$ such that $\nu$ and $\mu$ have the same $e$-core. Let $\tilde{b}_{1^{ek}}(\nu)$ be the most dominant partition obtained from $\nu$ by adding a thickened union of ribbons $R$ to the border of $\nu$ such that $R$ has $2ek$ boxes, the $e$-core of $\tilde{b}_{1^{ek}}(\nu)$ is the same as the $e$-core of $\nu$, and $R$ locally has width at most $2$ (meaning that for any box $b$ in $R$, there is at most one box in the same row and to the right of $b$). Define $\tilde{b}_{1^{ek}}(\mu)$ similarly. Then clearly $\tilde{b}_{1^{ek}}(\nu)\vartriangleleft \tilde{b}_{1^{ek}}(\mu)$. 

We claim that $\tilde{b}_{1^{ek}}(\mu)=\tilde{a}_{1^{ek}}(\mu)$. 
Since $\mu$ is weakly cuspidal, the abacus of $\tau(\mu)=|\mu^1.\mu^2,\bs_t\rangle$ is totally $e$-periodic. 
Let $P_1,\dots, P_k$ be the first $k$ $e$-periods of the abacus of $\tau(\mu)$. Then $\tilde{a}_{1^{ek}}$ slides each of $P_1,\dots ,P_k$ one step to the right. On the Young diagram of the charged bipartition $|\mu_1.\mu_2,\bs_t\rangle$, $\tilde{a}_{1^{ek}}$ adds the biggest addable vertical strip (possibly disconnected) of length $ek$ such that exactly $k$ boxes of each residue in $\Z/e\Z$ are added. Now we use the formula given in Section \ref{unip unitary combinat} for recovering the partition $\mu$ from $\tau(\mu)$ to see what this does to $\mu$. Recall that up to adding some integer $z$ to every $\beta$-number, $\mu$ is the partition whose $\beta$-numbers are given by $\{2\beta^1\}\cup \{2\beta^2-e\}$ where $\beta^j$ are the $\beta$-numbers of $\mu^j$ with charge $(\bs_t)_j$ for $j=1,2$ (i.e. $\beta^j$ is the set of column positions of beads in row $j$ of the abacus of $|\mu^1.\mu^2,\bs_t\rangle$). All the odd $\beta$-numbers of $\mu$ come from row $2$ of the abacus of $\tau(\mu)$ and all the even $\beta$-numbers of $\mu$ come from row $1$ of that abacus. 

  When we move $P_1$ one step to the right on the abacus of $\tau(\mu)$, we move the biggest set of $e$ beads whose residues run through $\mathbb{Z}/e\mathbb{Z}$ one step each to the right. Since $e$ and $2$ are coprime, on $\cA_{e,t}(\mu^1.\mu^2)$ this corresponds to moving the biggest set of $e$ beads whose residues run through $\Z/e\Z$ two steps each to the right. Let $\tilde{P}_1$ be the set of beads on $\cA_{e,t}(\mu^1.\mu^2)$ corresponding to $P_1$. When we move $P_2$ one step to the right, on  $\cA_{e,t}(\mu^1.\mu^2)$ we move the  biggest set of $e$ beads in $\cA_{e,t}(\mu^1.\mu^2)\setminus \tilde{P}_1$ whose residues run through $\Z/e\Z$ two steps each to the right. Etc. Altogether, we take the biggest set of $ek$ distinct $\beta$-numbers in $\cA_{e,t}(\mu^1.\mu^2)$ whose residues run through $\Z/e\Z$ exactly $k$ times, and replace them all by $\beta+2$. Clearly this produces the most dominant partition from those obtained from $\mu$ by taking some set of $ek$ $\beta$-numbers in $\cA_{e,t}(\mu^1.\mu^2)$ with each residue mod $e$ occurring $k$ times and replacing them all with $\beta+2$.
  
  Replacing $\beta\in\cA_e(\mu)$ with $\beta+2$ (supposing that $\beta+2$ is not already in $\cA_e(\mu)$) can be visualized by a bead on a single-line abacus hopping two spaces to its right. If $\beta+1\in\cA_e(\mu)$ then it hops over another bead, if not then it hops over an empty space. Hopping over another bead adds a vertical domino (looks like $(1^2)$) to the Young diagram of $\mu$, while hopping over a space adds a horizontal domino to $\mu$ (looks like $(2)$). If $\beta\in\tilde{P}_i$ for $i<k$ hops over a bead $\beta+1\in P_{i'}$ for $i<i'\leq k$, then  the bead $\beta+1$ will subsequently hop to $\beta+3$  and the combined effect is two vertical dominoes added in the same row position, for a total width of $2$ boxes added in those rows. If $\beta$ hops over a space, then the row that it represents in $\mu$ has not changed after this move, and no more boxes are added to that row of $\mu$ by moving any other beads of $\tilde{P}_{i'}$, $i'\geq i$, two steps to the right. In either case, at most $2$ boxes are added to each row of $\mu$.  So moving $P_1,\dots,P_k$ once to the right in $\tau(\mu)$ adds a thickened ribbon to the border of $\mu$ whose width is at most $2$ in any row. Moreover, Lemma \ref{blocks slinf} shows that the $e$-core of $\tilde{a}_{1^{ek}}(\mu)$ is the same as the $e$-core of $\mu$. We deduce that $\tilde{b}_{1^{ek}}(\mu)=\tilde{a}_{1^{ek}}(\mu)$.

It follows that $\tilde{a}_{1^{ek}}(\mu)$ is the most dominant partition appearing as a unipotent constituent of $Q$. By unitriangularity of the decomposition matrix with respect to the dominance order \cite{Geck}, this implies that $P_{\tilde{a}_{1^{ek}}(\mu)}$ is a direct summand of $Q$ and hence that $P_{\tilde{a}_{1^{ek}}(\mu)}$ is a direct summand of $R_{\rL}^{\rG}\left(P_\mu\otimes P_{\st_{ek}}\right)$. This concludes the proof.
\end{proof}

\begin{corollary}\label{cusp depth bound} Keep the same assumptions on $e$ and $\ell$ as in Theorem \ref{projectives thm}. Let $\lambda\in\cP$ such that $X_\lambda\in\mathbbm{k}\GU_{|\lambda|}(q)$ is weakly cuspidal and write $\lambda=\tilde{a}_{\sigma}(\lambda_0)$ as in Theorem \ref{projectives thm}.  \begin{enumerate}\item
Let $(\rL,X_{\rL})$ be the cuspidal support of $X_\lambda$. Then $\rL\cong \GU_m(q)\times \GL_e(q^2)^{\times k}$ where $k\geq |\sigma|$ and $m+2ek=|\lambda|$.
\item If $|\lambda_0|<2e$ then Theorem \ref{conj} holds for $X_\lambda$.
\end{enumerate}
\end{corollary}

\begin{proof} (1)\;
	The simple $X_\lambda$ is weakly cuspidal if and only if $\lambda$ has depth $0$ in the $\sle$-crystal. By Theorem \ref{weakHC crystal} this happens if and only if no copies of $\GL_1(q^2)$ occur in $\rL$. Thus $\rL\cong \GU_m(q)\times \GL_e(q^2)^{\times k}$ for some $m\geq 0 $ and some $k\geq0$ such that $m+2ek=|\lambda|$. By Theorem \ref{projectives thm}, $P_\lambda$ is a direct summand of $R_{\GU_{|\lambda_0|}(q)\times \GL_{e}(q^2)^{|\sigma|}}^{\GU_{|\lambda_0|+2e|\sigma|}(q)}P_{\lambda_0}$. The statement then follows from Lemma \ref{proj lemma 1}.
	
	(2)\;  Let $\lambda=\tilde{a}_\sigma(\lambda_0)$ such that $|\lambda_0|<2e$. By Corollary \ref{cusp depth bound}, the Levi $\rL$ in the cuspidal support of $X_\lambda$ is $\GU_m(q)\times \GL_e(q)^{\times k}$ where $k\geq |\sigma|$, and such that $|\lambda|=m+2ek=|\lambda_0|+2e|\sigma|$. But $|\lambda_0|<2e$ so we must have $k=|\sigma|$.
\end{proof}

\begin{theorem}\label{cusp supp thm} Keep the same assumptions on $e$ and $\ell$ as in Theorem \ref{projectives thm}. 
Let $\lambda\vdash n$ such that $X_\lambda\in\mathbbm{k}\GU_{n}(q)$ is weakly cuspidal and write $\lambda=\tilde{a}_\sigma(\lambda_0)$ for $\lambda_0$ a source vertex of the $\slinf$-crystal as in the statement of Theorem \ref{conj}. Then  $\left(\GU_{|\lambda_0|}(q)\times \GL_e(q^2)^{\times|\sigma|}, X_{\lambda_0}\otimes \st_e^{\otimes|\sigma|}\right)$ is the cuspidal support of $X_\lambda$.
\end{theorem}
\begin{proof}
 It is shown in \cite{GHM2} that given $\rG=\GU_{m+2ek}(q)$,  $\rL=\GU_m(q)\times\GL_e(q^2)^{\times k}$ and $\charac(\mathbbm{k})>k$, and $X=X_{\lambda_0}\otimes\st_e^{\otimes k}$ a cuspidal unipotent $\mathbbm{k}\rL$-module, then $\End(R_{\rL}^{\rG}X)$ is a Hecke algebra of type $B_k$ with parameter $1$ associated to the $k-1$ simple reflections generating the subgroup $S_k<B_k$, and some parameter $p_1'$ associated to the other simple reflection \cite[Proposition 4.4]{GHM2}. Furthermore, it follows from \cite[Lemmas 3.15 and 3.16]{GHM2} that $p_1'=-1$ if and only if in the case when $k=1$, $\End\left(R_{\rL}^{\rG}X\right)=\End\left( R_{\rL}^{\rG}\left( X_{\lambda_0}\otimes\st_e\right)\right)$ is indecomposable. Moreover, if this is the case then the isomorphism classes of simple $H_q(B_k)$-modules are in bijection with the isomorphism classes of simple $\mathbbm{k}S_k$-modules \cite[Section 4.10]{GHM2}, so, by our assumption that $\charac(\mathbbm{k})>k$, in bijection with partitions of $k$.

First we will prove that $p_1'=-1$. In fact this is already a theorem of Gruber \cite[Proposition 2.3.5]{Gruber}, so we could just cite that result and move on, but since we have another proof using crystals instead of Green vertices, we give a new proof of his theorem here. Suppose that $X_\lambda$ is weakly cuspidal 
with cuspidal support $\left(\GU_m(q)\times\GL_e(q^2),X_{\mu_0}\otimes \st_e\right)$. By Corollary \ref{cusp depth bound}, the depth of $\lambda$ in the $\slinf$-crystal is either $0$ or $1$. If it is $0$ then $\lambda$ is a source vertex of the $\slinf$- and $\sle$-crystals and then by \cite{DVV2} $X_\lambda$ is a cuspidal $\GU_{m+2e}(q)$-representation, contradicting the assumption about the cuspidal support of $X_\lambda$. Therefore the depth of $\lambda$ in the $\slinf$-crystal has to be $1$, i.e. $\lambda=\tilde{a}_1(\lambda_0)$ for some partition $\lambda_0\vdash n-2e$ such that $X_{\lambda_0}$ is cuspidal. 
  We have $P_{\lambda}\mid R_{\rL}^{\rG}\left(P_{\lambda_0}\otimes P_{\st_e}\right)$ by Theorem \ref{projectives thm}, and $P_{\lambda}\mid R_{\rL}^{\rG}\left( P_{\mu_0}\otimes P_{\st_e}\right)$ by Lemma \ref{proj lemma 0}, 
   where $P_{\lambda_0}$ and $P_{\mu_0}$ are the projective covers of the cuspidal $\mathbbm{k}\GU_m(q)$-modules $X_{\lambda_0}$ and $X_{\mu_0}$, respectively. By Lemma \ref{proj lemma 2}, it follows that $X_{\lambda_0}\cong X_{\mu_0}$ and thus $\lambda_0=\mu_0$. Therefore Theorem \ref{conj} holds when $X_\lambda$ is weakly cuspidal and $\lambda$ has depth $1$ in the $\slinf$-crystal. Moreover, since there is a unique partition $\lambda$ of depth $1$ in the $\slinf$-crystal whose source vertex in the $\slinf$-crystal is $\lambda_0$, $R_{\rL}^{\rG}\left(X_{\lambda_0}\otimes \st_e\right)$ has simple head $X_\lambda$. By \cite[Lemma 3.15]{GHM2}, $R_{\rL}^{\rG}\left(X_{\lambda_0}\otimes \st_e\right)$ is indecomposable. Now it follows that the unknown parameter $p_1'$ for the Hecke algebra $\End\left(R_{\rL}^{\rG}X\right)$ is equal to $-1$ for any $\rL=\GU_m(q)\times\GL_e(q^2)^{\times k}$, $\rG=\GU_{m+2ek}(q)$, and $X=X_{\lambda_0}\otimes\st_e^{\otimes k}$ with $X_{\lambda_0}$ a cuspidal unipotent $\mathbbm{k}\GU_m(q)$-module. 
  
Now we can prove the statement about the cuspidal support of $X_\lambda$. Let $n=|\lambda|$. Suppose that $(\GU_m(q)\times\GL_e(q^2)^{\times k}, X_{\mu_0}\otimes \st_e^{\otimes k})$ is the cuspidal support of $X_\lambda$ for some $k\geq |\sigma|$ and $m\leq |\lambda|$. Then $P_\lambda$ is a direct summand of $R_{\rL}^{\rG}P_{\mu_0}$.
But also by Theorem \ref{projectives thm}, the projective indecomposable $\mathbbm{k}\rG$-module $P_{\tilde{a}_\sigma(\lambda_0)}$ is a direct summand of $R_{\rL}^{\rG}P_{\lambda_0}$, so if $|\mu_0|=|\lambda_0|$ then by Lemma \ref{proj lemma 2} we must have $\mu_0=\lambda_0$. Suppose that $k>|\sigma|$, so $m:=|\mu_0|=|\lambda_0|-2e(k-|\sigma|)<|\lambda_0|$. By induction on $|\mu_0|$ using Corollary \ref{cusp depth bound}(2) as the base case, we assume that for all $\rho\in\cP$, $X_{\tilde{a}_\rho(\mu_0)}$ has cuspidal support $(\GU_m(q)\times\GL_e(q)^{\times j},X_{\mu_0}\otimes\st_e^{\otimes j})$, where $j=|\rho|$. Since we showed above that the parameter $p_1'$ for the Hecke algebra $\End\left(R_{\GU_m(q)\times \GL_e(q^2)^{\times j}}^{\GU_n(q)}(X_{\mu_0}\otimes\st_e^{\otimes k})\right)$ is equal to $-1$, it now follows from \cite[Section 4.10]{GHM2} that if $Y$ is a simple composition factor of the head of $R_{\GU_m(q)\times \GL_e(q^2)^{\times j}}^{\GU_n(q)}(X_{\mu_0}\otimes\st_e^{\otimes k})$ then $Y\cong X_{\tilde{a}_\rho(\mu_0)}$ for some $\rho\vdash k$. By definition of cuspidal support, $X_\lambda$ is a composition factor of the head of $R_{\GU_m(q)\times \GL_e(q^2)^{\times j}}^{\GU_n(q)}(X_{\mu_0}\otimes\st_e^{\otimes k})$. We then have $X_{\tilde{a}_\sigma(\lambda_0)}\cong X_\lambda \cong X_{\tilde{a}_\rho(\mu_0)}$ implying  $\tilde{a}_\sigma(\lambda_0)=\tilde{a}_\rho(\mu_0)$ implying $\lambda_0=\mu_0$ and $\sigma=\rho$, contradicting the assumption that $k>|\sigma|$.
\end{proof}

\begin{corollary}\label{cusp supp cor}
Theorem \ref{conj} is true for any unipotent $\mathbbm{k}\GU_n(q)$-representation $X_\lambda$.
\end{corollary}
\begin{proof} Let $(\rL, X)$ be the cuspidal support of $X_{\lambda}$.
	The main result of \cite[Theorem B]{DVV1} says that the depth of $\lambda$ in the $\sle$-crystal is the number $r$ of factors of $\GL_1(q^2)$ in $\rL$, and that $(\rL,X)$ is the cuspidal support of the weak cuspidal $X_{\widehat{\lambda}}\in \mathbbm{k}\GU_{|\lambda|-2r}(q)$ where $\widehat{\lambda}$ is the source vertex of the connected component of the $\sle$-crystal containing $\lambda$. This reduces the proof of Theorem \ref{conj} to the case that $\lambda$ is weakly cuspidal by transitivity of Harish-Chandra induction and by the fact that the $\sle$- and $\slinf$-crystals commute. Theorem \ref{conj} then follows from Theorem \ref{cusp supp thm}.
\end{proof}

 \section{Appendix: direct checks in small rank, tables of examples, and inducing the projective cover of the Steinberg representation}\label{pile of artichoke leaves with teeth scrapes}
 In this appendix we present some computations, specific examples, and case by case proofs in small rank which we wrote for an earlier version of this paper when Theorem \ref{conj} was a conjecture. We check directly that Theorem \ref{conj} is consistent with the examples in the literature of modular Harish-Chandra series of finite unitary groups, we compute the Harish-Chandra series of weak cuspidals of $\mathbbm{k}\GU_n(q)$ when $e=3$ for $n$ up to $17$ by hand using the methods used by \cite{DM} in their computations of decomposition matrices, and we study the behavior of the projective cover of the Steinberg representation of $\GL_{n}(q^2)$ under Harish-Chandra induction to $\GU_{2n}(q)$ or $\GU_{2n+1}(q)$ when $e$ divides $n$.
 \subsection{Verification in small rank}
 Going one step towards a source vertex in the $\mathfrak{sl}_\infty$-crystal involves removing $e$ boxes from the twisted $2$-quotient of $\lambda$, and so it is removing $2e$ boxes from $\lambda$ itself. Likewise, a factor of $\GL_e(q^2)$ in a Levi subgroup accounts for a subgroup of block matrices of size $2e\times 2e$ (up to conjugation). This means the examples that have been computed in the literature (they go up to $n=10$, see \cite{DM}) provide only the slightest glimpse of the Harish-Chandra series involving type $A$ Levi subgroups or combinatorics other than that suggested by the $\sle$-crystal. On the other hand, computing decomposition matrices together with modular Harish-Chandra series for much bigger $n$ seems to be beyond current limits. We verify Theorem \ref{conj} for all $n\leq 10$ by checking in Table \ref{tab:Unitary3,5} that the cuspidal support predicted by the action of the $\slinf$-crystal agrees with the cuspidal supports found in \cite{DM} in all cases that $\tau(\lambda)$ has nonzero depth in the $\slinf$-crystal. However, because the rank is so small in all the examples that exist, the maximum depth in the $\slinf$-crystal in all of these examples is $1$. For this reason we need to do some work and compute new examples of Harish-Chandra series to see further into the branching graph.
 
 \begin{example} The partition $32^3$ labels a unipotent representation $X_{32^3}$ of $\GU_9(q)$ when $\ell\mid q^2-q+1=\Phi_6(q)=\Phi_{2e}(q)$ with $e=3$, and according to the column for $32^3$ in the decomposition matrix displayed in \cite[Table 11]{DM} the cuspidal support of $X_{32^3}$ is $(\GU_1(q)\times\GL_3(q^2),X_1\otimes\st_3)$ (in the table the rank $1$ factor is omitted because they are working with special unitary groups, not general unitary groups). We explain how to compute the position in the $\sle$- and $\slinf$-crystals of $\tau(32^3)$ when $e=3$. First, we compute that the $2$-quotient of $32^3$ is $\lambda^1.\lambda^2=1^3.1$ and the $2$-core of $32^3$ is $(1)=\Delta_1$. Since $t=1$, $\tau(32^3)=|\lambda^2.\lambda^1,\bs_1\rangle=|1.1^3,(-1,2)\rangle$. Drawing the abacus of $\tau(32^3)$, we then compute its depth in both the $\slinf$- and $\widehat{\mathfrak{sl}}_3$-crystals.  Since the two crystals commute, it does not matter in which order we do this. The first squiggly arrow moves one edge up in the $\slinf$-crystal to a source vertex $|\emp.1, (-1,2)\rangle$ in the $\slinf$-crystal, contributing $(\GL_3(q^2),\st_3)$ to the cuspidal support of $X_{32^3}$. The charged bipartition $|\emp.1, (-1,2)\rangle$ is not a source vertex of the $\widehat{\mathfrak{sl}}_3$-crystal since it is not totally $3$-periodic, so we then make one move up an edge in the $\widehat{\mathfrak{sl}}_3$-crystal to reach the source vertex for both crystals, which is $|\emp.\emp,(-1,2)\rangle=\tau(1)$. 
 	$$
 	\TikZ{[scale=.5]
 		\draw
 		(0,0)node[fill,circle,inner sep=3pt]{}
 		(1,0)node[fill,circle,inner sep=3pt]{}
 		(2,0)node[fill,circle,inner sep=.5pt]{}
 		(3,0)node[fill,circle,inner sep=3pt]{}
 		(4,0)node[fill,circle,inner sep=.5pt]{}
 		(5,0)node[fill,circle,inner sep=.5pt]{}
 		(6,0)node[fill,circle,inner sep=.5pt]{}
 		(7,0)node[fill,circle,inner sep=.5pt]{}
 		(0,1)node[fill,circle,inner sep=3pt]{}
 		(1,1)node[fill,circle,inner sep=3pt]{}
 		(2,1)node[fill,circle,inner sep=3pt]{}
 		(3,1)node[fill,circle,inner sep=.5pt]{}
 		(4,1)node[fill,circle,inner sep=3pt]{}
 		(5,1)node[fill,circle,inner sep=3pt]{}
 		(6,1)node[fill,circle,inner sep=3pt]{}
 		(7,1)node[fill,circle,inner sep=.5pt]{}
 		;}
 	\quad
 	\rightsquigarrow
 	\quad
 	\TikZ{[scale=.5]
 		\draw
 		(0,0)node[fill,circle,inner sep=3pt]{}
 		(1,0)node[fill,circle,inner sep=3pt]{}
 		(2,0)node[fill,circle,inner sep=3pt]{}
 		(3,0)node[fill,circle,inner sep=.5pt]{}
 		(4,0)node[fill,circle,inner sep=.5pt]{}
 		(5,0)node[fill,circle,inner sep=.5pt]{}
 		(6,0)node[fill,circle,inner sep=.5pt]{}
 		(7,0)node[fill,circle,inner sep=.5pt]{}
 		(0,1)node[fill,circle,inner sep=3pt]{}
 		(1,1)node[fill,circle,inner sep=3pt]{}
 		(2,1)node[fill,circle,inner sep=3pt]{}
 		(3,1)node[fill,circle,inner sep=3pt]{}
 		(4,1)node[fill,circle,inner sep=3pt]{}
 		(5,1)node[fill,circle,inner sep=.5pt]{}
 		(6,1)node[fill,circle,inner sep=3pt]{}
 		(7,1)node[fill,circle,inner sep=.5pt]{}
 		;}
 	\quad
 	\rightsquigarrow
 	\quad
 	\TikZ{[scale=.5]
 		\draw
 		(0,0)node[fill,circle,inner sep=3pt]{}
 		(1,0)node[fill,circle,inner sep=3pt]{}
 		(2,0)node[fill,circle,inner sep=3pt]{}
 		(3,0)node[fill,circle,inner sep=.5pt]{}
 		(4,0)node[fill,circle,inner sep=.5pt]{}
 		(5,0)node[fill,circle,inner sep=.5pt]{}
 		(6,0)node[fill,circle,inner sep=.5pt]{}
 		(7,0)node[fill,circle,inner sep=.5pt]{}
 		(0,1)node[fill,circle,inner sep=3pt]{}
 		(1,1)node[fill,circle,inner sep=3pt]{}
 		(2,1)node[fill,circle,inner sep=3pt]{}
 		(3,1)node[fill,circle,inner sep=3pt]{}
 		(4,1)node[fill,circle,inner sep=3pt]{}
 		(5,1)node[fill,circle,inner sep=3pt]{}
 		(6,1)node[fill,circle,inner sep=.5pt]{}
 		(7,1)node[fill,circle,inner sep=.5pt]{}
 		;}
 	$$
 	This verifies Theorem \ref{conj} for the partition $32^3$ when $e=3$.
 \end{example}

 \begin{table}[h]
 	\begin{tabular}{llllllll}
 		\hline
 		&&&&&&&\\
 		$e$\quad &$\lambda$ \quad & $t$ \quad & $|\bla,\mathbf{s}_t\rangle$\quad & $\mathfrak{sl}_\infty$-depth \quad &$\widehat{\mathfrak{sl}}_e$-depth\quad  & Cusp. supp. \quad & Agrees?\\ \hline
 		&&&&&&&\\
 		$3$\quad      &      $2^3$ \quad          & $0$ \quad& $|1.1^2,(-2,0)\rangle$ \quad& $1$\quad &$0$\quad & $(\GL_3(q^2),\st_3)$\quad  & $\checkmark$ \\
 		&&&&&&&\\
 		& $321^2$ & $1$ & $|\emp.1^3,(-1,2)\rangle$ & $1$ & $0$ & $(\GL_3(q^2),\st_3)$ & $\checkmark$ \\
 		&  &  &  &  &  & & \\ 
 		& $2^4$ & $0$ & $|1^2.1^2,(-2,0)\rangle$ & $1$ & $1$ & $(\GL_3(q^2)\times\GL_1(q^2),\st_3\otimes\st_1)$ & $\checkmark$ \\ 
 		&  &  &  &  &  & & \\ 
 		& $2^31^2$ & $0$ & $|1.1^3,(-2,0)\rangle$ & $1$ & $1$ & $(\GL_3(q^2)\times \GL_1(q^2),\st_3\otimes\st_1)$ & $\checkmark$  \\ 
 		&  &  &  &  &  & & \\ 
 		& $432$ & $2$ & $|\emp.1^3,(-3,1)\rangle$ & $1$ & $0$ & $(\GU_3(q)\times\GL_3(q^2),X_{21}\otimes\st_3)$ & $\checkmark$   \\ 
 		&  &  &  &  &  & & \\ 
 		
 		& $3^3$ & $1$ & $|2.1^2,(-1,2)\rangle$ & $1$ & $0$ & $(\GU_3(q)\times\GL_3(q^2),X_{1^3}\otimes\st_3)$ & $\checkmark$  \\ 
 		&  &  &  &  &  & & \\          
 		& $32^3$ & $1$ & $|1.1^3,(-1,2)\rangle$ & $1$ & $1$ & $(\GL_3(q^2)\times\GL_1(q^2),\st_3\otimes\st_1)$ & $\checkmark$  \\        
 		&  &  &  &  &  & & \\                      
 		
 		& $3^31$ & $0$ & $|2^2.1,(-2,0)\rangle$ & $1$ & $0$ & $(\GU_4(q)\times \GL_3(q^2),X_{1^4}\otimes\st_3)$ & $\checkmark$ \\ 
 		&  &  &  &  &  & & \\ 
 		$5$\quad &        $2^5$ \quad          & $0$ \quad& $|1^2.1^3,(-3,0)\rangle$ \quad& $1$\quad &$0$\quad & $(\GL_5(q^2),\st_5)$\quad  & $\checkmark$ \\
 		&  &  &  &  &  & & \\ 
 	\end{tabular}
 	\caption{Comparison of depth in the $\mathfrak{sl}_\infty$-crystal with cuspidal supports for $\lambda$ of depth $1$ in the $\slinf$-crystal when $|\lambda|\leq 10$ and $e=3,5$.}
 	\label{tab:Unitary3,5}
 \end{table}

 We have computed the Harish-Chandra series of all weakly cuspidal $\lambda$ by hand in the cases $n=12,13,15,16$ and $e=3$. We find that the data agrees with Theorem \ref{conj} and we record this information in Tables \ref{tab:Unitary3,n=12}, \ref{tab:Unitary3,n=13}, \ref{tab:Unitary3,n=15}, and \ref{tab:Unitary3,n=16}. For $n=2\mod 3$ there are no weak cuspidals and the statement of Theorem \ref{conj} reduces to the statement for $n-2$.
 
 \begin{theorem}\label{up to 15} Theorem \ref{conj} holds for $\GU_n(q)$ when $e=3$ and $n\leq 17$.
 \end{theorem}
 \begin{proof}
 	The proof for $n\leq 10$ follows by comparing the Harish-Chandra series labeling the columns in the decomposition matrices in \cite{DM} with all weakly cuspidal $\lambda$ for $e=3$ with $|\lambda|\leq 10$, as discussed above and summarized in Table \ref{tab:Unitary3,5}. The conjecture is almost trivially true for $e=5$ and $n\leq 10$ as well, as there is only one partition of depth $1$ in the $\slinf$-crystal in that case.
 	
 	 Set $\st=\st_3\in\mathbbm{k}\GL_3(q^2)$-mod and $\St=\St_3\in K\GL_3(q^2)$-mod.
 	
 	We are going to investigate the Harish-Chandra series of weakly cuspidal unipotent modules of $\GU_n(q)$ for $n=11,12,13,14,15,16,17$ when $e=3$. First of all, for $\lambda\vdash 11 $ or $12$, if the depth of $\lambda$ in the $\sle$-crystal isn't $0$, then since we know the Harish-Chandra series of $\mu\vdash n\leq 10$, we know the Harish-Chandra series of $\lambda$ by \cite{GHJ},\cite{GH},\cite{DVV1}. Then the truth of Theorem \ref{conj} for $n\leq 10$ and the fact that the $\sle$- and $\slinf$-crystals commute \cite{Gerber1} implies that the Harish-Chandra series of $\lambda$ agrees with Theorem \ref{conj}. Thus it is enough to consider the case that $\lambda$ is weakly cuspidal, i.e. of depth $0$ in the $\sle$-crystal. If $\lambda$ is weakly cuspidal, then the $3$-core of $\lambda$ is a $2$-core by \cite[Conjecture 5.5]{GHJ}, so for $n=2$ mod $3$ there cannot be any weak cuspidals. Therefore there is nothing to show for $n=11$, for $n=14$ the conjecture is true if it is true for $n=12$, and for $n=17$ it reduces to the case for $n=15$. We will now check the cases $n=12,13,15,16.$
 	
 	There are $77$ partitions of $12$ and and the $2$-cores that can occur are $\Delta_0$, $\Delta_3$, and $\Delta_4$; for each $\lambda\in\cP(12)$ we apply the bijection $\tau$ to obtain a charged bipartition of size $\frac{12-|\Delta_t|}{2}$ in the Fock space $\cF_{3,\bs_t}$,  $t\in\{0,3,4\}$, $\Delta_t=2\hbox{-core}(\lambda)$. 
 	Of these $77$ charged bipartitions, $13$ have depth $0$ in the $\widehat{\mathfrak{sl}}_3$-crystal. We then compute the depth and source vertex in the $\slinf$-crystal on $\cF_{3,\bs_t}$ of each of these $13$ weakly cuspidal partitions using the results of \cite{GerberN}. This data is summarized in the first four columns of Table \ref{tab:Unitary3,n=12}. In particular, for the three weakly cuspidal partitions $\lambda$ of depth $1$ in the $\slinf$-crystal, we find:
 	\begin{enumerate}
 		\item $\tau(3^31^3)=|2^21.1,(-2,0)\rangle=\tilde{a}_1|1^3.\emp,(-2,0)\rangle$ and $1^3.\emp=\tau(1^6)$ is cuspidal;
 		\item $\tau(432^21)=|\emp.2^3,(-2,0)\rangle=\tilde{a}_1|\emp.1^3,(-2,0)\rangle$ and $\emp.1^3=\tau(21^4)$ is cuspidal;
 		\item $\tau(543)=|\emp.1^3,(-2,3)\rangle=\tilde{a}_1|\emp.\emp,(-2,3)\rangle$ and $\emp.\emp=\tau(321)$ is cuspidal.
 	\end{enumerate}

 	Now we need to check that the modular Harish-Chandra series of the simple representations labeled by these $13$ partitions $\lambda$ match with their positions in the crystal. \cite[Theorem 5.10]{DVV2} says that $X_\lambda$ is cuspidal if and only if $X_\lambda$ is weakly cuspidal and $\tau(\lambda)$ has depth $0$ in the $\slinf$-crystal, so for cuspidals Theorem \ref{conj} is true. There are eight partitions $\lambda$ yielding cuspidals. That leaves five partitions to check: $3^31^3$, $432^21$, $543$, $2^6$, and $4^3$. Of these five, four have $2$-core equal to $\emp=\Delta_0$ while only $543$ has $2$-core $\Delta_3$. Since $(\GU_6(q)\times\GL_3(q^2),\Delta_3\otimes\st)$ is a cuspidal pair, it must be the cuspidal support of at least one simple module of $\GU_{12}(q)$. Since all other $\lambda$ with $2$-core $\Delta_3$ have nonzero depth in the $\sle$-crystal, by \cite[Theorem B]{DVV1} $^*R^{12}_{10}X_\lambda\neq 0$ and consequently the cuspidal support of such $X_\lambda$ is not $(\GU_6(q)\times\GL_3(q^2),\Delta_3\otimes\st)$.
 	Therefore $X_{543}$ is the unique unipotent $\mathbbm{k}\GU_{12}(q)$-module whose cuspidal support is $(\GU_6(q)\times\GL_3(q^2),\Delta_3\otimes\st)$. 
 	
 	Consider the ordinary unipotent character $1^6$ of $\GL_6(q^2)$. It is the unique unipotent constituent of the projective character $[P_{1^6}]$, and $X_{1^6}\in\mathbbm{k}\GL_6(q^2)$-mod has cuspidal support $(\GL_3(q^2)^2,\st^2)$. The unipotent part of the projective character $[R_{\GL_6(q^2)}^{\GU_{12}(q)} P_{1^6}]$ is the following sum of unipotent characters, written as bipartitions using $\tau$ and applying the Pieri rule: $1^6\uparrow=1^6.\emp+1^5.1+1^4.1^2+1^3.1^3+1^2.1^4+1.1^5+\emp.1^6$. We let $\overline{P}$ denote the unipotent part of the character of a projective indecomposable module $P$. Now applying the inverse to $\tau$ we have that the following expression in the Grothendieck group   $$[\overline{R_{\GL_6(q^2)}^{\GU_{12}(q)} P_{1^6}}]=[R_{\GL_6(q^2)}^{\GU_{12}(q)} \overline{P}_{1^6}]=[R_{\GL_6(q^2)}^{\GU_{12}(q)} 1^6]=1^{12}+2^21^8+2^41^4+2^6+2^51^2+2^31^6+21^{10}$$
 	is the unipotent part of a projective character.
 	Discarding the two partitions with non-empty $3$-core, which is the operation on the level of characters given by the exact functor of cutting to the principal block,
 	$$\Psi:=1^{12}+2^41^4+2^6+2^31^6+21^{10}$$
 	is the unipotent part of a projective character.
 	The only partition not labeling a cuspidal in the expression for $\Psi$ is $2^6$. But if $\lambda$ labels a cuspidal then $P_\lambda$ cannot appear as a summand of an induced projective module. It follows that $\Psi=[\overline{P}_{2^6}]$, and therefore that $(\GL_3(q^2)^2,\st^2)$ is the cuspidal support of $2^6$, as predicted by Theorem \ref{conj}.
 	
 	It remains to find the Harish-Chandra series of $X_{3^31^3}$, $X_{432^21}$, and $X_{4^3}$. Let us cast a glance back at the decomposition matrix of $\GU_6(q)$ for $e=3$ given in \cite[Table 8]{DM}. There are two cuspidals: $X_{1^6}$ and $X_{21^4}$. 
 	Moreover, the unipotent part of the projective character $[P_{1^6}]\in\mathbbm{k}\GU_6(q)$-mod is just $1^6$, and the unipotent part of $[P_{21^4}]$ is just $21^4$ \cite[Table 8]{DM}. The cuspidal $X_{1^6}\otimes\st\in\mathbbm{k}\left(\GU_6(q)\times\GL_3(q^2)\right)$-mod must account for the cuspidal support of some unipotent $\mathbbm{k}\GU_{12}(q)$-module, and likewise with the cuspidal $X_{21^4}\otimes\st$. However, 
 	\begin{align*} &[R_{\GU_6(q)\times\GL_3(q^2)}^{\GU_{12}(q)} \overline{P_{1^6}\otimes P_{\st}}]=\tau^{-1}\left(\Ind_{B_3\times S_3}^{B_6} 1^3.\emp \otimes 1^3\right),\\&[R_{\GU_6(q)\times\GL_3(q^2)}^{\GU_{12}(q)} \overline{P_{21^4}\otimes P_{\st}}]=\tau^{-1}\left(\Ind_{B_3\times S_3}^{B_6} \emp.1^3 \otimes 1^3\right) 
 	\end{align*}
 	where $\bs_0$ is taken for the charge. This can be computed using the Pieri rule and in neither case can the bipartition $4^3$ appear in either sum since $\tau(4^3)=2.2^2$ has no column of length $3$. Therefore $P_{4^3}$ is not a summand of $R_{\GU_6(q)\times\GL_3(q^2)}^{\GU_{12}(q)} \overline{P_{1^6}\otimes P_{\st}}$ or $R_{\GU_6(q)\times\GL_3(q^2)}^{\GU_{12}(q)} \overline{P_{21^4}\otimes P_{\st}}$, and therefore $X_{4^3}$ does not have cuspidal support $(\GU_6(q)\times\GL_3(q^2), X_{1^6}\otimes \st)$ or $(\GU_6(q)\times\GL_3(q^2), X_{21^4}\otimes \st)$. The only possibility left is that the cuspidal support of $X_{4^3}$ is $(\GL_3(q^2)^2,\st^2)$. 
 	
 	As for $[R_{\GU_6(q)\times\GL_3(q^2)}^{\GU_{12}(q)} \overline{P_{1^6}\otimes P_{\st}}]$, computing the ordinary unipotent characters occurring then discarding those without empty $3$-core, we find that the maximal partition in lexicographic order occurring is $3^31^3$, implying that $P_{3^31^3}$ is a direct summand of  $R_{\GU_6(q)\times\GL_3(q^2)}^{\GU_{12}(q)} P_{1^6}\otimes P_{\st}$. Therefore the cuspidal support of $X_{3^31^3}$ is $(\GU_6(q)\times\GL_3(q^2),X_{1^6}\otimes\st)$ since $P_{3^31^3}$ does not appear as a summand of the induced projective cover of a simple in a bigger Harish-Chandra series. By the pigeonhole principal, the cuspidal support of $X_{432^21}$ has to be $(\GU_6(q)\times\GL_3(q^2),X_{21^4}\otimes\st)$. This completes the classification of the Harish-Chandra series of $\GU_{12}(q)$ when $e=3$ and confirms Theorem \ref{conj} in that case. By the remarks at the beginning of the proof, this also verifies Theorem \ref{conj} for $\GU_{14}(q)$.

 	\begin{table}[h]
 		\begin{tabular}{lllll}
 			\hline\\
 			$\slinf$-depth \quad\quad  & $t$ \quad\quad \quad  & $\lambda$ \quad\quad\quad   & $\tau(\lambda)$ \qquad \qquad \qquad & cuspidal support \\ \hline\\
 			$0$ \quad \quad \quad\quad  & $0$\quad \quad \quad  & $1^{12}$ \quad\quad \quad  & $|1^6.\emp,(-2,0)\rangle$\qquad \qquad  & $(\GU_{12}(q),\hbox{itself})$\\
 			&&&&\\
 			& \quad \quad \quad  & $21^{10}$ \quad\quad \quad  & $|\emp.1^6,(-2,0)\rangle$& \\
 			&&&&\\
 			& \quad \quad \quad  & $2^31^6$ \quad\quad \quad  & $|1.1^5,(-2,0)\rangle$& \\
 			&&&&\\
 			& \quad \quad \quad  & $2^41^4$ \quad\quad \quad  & $|1^4.1^2,(-2,0)\rangle$ &\\
 			&&&&\\
 			& \quad \quad \quad  & $32^41$ \quad\quad \quad  & $|2^3.\emp
 			,(-2,0)\rangle$ &\\
 			&&&&\\
 			& $3$ & $321^7$ & $|1^3.\emp,(-2,3)\rangle$&\\
 			&&&&\\
 			&  & $3^321$ & $|3.\emp,(-2,3)\rangle$&\\
 			&&&&\\
 			& $4$ & $4321^3$ & $|1.\emp,(-4,2)\rangle$&\\
 			\\
 			\hline
 			&&&&\\
 			$1$ & $0$ & $3^31^3$ & $|2^21.1,(-2,0)\rangle$& $(\GU_6(q)\times \GL_3(q^2), X_{1^6}\otimes\st_3)$\\
 			
 			&&&&\\
 			&  & $432^21$ & $|\emp.2^3,(-2,0)\rangle$& $(\GU_6(q)\times \GL_3(q^2), X_{21^4}\otimes\st_3)$\\
 			&&&&\\
 			& $3$ & $543$ & $|\emp.1^3,(-2,3)\rangle$& $(\GU_6(q)\times\GL_3(q^2),X_{321}\otimes\st_3)$\\
 			&&&&\\
 			\hline
 			&&&&\\
 			$2$ & $0$ & $2^6$ & $|1^3.1^3,(-2,0)\rangle$ & $(\GL_3(q^2)^2, \st_3^2)$\\
 			&&&&\\
 			& & $4^3$ & $|2.2^2,(-2,0)\rangle$ &$(\GL_3(q^2)^2, \st_3^2)$\\
 			&&&&\\
 		\end{tabular}
 		\caption{Depth in the $\slinf$-crystal of all weakly cuspidal partitions  of $12$ (i.e. partitions of depth $0$ in the $\sle$-crystal) and cuspidal supports of the corresponding unipotent $\mathbbm{k}\GU_{12}$-modules when $e=3$. This illustrates Theorem \ref{conj} for $n=12$ and $e=3$.}
 		\label{tab:Unitary3,n=12}
 	\end{table}
 	
 	Let's move on to the case $n=13$. The possible $2$-cores of a partition of $13$ are $\Delta_2$ and $\Delta_1$, so there are two Fock spaces involved, $\cF_{3,1}$ and $\cF_{3,2}$. There are $101$ partitions of $13$ and by computing their charged twisted $2$-quotients we find that only nine of them are weakly cuspidal when $e=3$. Of these nine weakly cuspidal partitions, five are cuspidal, i.e. label a source vertex of both the $\sle$- and $\slinf$-crystals: $4321^4$, $321^8$, $2^61$, $2^31^7$, and $1^{13}$. That leaves four partitions whose Harish-Chandra series need to be sorted out: $5431$, $4^31$, $3^31^4$, and $32^41^2$. In the $\slinf$-crystal on $\mathcal{F}_{e,1}\oplus\mathcal{F}_{e,2}$:
 	\begin{align*}
 	\tau(5431)&=|\emp.2^3,(-1,2)\rangle=
 	\tilde{a}_{(2)}|\emp.\emp,(-1,2)\rangle\\
 	\tau(4^31)&=|3.1^2,(-3,1)\rangle=\tilde{a}_{(1)}|2.\emp,(-3,1)\rangle\\
 	\tau(3^31^4)&=|21^2.1^2,(-1,2)\rangle =\tilde{a}_{(1)}|1^3.\emp,(-1,2)\rangle\\
 	\tau(32^41^2)&=|1^3.1^3,(-1,2)\rangle=\tilde{a}_{(1^2)}|\emp.\emp,(-1,2)\rangle
 	\end{align*}
 	Applying the inverse to $\tau$ to the source vertices on the right-hand-side of these equations, we find the partitions labeling the source vertices are (listed in the same order): 
 	$1,\;
 	2^31,\;
 	1^7,\;
 	1$.

 	The unipotent part of the projective character $[P_{1^7}]$ is just $1^7$. We compute $[R_{\GU_7(q)\times \GL_3(q^2)}^{\GU_{13}(q)}\overline{P_{1^7}\otimes P_{\st}}]=[R_{\GU_7(q)\times \GL_3(q^2)}^{\GU_{13}(q)}1^7\otimes \St]$ then cut to the principal block  
 	obtaining the following projective character:
 	\[\Psi=1^{13}+2^21^9+2^61+3^31^4+32^41^2\]
 	The PIM $P_{3^31^4}$ has to be a summand of $R_{\GU_7(q)\times \GL_3(q^2)}^{\GU_{13}(q)}P_{1^7}\otimes P_{\st}$ because $3^31^4$ is the maximal partition in lexicographical order that occurs in $\Psi$. This implies that the simple labeled by $3^31^4$ has cuspidal support $(\GU_7(q)\times\GL_3(q^2),X_{1^7}\otimes\st)$.
 	Next, using the column labeled $2^31$ of the decomposition matrix of $\GU_7(q)$ \cite[Table 9]{DM}, we compute the projective character $\Psi'$ given by projection to the principal block of $ [R_{\GU_7(q)\times\GL_3(q^2)}^{\GU_{13}(q)} \overline{P}_{2^31}\otimes \St]$:
 	\begin{align*}
 	\Psi'&=2^31^7+43^21^3+4^31+43^3+\\
 	&\qquad2*(1^{13})+2*(2^21^9)+2*(2^61)+2*(3^31^4)+2*(32^41^2)
 	\end{align*}
 	As in the previous argument, $P_{4^31}$ must be a summand of $R_{\GU_7(q)\times\GL_3(q^2)}^{\GU_{13}(q)} P_{2^31}\otimes P_{\st}$ and $X_{4^31}$ then has cuspidal support $(\GU_7(q)\times\GL_3(q^2),X_{2^31}\otimes\st)$. 
 	
 	It remains to determine the Harish-Chandra series of $X_{5431}$ and $X_{32^41}$. Since $5431$ does not appear in either of the expressions $[R_{\GU_7(q)\times\GL_3(q^2)}^{\GU_{13}(q)} 1^7\otimes 1^3]$ or $[R_{\GU_7(q)\times\GL_3(q^2)}^{\GU_{13}(q)} \overline{P}_{2^31}\otimes 1^3]$, we know that $P_{5431}$ is not a summand of either $[R_{\GU_7(q)\times\GL_3(q^2)}^{\GU_{13}(q)} P_{1^7}\otimes P_{\st}$ or $R_{\GU_7(q)\times\GL_3(q^2)}^{\GU_{13}(q)} P_{2^31}\otimes P_{\st}$. 
 	Since $X_{5431}$ does not belong to any other Harish-Chandra series and is not cuspidal, it must belong to the Harish-Chandra series of $(\GU_1(q)\times\GU_3(q^2)^2,X_1\otimes\st^2)$. Next, to deal with $32^41$ we Harish-Chandra induce $1\otimes 1^6$ from $\GU_1(q)\times\GL_6(q^2)$. The simple $\GL_6(q^2)$-module $\st=X_{1^6}$ has cuspidal support $(\GL_3(q^2)^2,\st^2)$, and we know that some projective indecomposable summand of $R_{\GL_3(q^2)^2}^{\GU_{13}(q)}\st^2$
 	is the projective cover of a unipotent $\mathbbm{k}\GU_{13}(q)$-module with cuspidal support $(\GU_1(q)\times\GL_3(q)^2,X_1\otimes\st^2)$.  We find that $5431$ does not occur as a constituent $R_{\GL_6(q^2)}^{\GU_{13}(q)}1\otimes 1^6$
 	but $32^41$ does.  
 	We deduce that $X_{32^41}$ has cuspidal support $(\GU_1(q)\times\GL_3(q)^2,X_1\otimes\st^2)$. This concludes the classification of the Harish-Chandra series of the unipotent $\mathbbm{k}\GU_{13}(q)$-modules. The data matches the prediction of Theorem \ref{conj}. We summarize the positions in the $\slinf$-crystal and the Harish-Chandra series of the weakly cuspidal  unipotent $\mathbbm{k}\GU_{13}(q)$-modules in Table \ref{tab:Unitary3,n=13}.

 	\begin{table}[h]
 		\begin{tabular}{lllll}
 			\hline\\
 			$\slinf$-depth \quad\quad  & $t$ \quad\quad \quad  & $\lambda$ \quad\quad\quad   & $\tau(\lambda)$ \qquad \qquad \qquad & cuspidal support \\ 
 			\hline\\
 			$0$ \quad \quad \quad\quad  & $1$\quad \quad \quad  & $1^{13}$ \quad\quad \quad  & $|1^6.\emp,(-1,2)\rangle$\qquad \qquad  & $(\GU_{13}(q),\hbox{itself})$\\
 			&&&&\\
 			& \quad \quad \quad  & $2^61$ \quad\quad \quad  & $|2^3.\emp,(-1,2)\rangle$& \\
 			&&&&\\
 			& \quad \quad \quad  & $321^8$ \quad\quad \quad  & $|\emp.1^6,(-1,2)\rangle$& \\
 			
 			&&&&\\
 			& $2$ & $2^31^7$ & $|21^3.\emp,(-3,1)\rangle$&\\
 			&&&&\\
 			&  & $4321^4$ & $|\emp.1^5,(-3,1)\rangle$&\\
 			\\
 			\hline
 			&&&&\\
 			$1$ & $1$ & $3^31^4$ & $|21^2.1^2,(-1,2)\rangle$& $(\GU_7(q)\times \GL_3(q^2), X_{1^7}\otimes\st_3)$\\
 			
 			&&&&\\
 			& $2$ & $4^31$ & $|3.1^2,(-3,1)\rangle$ & $(\GU_7(q)\times\GL_3(q^2), X_{2^31}\otimes\st_3)$\\
 			&&&&\\
 			\hline
 			&&&&\\
 			$2$ & $1$ & $ 32^41^2$ & $|1^3.1^3,(-1,2) \rangle$ & $(\GU_1(q)\times \GL_3(q^2)^2, X_1\otimes\st_3^2)$\\
 			&&&&\\
 			& & $5431 $ & $|\emp.2^3, (-1,2)\rangle$ &$(\GU_1(q)\times\GL_3(q^2)^2, X_1\otimes\st_3^2)$\\
 			&&&&\\
 		\end{tabular}
 		\caption{Depth in the $\slinf$-crystal of all weakly cuspidal partitions  of $13$ (i.e. partitions of depth $0$ in the $\sle$-crystal) and cuspidal supports of the corresponding unipotent $\mathbbm{k}\GU_{13}$-modules when $e=3$. This illustrates Theorem \ref{conj} for $n=13$ and $e=3$.}
 		\label{tab:Unitary3,n=13}
 	\end{table}

 	\begin{table}[]
 		\begin{tabular}{lllll}
 			\hline\\
 			$\slinf$-depth \quad\quad  & $t$ \quad\quad \quad  & $\lambda$ \quad\quad\quad   & $\tau(\lambda)$ \qquad \qquad \qquad & cuspidal support \\ 
 			\hline\\
 			$0$ \quad \quad \quad\quad  & $1$\quad \quad \quad  & $1^{15}$ \quad\quad \quad  & $|1^7.\emp,(-1,2)\rangle$\qquad \qquad  & $(\GU_{15}(q),\hbox{itself})$\\
 			&&&&\\
 			& \quad \quad \quad  & $2^41^7$ \quad\quad \quad  & $|2^21^3.\emp,(-1,2)\rangle$& \\
 			&&&&\\
 			& \quad \quad \quad  & $2^61^3$ \quad\quad \quad  & $|2^31.\emp,(-1,2)\rangle$& \\
 			&&&&\\
 			& \quad \quad \quad  & $321^{10}$ \quad\quad \quad  & $|\emp.1^7,(-1,2)\rangle$& \\
 			&&&&\\
 			& \quad \quad \quad  & $3^321^4$ \quad\quad \quad  & $|2.1^5,(-1,2)\rangle$& \\
 			&&&&\\
 			& \quad \quad \quad  & $32^41^4$ \quad\quad \quad  & $|1^4.1^3,(-1,2)\rangle$& \\
 			
 			&&&&\\
 			& $2$ & $21^{13}$ & $|1^6.\emp,(-3,1)\rangle$&\\
 			&&&&\\
 			&  & $2^31^9$ & $|21^4.\emp,(-3,1)\rangle$&\\
 			&&&&\\
 			&  & $2^71$ & $|2^3.\emp,(-3,1)\rangle$&\\
 			&&&&\\
 			&  & $3^421$ & $|3^2.\emp,(-3,1)\rangle$&\\
 			&&&&\\
 			&  & $4321^6$ & $|\emp.1^6,(-3,1)\rangle$&\\
 			
 			&&&&\\
 			& $5$ & $54321$ & $|\emp.\emp,(-3,4)\rangle$&\\
 			\\
 			
 			\hline
 			&&&&\\
 			$1$ & $1$ & $3^31^6$ & $|21^3.1^2,(-1,2)\rangle$&$(\GU_9(q)\times\GL_3(q^2), X_{1^9}\otimes\st_3)$ \\
 			&&&&\\
 			& & $4^321$ & $|3^2.1,(-1,2)\rangle $ & $(\GU_9(q)\times\GL_3(q^2), X_{2^41}\otimes\st_3)$\\
 			&&&&\\
 			& & $5431^3$ & $|\emp.2^31,(-1,2)\rangle $ & $(\GU_9(q)\times\GL_3(q^2), X_{321^4}\otimes\st_3)$\\
 			&&&&\\
 			& $2$ & $432^21^4$ & $|1^3.1^3,(-3,1)\rangle$&$(\GU_9(q)\times\GL_3(q^2), X_{21^7}\otimes\st_3)$ \\
 			&&&&\\
 			&  & $4^31^3 $ & $|31.1^2,(-3,1)\rangle$& $(\GU_9(q)\times\GL_3(q^2), X_{2^31^3}\otimes\st_3)$\\
 			&&&&\\
 			\hline
 			&&&&\\
 			$2$ & $1$ & $3^32^3$ & $|21.1^4,(-1,2)\rangle$ & $(\GU_3(q)\times\GL_3(q^2)^2,X_{1^3}\otimes \st_3^2)$\\
 			&&&&\\
 			& & $5^3$ & $|3.2^2,(-1,2)\rangle$ & $(\GU_3(q)\times\GL_3(q^2)^2,X_{1^3}\otimes \st_3^2)$\\
 			&&&&\\
 			& $2$ & $432^4$ & $|1^2.1^4,(-3,1)\rangle$ & $(\GU_3(q)\times\GL_3(q^2)^2,X_{21}\otimes \st_3^2)$\\
 			&&&&\\
 			& & $654$ & $|\emp.2^3,(-3,1)\rangle$ & $(\GU_3(q)\times\GL_3(q^2)^2,X_{21}\otimes \st_3^2)$\\
 			&&&&\\
 			
 		\end{tabular}
 		\caption{Depth in the $\slinf$-crystal of all weakly cuspidal partitions  of $15$ (i.e. partitions of depth $0$ in the $\sle$-crystal) and cuspidal supports of the corresponding unipotent $\mathbbm{k}\GU_{15}$-modules when $e=3$. This illustrates Theorem \ref{conj} for $n=15$ and $e=3$.}
 		\label{tab:Unitary3,n=15}
 	\end{table}

 	\begin{table}[h]
 		\begin{tabular}{lllll}
 			\hline\\
 			$\slinf$-depth \quad\quad  & $t$ \quad\quad \quad  & $\lambda$ \quad\quad\quad   & $\tau(\lambda)$ \qquad \qquad \qquad & cuspidal support \\ \hline\\
 			$0$ \quad \quad \quad\quad  & $0$\quad \quad \quad  &  $1^{16}$\quad\quad \quad  & $|1^8.\emp,(-2,0)\rangle$\qquad \qquad  & $(\GU_{16}(q),\hbox{itself})$\\
 			&&&&\\
 			& \quad \quad \quad  & $2^31^{10}$\quad\quad \quad  & $|1.1^7,(-2,0)\rangle$& \\
 			&&&&\\
 			& \quad \quad \quad  & $2^61^4$\quad\quad \quad  & $|1^5.1^3,(-2,0)\rangle$ &\\
 			&&&&\\
 			& \quad \quad \quad  & $32^41^5$\quad\quad \quad  & $|2^31^2.\emp,(-2,0)\rangle$& \\
 			&&&&\\
 			& \quad \quad \quad  & $432^41$ \quad\quad \quad  & $|\emp.2^4,(-2,0)\rangle$  &\\
 			&&&&\\
 			& $3$ & $321^{11}$ &$|1^5.\emp,(-2,3)\rangle$ &\\
 			&&&&\\
 			&  &$3^32^31$ &$|32.\emp,(-2,3)\rangle$ &\\
 			&&&&\\
 			& $4$ & $4321^7$&$|1^3.\emp,(-4,2)\rangle$ &\\
 			\\
 			\hline
 			&&&&\\
 			$1$ & $0$ & $3^31^7$ & $|2^21^3.1,(-2,0)\rangle$& $(\GU_{10}(q)\times \GL_3(q^2), X_{1^{10}}\otimes\st_3)$\\
 			&&&&\\
 			&  & $4^31^4$ & $|2.2^21^2,(-2,0)\rangle$& $(\GU_{10}(q)\times \GL_3(q^2), X_{2^31^4}\otimes\st_3)$\\
 			&&&&\\
 			& $3$ & $5431^4$ & $|1^2.1^3,(-2,3)\rangle$& $(\GU_{10}(q)\times \GL_3(q^2), X_{321^5}\otimes\st_3)$\\
 			&&&&\\
 			& $4$ & $6541$ & $|\emp.1^3,(-4,2)\rangle$& $(\GU_{10}(q)\times \GL_3(q^2), X_{4321}\otimes\st_3)$\\
 			\\
 			\hline
 			&&&&\\
 			$2$ & $0$ & $3^421^2$ & $|2^2.1^4,(-2,0)\rangle$& $(\GU_{4}(q)\times \GL_3(q^2)^2, X_{1^4}\otimes\st_3^2)$\\
 			&&&&\\
 			&  & $5^31$ & $|3^2.2,(-2,0)\rangle$& $(\GU_{4}(q)\times \GL_3(q^2)^2, X_{1^4}\otimes\st_3^2)$\\
 			\\
 		\end{tabular}
 		\caption{Depth in the $\slinf$-crystal of all weakly cuspidal partitions  of $16$ (i.e. partitions of depth $0$ in the $\sle$-crystal) and cuspidal supports of the corresponding unipotent $\mathbbm{k}\GU_{16}$-modules when $e=3$. This illustrates Theorem \ref{conj} for $n=16$ and $e=3$.}
 		\label{tab:Unitary3,n=16}
 	\end{table}

 
 	\begin{table}[]
 		\begin{tabular}{lllll}
 			\hline\\
 			$\slinf$-depth \quad\quad  & $t$ \quad\quad \quad  & $\lambda$ \quad\quad\quad   & $\tau(\lambda)$ \qquad \qquad \qquad & cuspidal support \\ 
 			\hline\\
 			$0$ \quad \quad \quad\quad  & $0$\quad \quad \quad  & $1^{18}$ \quad\quad \quad  & $|1^9.\emp,(-2,0)\rangle$\qquad \qquad  & $(\GU_{18}(q),\hbox{itself})$\\
 			&&&&\\
 			& \quad \quad \quad  &  $21^{16}$\quad\quad \quad  & $|\emp.1^9,(-2,0)\rangle$& \\
 			&&&&\\
 			& \quad \quad \quad  & $2^31^{12}$ \quad\quad \quad  & $|1.1^8,(-2,0)\rangle$& \\
 			&&&&\\
 			& \quad \quad \quad  & $2^41^{10}$ \quad\quad \quad  & $|1^7.1^2,(-2,0)\rangle$& \\
 			&&&&\\	
 			& \quad \quad \quad  & $2^61^6$ \quad\quad \quad  & $|1^6.1^3,(-2,0)\rangle$& \\
 			&&&&\\
 			& \quad \quad \quad  & $2^71^4$ \quad\quad \quad  & $|1^3.1^6,(-2,0)\rangle$& \\
 			&&&&\\
 			& \quad \quad \quad  & $32^41^7$ \quad\quad \quad  & $|2^31^3.\emp,(-2,0)\rangle$& \\
 			&&&&\\
 			& \quad \quad \quad  & $3^32^41$ \quad\quad \quad  & $|2^4.1,(-2,0)\rangle$& \\
 			&&&&\\
 			& \quad \quad \quad  & $3^421^4$ \quad\quad \quad  & $|2^2.1^5,(-2,0)\rangle$& \\
 			&&&&\\
 			& \quad \quad \quad  & $432^41^3$ \quad\quad \quad  & $|\emp.2^41,(-2,0)\rangle$& \\
 			&&&&\\
 			\quad \quad \quad\quad  & $3$\quad \quad \quad  & $321^{13}$ \quad\quad \quad  & $|1^6.\emp,(-2,3)\rangle$\qquad \qquad  & \\
 			&&&&\\
 			\quad \quad \quad\quad  & \quad \quad \quad  & $32^71$ \quad\quad \quad  & $|2^3.\emp,(-2,3)\rangle$\qquad \qquad  & \\
 			&&&&\\
 			\quad \quad \quad\quad  & \quad \quad \quad  & $3^32^31^3$ \quad\quad \quad  & $|321.\emp,(-2,3)\rangle$\qquad \qquad  & \\
 			&&&&\\
 			\quad \quad \quad\quad  & \quad \quad \quad  & $3^321^7$ \quad\quad \quad  & $|31^3.\emp,(-2,3)\rangle$\qquad \qquad  & \\
 			&&&&\\
 			\quad \quad \quad\quad  & \quad \quad \quad  & $54321^4$ \quad\quad \quad  & $|\emp.1^6,(-2,3)\rangle$\qquad \qquad  & \\
 			&&&&\\
 			\quad \quad \quad\quad  & $4$\quad \quad \quad  & $4321^9$ \quad\quad \quad  & $|1^4.\emp,(-4,2)\rangle$\qquad \qquad  & \\
 			&&&&\\
 			\quad \quad \quad\quad  & \quad \quad \quad  & $4^3321$ \quad\quad \quad  & $|4.\emp,(-4,2)\rangle$\qquad \qquad  & \\
 			&&&&\\
 			\quad \quad \quad\quad  & \quad \quad \quad  & $432^51$ \quad\quad \quad  & $|2^2.\emp,(-4,2)\rangle$\qquad \qquad  & \\
 		\end{tabular}
 		\caption{The $18$ partitions of $18$ labeling cuspidal unipotent $\GU_{18}(q)$-representations when $e=3$.}
 		\label{tab:Unitary3,n=18, Cuspidals}
 	\end{table}

 	\begin{table}[]
 		\begin{tabular}{lllll}
 			\hline\\
 			$\slinf$-depth \quad\quad  & $t$ \quad\quad \quad  & $\lambda$ \quad\quad\quad   & $\tau(\lambda)$ \qquad \qquad \qquad & expected cuspidal support \\ 
 			\hline\\
 			
 			&&&&\\
 			$1$ & $0$ & $3^31^9$ & $|2^21^4.1,(-2,0)\rangle$ & $(\GU_{12}(q)\times\GL_3(q^2),X_{1^{12}}\otimes \st_3)$\\
 			&&&&\\
 			&  & $432^21^7$ & $|\emp.2^31^3,(-2,0)\rangle$ & $(\GU_{12}(q)\times\GL_3(q^2),X_{21^{10}}\otimes \st_3)$\\
 			&&&&\\
 			&  & $4^31^6$ & $|2.2^21^3,(-2,0)\rangle$ & $(\GU_{12}(q)\times\GL_3(q^2),X_{2^31^6}\otimes \st_3)$\\
 			&&&&\\
 			&  & $4^321^4$ & $|21^3.2^2,(-2,0)\rangle$ & $(\GU_{12}(q)\times\GL_3(q^2),X_{2^41^4}\otimes \st_3)$\\
 			&&&&\\
 			&  & $543^221$ & $|3^3.\emp,(-2,0)\rangle$ & $(\GU_{12}(q)\times\GL_3(q^2),X_{32^41}\otimes \st_3)$\\
 			&&&&\\
 			& $3$ & $5431^6$ & $|1^3.1^3,(-2,3)\rangle$ & $(\GU_{12}(q)\times\GL_3(q^2),X_{321^7}\otimes \st_3)$\\
 			&&&&\\
 			&  & $5^321$ & $|4.1^2,(-2,3)\rangle$ & $(\GU_{12}(q)\times\GL_3(q^2),X_{3^321}\otimes \st_3)$\\
 			&&&&\\
 			& $4$ & $6541^3$ & $|1.1^3,(-4,2)\rangle$ & $(\GU_{12}(q)\times\GL_3(q^2),X_{4321^3}\otimes \st_3)$\\
 			&&&&\\
 			
 			\hline
 			&&&&\\
 			$2$ & $0$ & $3^6$ & $|2^3.1^3,(-2,0)\rangle$ & $(\GU_{6}(q)\times\GL_3(q^2)^2,X_{1^6}\otimes \st_3^2)$\\
 			&&&&\\
 			&  & $5^31^3$ & $|3^21.2,(-2,0)\rangle$ & $(\GU_{6}(q)\times\GL_3(q^2)^2,X_{1^6}\otimes \st_3^2)$\\
 			&&&&\\
 			&  & $43^42$ & $|1^3.2^3,(-2,0)\rangle$ & $(\GU_{6}(q)\times\GL_3(q^2)^2,X_{21^4}\otimes \st_3^2)$\\
 			&&&&\\
 			&  & $65421$ & $|\emp.3^3,(-2,0)\rangle$ & $(\GU_{6}(q)\times\GL_3(q^2)^2,X_{21^4}\otimes \st_3^2)$\\
 			&&&&\\
 			& $3$ & $5432^3$ & $|1.1^5,(-2,3)\rangle$ & $(\GU_{6}(q)\times\GL_3(q^2)^2,X_{321}\otimes \st_3^2)$\\
 			&&&&\\
 			&  & $765$ & $|\emp.2^3,(-2,3)\rangle$ & $(\GU_{6}(q)\times\GL_3(q^2)^2,X_{321}\otimes \st_3^2)$\\
 			&&&&\\
 			
 			\hline
 			&&&&\\
 			$3$ & $0$ & $2^9$ & $|1^4.1^5,(-2,0)\rangle$ & $(\GL_3(q^2)^3,\st_3^3)$\\
 			&&&&\\
 			&  & $4^32^3$ & $|21^2.2^21,(-2,0)\rangle$ & $(\GL_3(q^2)^3,\st_3^3)$\\
 			&&&&\\
 			&  & $6^3$ & $|3.3^2,(-2,0)\rangle$ & $(\GL_3(q^2)^3,\st_3^3)$\\
 			&&&&\\
 		\end{tabular}
 		\caption{Depth in the $\slinf$-crystal of all weakly cuspidal partitions  of $18$  of depths $1$, $2$, and $3$ in the $\slinf$-crystal, and cuspidal supports of the corresponding unipotent $\mathbbm{k}\GU_{18}$-modules when $e=3$. Together with Table \ref{tab:Unitary3,n=18, Cuspidals} this illustrates Theorem \ref{conj} for $n=18$ and $e=3$.}
 		\label{tab:Unitary3,n=18 depth 1,2,3}
 	\end{table}
 	
 	\begin{table}[]
 		\begin{tabular}{lllll}
 			\hline\\
 			$\slinf$-depth \quad\quad  & $t$ \quad\quad \quad  & $\lambda$ \quad\quad\quad   & $\tau(\lambda)$ \qquad \qquad \qquad & cuspidal support \\ 
 			\hline\\
 			$0$ \quad \quad \quad\quad  & $1$\quad \quad \quad  & $1^{19}$ \quad\quad \quad  & $|1^9.\emp,(-1,2)\rangle$\qquad \qquad  & $(\GU_{19}(q),\hbox{itself})$\\
 			&&&&\\
 			& \quad \quad \quad  & $2^61^7$ \quad\quad \quad  & $|2^31^3.\emp,(-1,2)\rangle$& \\
 			&&&&\\
 			& \quad \quad \quad  &  $321^{14}$\quad\quad \quad  & $|\emp.1^9,(-1,2)\rangle$& \\
 			&&&&\\
 			& \quad \quad \quad  &  $32^41^8$\quad\quad \quad  & $|1^6.1^3,(-1,2)\rangle$& \\
 			&&&&\\
 			& \quad \quad \quad  & $3^32^31^4$ \quad\quad \quad  & $|21.1^6,(-1,2)\rangle$& \\
 			&&&&\\
 			& \quad \quad \quad  &  $43^421$ \quad\quad \quad  & $|3^3.\emp,(-1,2)\rangle$& \\
 			
 			&&&&\\
 			& $2$ & $2^31^{13}$ & $|21^6.\emp,(-3,1)\rangle$&\\
 			&&&&\\
 			&  & $2^91$ & $|2^4.\emp,(-3,1)\rangle$&\\
 			&&&&\\
 			&  & $3^421^5$ & $|3^21^2.\emp,(-3,1)\rangle$&\\
 			&&&&\\
 			&  & $4321^{10}$ & $|\emp.1^8,(-3,1)\rangle$&\\
 			&&&&\\
 			&  & $432^41^4$  & $|1^4.1^4,(-3,1)\rangle$&\\
 			
 			&&&&\\
 			& $5$ & $5432^31$ & $|2.\emp,(-3,4)\rangle$&\\
 			\\

 		\end{tabular}
 		\caption{The $12$ partitions of $19$ labeling cuspidal unipotent $\GU_{19}(q)$-representations when $e=3$.}
 		\label{tab:Unitary3,n=19, Cuspidals}
 	\end{table}

 	\begin{table}[]
 		\begin{tabular}{lllll}
 			\hline\\
 			$\slinf$-depth \quad\quad  & $t$ \quad\quad \quad  & $\lambda$ \quad\quad\quad   & $\tau(\lambda)$ \qquad \qquad \qquad & expected cuspidal support \\ 
 			\hline
 			&&&&\\
 			$1$ & $1$ & $3^31^{10}$ & $|21^5.1^2,(-1,2)\rangle$&$(\GU_{13}(q)\times\GL_3(q^2), X_{1^{13}}\otimes\st_3)$ \\
 			&&&&\\
 			& &$4^32^31$  & $|3^22.1,(-1,2)\rangle$&$(\GU_{13}(q)\times\GL_3(q^2), X_{2^61}\otimes\st_3)$ \\
 			&&&&\\
 			& & $5431^7$ & $|\emp.2^31^3,(-1,2)\rangle$&$(\GU_{13}(q)\times\GL_3(q^2), X_{321^8}\otimes\st_3)$ \\
 			&&&&\\
 			& $2$ & $4^31^7$ & $|31^3.1^2,(-3,1)\rangle$&$(\GU_{13}(q)\times\GL_3(q^2), X_{2^31^7}\otimes\st_3)$ \\
 			&&&&\\
 			&  & $6541^4$ & $|\emp.2^31^2,(-3,1)\rangle$&$(\GU_{13}(q)\times\GL_3(q^2), X_{4321^4}\otimes\st_3)$ \\
 			&&&&\\
 			
 			\hline\\

 			$2$ & $1$ & $3^61$ & $|2^3.1^3,(-1,2)\rangle$ & $(\GU_{7}(q)\times\GL_3(q^2)^2,X_{1^7}\otimes \st_3^2)$\\
 			&&&&\\
 			&  & $5^31^4$ & $|31^2.2^2,(-1,2)\rangle$ & $(\GU_{7}(q)\times\GL_3(q^2)^2,X_{1^7}\otimes \st_3^2)$\\
 			&&&&\\
 			& $2$ & $4^3321^2$  & $|3.1^5,(-3,1)\rangle$ & $(\GU_{7}(q)\times\GL_3(q^2)^2,X_{2^31}\otimes \st_3^2)$\\
 			&&&&\\
 			&  & $6^31$ & $|4.2^2,(-3,1)\rangle$ & $(\GU_{7}(q)\times\GL_3(q^2)^2,X_{2^31}\otimes \st_3^2)$\\
 			&&&&\\
 			
 			\hline
 			&&&&\\
 			$3$ & $1$ & $32^71^2$ & $|1^3.1^6,(-1,2)\rangle$ & $(\GU_{1}(q)\times\GL_3(q^2)^3,X_{1}\otimes \st_3^3)$\\
 			&&&&\\
 			& & $543^221^2$ & $|1^3.2^3,(-1,2)\rangle$ & $(\GU_{1}(q)\times\GL_3(q^2)^3,X_{1}\otimes \st_3^3)$\\
 			&&&&\\
 			&  & $7651$ & $|\emp.3^3,(-1,2)\rangle$ & $(\GU_{1}(q)\times\GL_3(q^2)^3,X_{1}\otimes \st_3^3)$\\
 			&&&&\\
 		\end{tabular}
 		\caption{Depth in the $\slinf$-crystal of all weakly cuspidal partitions  of $19$  of depths $1$, $2$ and $3$ in the $\slinf$-crystal, and cuspidal supports of the corresponding unipotent $\mathbbm{k}\GU_{19}$-modules when $e=3$. Together with Table \ref{tab:Unitary3,n=19, Cuspidals} this  illustrates Theorem \ref{conj} for $n=19$ and $e=3$.}
 		\label{tab:Unitary3,n=19 depth 1,2,3}
 	\end{table}

 	The case $n=15$ follows from similar arguments; the data on Harish-Chandra series of weak cuspidals and their depths in the $\slinf$-crystal is summarized in Table \ref{tab:Unitary3,n=15}.
 	There are $12$ cuspidals. We begin by checking the weakly cuspidal partitions of depth $1$ in the $\slinf$-crystal. Recall from \cite[Table 11]{DM} that there are $5$ cuspidals of $\GU_9(q)$, given by $1^9,\;21^7,\;,2^31^3,\;2^41,\;321^4$. We will check that for $\lambda$ one of these five partitions, $R_{\GU_9(q)\times\GL_3(q^2)}^{\GU_{15}(q)}X_\lambda\otimes\st$ is indecomposable with simple head $X_{\tilde{a}_1(\lambda)}$, where we write $\tilde{a}_1(\lambda)$ as shorthand for $\tau^{-1}(\tilde{a}_1(\tau(\lambda)))$.
 	
 	The partition $1^9$ is the unipotent part of the projective character of $P_{1^9}$, and $\tau(1^9)=|1^4.\emp,(-1,2)\rangle$. Then $\tilde{a}_1(\tau(1^9))=|21^3.1^2,(-1,2)\rangle$, and $\tau^{-1}(\tilde{a}_1(\tau(1^9)))=3^31^6$. We compute $\tau^{-1}(\Ind_{B^4\times S_3}^{B_7}1^4.\emp\otimes 1^3)$ to obtain that the unipotent part of the projective character of  $P=R_{\GU_9(q)\times\GL_3(q^2)}^{\GU_{15}(q)}P_{1^9}\otimes P_{\st}$ is given by
 	\(\Psi= 2^61^3+2^41^7+1^{15}+32^41^4+3^31^6+31^{12}+3^21^9\). 
 	Since $3^31^6$ dominates the other partitions, $P_{3^31^6}\mid P$. Moreover, $2^61^3,\;2^41^7,\;1^{15},\;32^41^4$ all label cuspidals, while $31^{12}$ and $3^21^9$ have nonzero depth in the $\sle$-crystal. It follows that among simples labeled by these partitions only $X_{3^31^6}$ can belong to the same Harish-Chandra series as $X_{1^9}\otimes \st$.
 	
 	The partition $21^7$ is the unipotent part of the projective character of $P_{21^7}$, and $\tau(21^7)=|1^3.\emp,(-3,1)\rangle$. As in the previous computation, \(\Psi=2^71+2^31+21^{13}+432^21^4+41^{11}\) is the unipotent part of the projective character obtained by Harish-Chandra inducing $P_{21^7}\otimes P_{\st}$ then cutting to the principal block. Since $432^21^4$ dominates, $P_{432^21^4}$ is a summand, and moreover $432^21^4=\tau^{-1}(|1^3.1^3,(-3,1)\rangle)=\tau^{-1}(\tilde{a}_1(\tau(21^7)))$. All other partitions appearing in $\Psi$ either label cuspidals or have nonzero depth in the $\sle$-crystal. The claim follows for $21^7$.
 	
 	The remaining three partitions to consider label projectives with more complicated characters. Write $\rL=\GU_9(q)\times\GL_3(q^2)$ and $\rG=\GU_{15}(q)$. Here, what we do is compose projection to the principal block with Harish-Chandra induction from $\rL$ of the unipotent part of the projective characters which are non-negative linear combinations of $\lambda$, $21^7$, and $1^9$, and observe, that of all $\mu\vdash 15$ appearing in $R_{\rL}^{\rG}P_\lambda$, the most dominant in each case is $\tilde{a}_1(\lambda)$. This implies $X_{\tilde{a}_1(\lambda)}$ has the desired cuspidal support. Moreover, the four partitions of depth $2$ in the $\slinf$-crystal do not appear in these characters at all. This implies the claim.
 	
 	All that's left to check are the Harish-Chandra series of the four weakly cuspidal partitions of depth $2$ in the $\slinf$-crystal. Computing as usual, we observe that $5^3$ dominates among all $\mu\vdash 15$ appearing in $R_{L}^GP_{3^3}\otimes P_{\st}$, while $3^32^3$ dominates among all $\mu\vdash 15$ appearing in $R_{\GU_3(q)\times\GL_6(q^2)}^{\GU_{15}(q)}P_{1^3}\otimes P_{\st}$, which implies (since they do not belong to a bigger Harish-Chandra series) that $X_{5^3}$ and $X_{3^32^3}$ are supported on $(\GU_3(q)\times\GL_3(q^2)^2,X_{1^3}\otimes \st^2)$ (this is the cuspidal support of $X_{3^3}\otimes\st$). Likewise, $654$ dominates in $R_{L}^GP_{3^3}\otimes P_{\st}$, $432^4$ dominates in $R_{\GU_3(q)\times\GL_6(q^2)}^{\GU_{15}(q)}P_{21}\otimes P_{\st}$, the cuspidal support of $X_{3^3}\otimes\st$ is $\GU_3(q)\times\GL_3(q^2)^2, X_{21}\otimes\st^2)$, confirming that $X_{654}$ and $X_{432^4}$ are supported on $(\GU_3(q)\times\GL_3(q^2)^2,X_{21}\otimes \st^2)$. This concludes the check that the $\slinf$-crystal describe the Harish-Chandra branching rule on weak cuspidals for $\GU_{15}(q)$ when $e=3$, implying Theorem \ref{conj} is true for $\GU_{15}(q)$ when $e=3$.
 	
 	The $n=16$ case is more straightforward than the $n=15$ case and uses similar arguments; the results are summarized in Table \ref{tab:Unitary3,n=16}.
 \end{proof}

 \subsection{Inducing the Steinberg representation from a maximal type $A$ Levi subgroup}
 One of the most-studied modules in characteristic $\ell$ is the Steinberg PIM $P_{\st_n}=P_{1^{n}}$ of $\GL_n(q)$, the projective cover of the irreducible Steinberg representation $\st_n=X_{1^n}\in\mathbbm{k}\GL_n(q)$-mod  \cite{Ackermann},\cite{Enomoto},\cite{Geck2}. We are going to study $R_{\GU_\iota(q)\times\GL_n(q^2)}^{\GU_{2n+\iota}(q)}P_{\iota}\otimes P_{\st_n}$  when $e\mid n$. This will allow us to check directly that 
 $R_{\GU_\iota(q)\times\GL_n(q^2)}^{\GU_{2n+\iota}(q)}X_{\iota}\otimes \st_n$  
 agrees with Theorem \ref{conj} for any $n\in\N$. 
 
 We need the following combinatorial lemma, which is easily proved using the characterization of the $e$-core of a partition in terms of its $e$-abacus as in  \cite{James}:
 \begin{lemma}\label{combinat2col}	Suppose $e\in\N_{\geq 3}$ is odd and $e\mid m$. \begin{enumerate}
 		\item The $e$-core of $2^j1^{2m-2j}$ is equal to $\emp$ if and only if $j=0$ or $1$ mod $e$.
 		\item The $e$-core of $32^j1^{2m-2j-2}$ is equal to $(1)$ if and only if $j=1$ or $-2$ mod $e$.
 	\end{enumerate}
 \end{lemma}
 \noindent Some general results on identifying Harish-Chandra series of some unipotent modules using the formalism of Hom functors and $q$-Schur algebras were proved in \cite{DipperGruber} but the particular statement we prove next seems to be new. A special case, when $e=3$ and $\lambda=2^3$, was shown in \cite[Section 2]{GHJ}, where it is deduced from \cite[Lemma 3.16]{GHM2} and \cite[Proposition 2.3.5]{Gruber}.
 \begin{theorem}\label{inducing the type A Steinberg}Suppose $e$ is the order of $-q$ mod $\ell$, $e\geq 3$ is odd, and $\ell$ is sufficiently large.  \begin{enumerate}\item For any integer $k\geq 1$, the indecomposable direct summand of 
 		$R_{\GL_{ek}(q^2)}^{\GU_{2ek}(q)} P_{\st_{ek}}$ lying in the principal block is the projective indecomposable module $P_{2^{ek}}$. The unipotent part of $P_{2^{ek}}$ is the sum over all $2$-column partitions in which the number of rows equal to $2$ is congruent to $0$ or $1$ mod $e$. The 
 		module $R_{\GL_{ek}(q^2)}^{\GU_{2ek}(q)} \st_{ek}$ is indecomposable with simple head $X_{2^{ek}}$.
 		\item For any integer $k\geq 1$, the indecomposable direct summand of $R_{\GU_1(q)\times\GL_{ek}(q^2)}^{\GU_{2ek+1}(q)}P_1\otimes P_{\st_{ek}}$ lying in the principal block is the projective indecomposable module $P_{32^{ek-2}1^2}$. The 
 		module $R_{\GU_1(q)\times\GL_{ek}(q^2)}^{\GU_{2ek+1}(q)}X_1\otimes \st_{ek}$ is indecomposable with simple head $X_{32^{ek-2}1^2}$.
 	\end{enumerate}
 \end{theorem}
 
 \begin{proof} (1)\; 
 	The unipotent part of the PIM $P_{\st}$ of $\GL_{ek}(q^2)$ is given by the partition $1^{ek}=\St_{ek}$. Using $\tau$ and performing the induction on the level of Weyl groups from $S_{ek}$ to $B_{ek}$, we get that \[\tau\left([R_{\GL_{ek}(q^2)}^{\GU_{2ek}(q)} \St_{ek}]\right)=1^{ek}.\emp + 1^{ek-1}.1 + 1^{ek-2}.1^2+\dots+1^2.1^{ek-2}+1.1^{ek-1}+\emp.1^{ek} \]
 	Applying the inverse to $\tau$ with $t=0$, 
 	\[R_{\GL_{ek}(q^2)}^{\GU_{2ek}(q)} \St_{ek}=1^{2ek}+2^21^{2ek-4}+2^41^{2ek-8}+\dots + 2^51^{2ek-10}+2^31^{2ek-6}+21^{2ek-2}\]
 	Since $\lambda$ and $\mu$ are in the same block if and only if $\lambda$ and $\mu$ have the same $e$-core, Lemma \ref{combinat2col} implies that the projection to the principal block of 
 	$[R_{\GL_{ek}(q^2)}^{\GU_{2ek}(q)}\St_{ek}]$ is the unipotent part of a projective character $\Psi$ given by
 	$$\Psi:=1^{2ek}+21^{2ek-2}+2^e1^{2ek-2e}+2^{e+1}1^{2ek-2e-2}+\dots+2^{ek}$$
 	Since $2^{ek}$ is maximal with respect to the lexicographic order among constitutents of $\Psi$, the projective cover $P_{2^{ek}}$ of $X_{2^{ek}}$ has to be a direct summand of the projective $P$ with character $\Psi$.

 	We claim that $P=P_{2^{ek}}$. This follows if $X_\lambda$ is cuspidal for all $\lambda\neq 2^{ek}$ appearing in $\Psi$. 
 	Indeed, if $X_\lambda$ is cuspidal then $P_\lambda$ cannot be a summand of $P$, because the projective cover of a cuspidal module cannot be the direct summand of the Harish-Chandra induction of a projective indecomposable module from a proper Levi subgroup. 
 	
 	To check that the partitions $\lambda=2^j1^{2ek-2j}$, $j=0$ or $1$ mod $e$, and $j\neq ek$, label cuspidal unipotent $\mathbbm{k}\GU_{2ek}(q)$-modules $X_\lambda$, we look at the abacus of $\tau(\lambda)$ for each such $\lambda$. The relevant charge $\bs_0$ for the Fock space is equivalent (up to a shift $(c,c)$ for some $c\in\Z$) to $(-\frac{e+1}{2},0)$. 
 	We have $\tau(\lambda)=1^a.1^{ek-a}$ with $a=0\mod e$ or $a=\frac{e-1}{2} \mod e$. If $a=0 \mod e$ then $\tau(\lambda)$ is totally $e$-periodic by \cite[Lemma 4.4]{GerberN} and therefore $X_\lambda$ is weakly cuspidal.
 	If, moreover, $a\neq \lfloor\frac{ek}{2}\rfloor$ then the abacus of $\tau(\lambda)$ also satisfies the pattern avoidance condition of Theorem \ref{slinf source} and so $X_\lambda$ is not only weakly cuspidal but actually cuspidal. 
 	
 	If $a=\frac{e-1}{2} $ mod $e$ and $a\neq \lfloor\frac{ek}{2}\rfloor$ then the abacus of $\tau(\lambda)$ satisfies the pattern avoidance condition of Theorem \ref{slinf source} and $\lambda$ is a source vertex of the $\slinf$-crystal. To verify that $\lambda$ is also a source vertex of the $\sle$-crystal we need to show that $\cA(\tau(\lambda))$ is totally $e$-periodic. We have $\tau(\lambda)\simeq|1^{\frac{e-1}{2}+em_1}.1^{\frac{e+1}{2}+em_2},(-\frac{e+1}{2},0)\rangle$ for $m_1+m_2=k-1$. Write $\cA:=\cA|1^{\frac{e-1}{2}+em_1}.1^{\frac{e+1}{2}+em_2},(-\frac{e+1}{2},0)\rangle$. Then 
 	$P:=\{(1,2),(0,2),(-1,2),\dots,(-\frac{e+1}{2}+2,2),(-\frac{e+1}{2}+1,1)(-e+2,1)\}\subset \cA$ but $(1,1),(0,1),\dots,(-\frac{e+1}{2}+2,1)\notin \cA.$  
 	By definition, the first $e$-period of $\cA$ then exists and is equal to $P$. If we remove $P$ from $\cA$ then the abacus $\cA\setminus P$ that remains is $\cA|1^{em_1}.1^{em_2},(-\frac{e+1}{2},0)\rangle$, and in it 
 	there are a multiple of $e$ beads to the right of the space in row $j$ that marks the bottom of the column partition $1^{em_j}$, $j=1,2$. Then by \cite[Lemma 4.4]{GerberN} the abacus $\cA\setminus P$ is totally $e$-periodic.  Therefore $\cA$ is totally $e$-periodic and  $\lambda$ is a source vertex of the $\sle$-crystal.

 	Therefore $P$ is indecomposable and equal to $P_{2^{ek}}$. We have $\tau(2^{ek})=|1^{\lfloor \frac{ek}{2}\rfloor}.1^{\lfloor \frac{ek}{2}\rfloor +\epsilon},(-\frac{e+1}{2},0)\rangle$ with $\epsilon=0$ if $k$ is even, and $\epsilon=1$ if $k$ is odd. Its abacus $\cA$ is totally $e$-periodic. 
 	Then $X_{2^{ek}}$ is weakly cuspidal and so cannot have cuspidal support on a Levi subgroup containing a factor of $\GL_1(q^2)$. The only totally $e$-periodic abaci of partitions $1^j.1^{ek-j}$ are those in the principal block (see the remarks on ``combinatorial blocks" in the proof of Theorem \ref{blocks with cuspidals}), so any other unipotent character appearing in $R_{\GL_{ek}(q^2)}^{\GU_{2ek}(q)} P_{\st_{ek}}$ belongs to the projective cover of a non-weakly cuspidal simple module.  It follows that $X_{2^{ek}}$ has the same cuspidal support as $\st_{ek}\in\mathbbm{k}\GL_{ek}(q^2)$, namely $(\GL_e(q^2)^k,\st_e^k)$, and that $R_{\GL_{ek}(q^2)}^{\GU_{2ek}(q)} \st_{ek}$ is indecomposable with simple head $X_{2^{ek}}$.
 	
 	Let us compute $\tilde{a}_{(1^k)}|\emp.\emp,(-\frac{e+1}{2},0)\rangle $. This is given by moving the $e$-periods $P_1,P_2,\dots,P_k$ of $\cA|\emp.\emp,\bs_0\rangle$ one step each to the right. For $e$ odd and charge $\bs_0=(-\frac{e+1}{2},0)$, the $e$-periods of $\cA|\emp.\emp,\bs_0\rangle$ form interlocking puzzle-pieces of alternating shapes: for $j$ odd, $P_j$ has $\frac{e+1}{2}$ beads in the top row and $\frac{e-1}{2}$ beads in the bottom row, while for $j$ even, $P_j$ is given by rotating $P_{j-1}$ $180$ degrees (so has $\frac{e-1}{2}$ beads in the top row, $\frac{e+1}{2}$ beads in the bottom row).
 	We then see that $\tilde{a}_{(1^k)}\emp.\emp = 1^{\lfloor \frac{ek}{2}\rfloor}.1^{\lfloor \frac{ek}{2}\rfloor +\epsilon}=\tau(2^{ek})$. The depth of $\tau(2^{ek})$ in the $\slinf$-crystal is therefore $k$ and its source vertex is $\emp.\emp$. Theorem \ref{conj} then predicts that the cuspidal support should be $(\GL_e(q^2)^k,\st_e^k)$, as we have found is the case.
 	This proves part (1).
 	
 	(2)\;The proof of part (2) is similar to the proof of part (1). Taking  $\sum_{m=0}^{ek}\tau^{-1}|1^{ek-m}.1^{m},\bs_1\rangle$ and then discarding all summands with $e$-core different from $(1)$, Lemma \ref{combinat2col} gives that the projection to the principal block of $[R_{\GU_1(q)\times\GL_{ek}(q^2)}^{\GU_{2ek+1}(q)}(1)\otimes\St_{ek}]$ is the sum over partitions of the form $32^j1^{2ek-2-2j}$ where $j=1$ or $-2$ mod $e$, plus the partition $1^{2ek+1}$. The largest of these in lexicographical order is $32^{ek-2}1^2$, and therefore $P_{32^{ek-2}1^2}$ is a direct summand of $R_{\GU_1(q)\times\GL_{ek}(q^2)}^{\GU_{2ek+1}(q)}P_1\otimes P_{1^{ek}}$. We have \[\tau(32^{ek-2}1^2)=\begin{cases}|1^{\frac{ek}{2}}.1^{\frac{ek}{2}},(-1,\frac{e+1}{2})\rangle\hbox{ if }k\hbox{ is even}\\
 	|1^{\frac{ek-3}{2}}.1^{\frac{ek-3}{2}+3},(-1,\frac{e+1}{2})\rangle\hbox{ if }k\hbox{ is odd}
 	\end{cases}
 	\]
 	It is then easy to check as in the proof of part (1) that $\tau(32^{ek-2}1^2)=\tilde{a}_{(k)}|\emp.\emp,\bs_1\rangle$. The module $X_{32^{ek-2}1^2}$ is weakly cuspidal since $\tau(32^{ek-2}1^2)$ is a source vertex of the $\sle$-crystal, and so $P_{32^{ek-2}1^2}$ being a summand of $R_{\GU_1(q)\times\GL_{ek}(q^2)}^{\GU_{2ek+1}(q)}P_1\otimes P_{1^{ek}}$ implies that $X_{32^{ek-2}1^2}$ has the same cuspidal support as $X_1\otimes X_{1^{ek}}\in\mathbbm{k}\left(\GU_1(q)\times\GL_{ek}(q^2)\right)$-mod.
 	This shows that Theorem \ref{conj} holds for the unipotent $\mathbbm{k}\GU_{2ek+1}$-module $X_{32^{ek-2}1^2}$. The verification that the other partitions $32^j1^{2ek-2-2j}$, $j=1$ or $-2$ mod $e$, label cuspidals $X_{32^j1^{2ek-2-2j}}$ is similar to the verification in Part (1). This concludes the proof of the theorem. 
 	
 \end{proof}
 
 \begin{example} Let $e=7$. We illustrate the interlocking pattern of $7$-periods on the abacus of $\emp.\emp$ with the charge $(-4,0)=(-\frac{e+1}{2},0)$:
 	$$
 	\dots\quad\TikZ{[scale=.5]
 		\draw
 		(0,0)node[fill,circle,inner sep=3pt]{}
 		(1,0)node[fill,circle,inner sep=3pt]{}
 		(2,0)node[fill,circle,inner sep=3pt]{}
 		(3,0)node[fill,circle,inner sep=3pt]{}
 		(4,0)node[fill,circle,inner sep=3pt]{}
 		(5,0)node[fill,circle,inner sep=3pt]{}
 		(6,0)node[fill,circle,inner sep=3pt]{}
 		(7,0)node[fill,circle,inner sep=3pt]{}
 		(8,0)node[fill,circle,inner sep=3pt]{}
 		(9,0)node[fill,circle,inner sep=3pt]{}
 		(10,0)node[fill,circle,inner sep=3pt]{}
 		(11,0)node[fill,circle,inner sep=3pt]{}
 		(12,0)node[fill,circle,inner sep=3pt]{}
 		(13,0)node[fill,circle,inner sep=3pt]{}
 		(14,0)node[fill,circle,inner sep=3pt]{}
 		(15,0)node[fill,circle,inner sep=3pt]{}
 		(16,0)node[fill,circle,inner sep=3pt]{}
 		(17,0)node[fill,circle,inner sep=3pt]{}
 		(18,0)node[fill,circle,inner sep=3pt]{}
 		(19,0)node[fill,circle,inner sep=.5pt]{}
 		(20,0)node[fill,circle,inner sep=.5pt]{}
 		(21,0)node[fill,circle,inner sep=.5pt]{}
 		(22,0)node[fill,circle,inner sep=.5pt]{}
 		(23,0)node[fill,circle,inner sep=.5pt]{}
 		(24,0)node[fill,circle,inner sep=.5pt]{}
 		(25,0)node[fill,circle,inner sep=.5pt]{}
 		(26,0)node[fill,circle,inner sep=.5pt]{}
 		(0,1)node[fill,circle,inner sep=3pt]{}
 		(1,1)node[fill,circle,inner sep=3pt]{}
 		(2,1)node[fill,circle,inner sep=3pt]{}
 		(3,1)node[fill,circle,inner sep=3pt]{}
 		(4,1)node[fill,circle,inner sep=3pt]{}
 		(5,1)node[fill,circle,inner sep=3pt]{}
 		(6,1)node[fill,circle,inner sep=3pt]{}
 		(7,1)node[fill,circle,inner sep=3pt]{}
 		(8,1)node[fill,circle,inner sep=3pt]{}
 		(9,1)node[fill,circle,inner sep=3pt]{}
 		(10,1)node[fill,circle,inner sep=3pt]{}
 		(11,1)node[fill,circle,inner sep=3pt]{}
 		(12,1)node[fill,circle,inner sep=3pt]{}
 		(13,1)node[fill,circle,inner sep=3pt]{}
 		(14,1)node[fill,circle,inner sep=3pt]{}
 		(15,1)node[fill,circle,inner sep=3pt]{}
 		(16,1)node[fill,circle,inner sep=3pt]{}
 		(17,1)node[fill,circle,inner sep=3pt]{}
 		(18,1)node[fill,circle,inner sep=3pt]{}
 		(19,1)node[fill,circle,inner sep=3pt]{}
 		(20,1)node[fill,circle,inner sep=3pt]{}
 		(21,1)node[fill,circle,inner sep=3pt]{}
 		(22,1)node[fill,circle,inner sep=3pt]{}
 		(23,1)node[fill,circle,inner sep=.5pt]{}
 		(24,1)node[fill,circle,inner sep=.5pt]{}
 		(25,1)node[fill,circle,inner sep=.5pt]{}
 		(26,1)node[fill,circle,inner sep=.5pt]{}
 		;
 		\draw[very thick,red] plot [smooth,tension=.1] coordinates{(22,1)(19,1)(18,0)(16,0)}; 
 		\draw[very thick,red] plot [smooth,tension=.1] coordinates{(18,1)(16,1)(15,0)(12,0)};
 		\draw[very thick,red] plot [smooth,tension=.1] coordinates{(15,1)(12,1)(11,0)(9,0)}; 
 		\draw[very thick,red] plot [smooth,tension=.1] coordinates{(11,1)(9,1)(8,0)(5,0)}; 
 		\draw[very thick,red] plot [smooth,tension=.1] coordinates{(8,1)(5,1)(4,0)(1,0)}; 
 		\draw[very thick,red] plot [smooth,tension=.1] coordinates{(4,1)(1,1)(0,0)(-.7,0)}; 
 		\draw[very thick,red] plot [smooth,tension=.1] coordinates{(0,1)(-.7,1)}; 
 	}\quad\dots$$
 	Suppose $k=4$. Moving $P_1,P_2,P_3,P_4$ one step each to the right, we obtain the abacus of $|1^{14}.1^{14},(-4,0)\rangle=\tau(2^{28})$:
 	$$
 	\dots\quad\TikZ{[scale=.5]
 		\draw
 		(0,0)node[fill,circle,inner sep=3pt]{}
 		(1,0)node[fill,circle,inner sep=3pt]{}
 		(2,0)node[fill,circle,inner sep=3pt]{}
 		(3,0)node[fill,circle,inner sep=3pt]{}
 		(4,0)node[fill,circle,inner sep=3pt]{}
 		(5,0)node[fill,circle,inner sep=.5pt]{}
 		(6,0)node[fill,circle,inner sep=3pt]{}
 		(7,0)node[fill,circle,inner sep=3pt]{}
 		(8,0)node[fill,circle,inner sep=3pt]{}
 		(9,0)node[fill,circle,inner sep=3pt]{}
 		(10,0)node[fill,circle,inner sep=3pt]{}
 		(11,0)node[fill,circle,inner sep=3pt]{}
 		(12,0)node[fill,circle,inner sep=3pt]{}
 		(13,0)node[fill,circle,inner sep=3pt]{}
 		(14,0)node[fill,circle,inner sep=3pt]{}
 		(15,0)node[fill,circle,inner sep=3pt]{}
 		(16,0)node[fill,circle,inner sep=3pt]{}
 		(17,0)node[fill,circle,inner sep=3pt]{}
 		(18,0)node[fill,circle,inner sep=3pt]{}
 		(19,0)node[fill,circle,inner sep=3pt]{}
 		(20,0)node[fill,circle,inner sep=.5pt]{}
 		(21,0)node[fill,circle,inner sep=.5pt]{}
 		(22,0)node[fill,circle,inner sep=.5pt]{}
 		(23,0)node[fill,circle,inner sep=.5pt]{}
 		(24,0)node[fill,circle,inner sep=.5pt]{}
 		(25,0)node[fill,circle,inner sep=.5pt]{}
 		(26,0)node[fill,circle,inner sep=.5pt]{}
 		(0,1)node[fill,circle,inner sep=3pt]{}
 		(1,1)node[fill,circle,inner sep=3pt]{}
 		(2,1)node[fill,circle,inner sep=3pt]{}
 		(3,1)node[fill,circle,inner sep=3pt]{}
 		(4,1)node[fill,circle,inner sep=3pt]{}
 		(5,1)node[fill,circle,inner sep=3pt]{}
 		(6,1)node[fill,circle,inner sep=3pt]{}
 		(7,1)node[fill,circle,inner sep=3pt]{}
 		(8,1)node[fill,circle,inner sep=3pt]{}
 		(9,1)node[fill,circle,inner sep=.5pt]{}
 		(10,1)node[fill,circle,inner sep=3pt]{}
 		(11,1)node[fill,circle,inner sep=3pt]{}
 		(12,1)node[fill,circle,inner sep=3pt]{}
 		(13,1)node[fill,circle,inner sep=3pt]{}
 		(14,1)node[fill,circle,inner sep=3pt]{}
 		(15,1)node[fill,circle,inner sep=3pt]{}
 		(16,1)node[fill,circle,inner sep=3pt]{}
 		(17,1)node[fill,circle,inner sep=3pt]{}
 		(18,1)node[fill,circle,inner sep=3pt]{}
 		(19,1)node[fill,circle,inner sep=3pt]{}
 		(20,1)node[fill,circle,inner sep=3pt]{}
 		(21,1)node[fill,circle,inner sep=3pt]{}
 		(22,1)node[fill,circle,inner sep=3pt]{}
 		(23,1)node[fill,circle,inner sep=3pt]{}
 		(24,1)node[fill,circle,inner sep=.5pt]{}
 		(25,1)node[fill,circle,inner sep=.5pt]{}
 		(26,1)node[fill,circle,inner sep=.5pt]{}
 		;
 		\draw[very thick,red] plot [smooth,tension=.1] coordinates{(23,1)(20,1)(19,0)(17,0)}; 
 		\draw[very thick,red] plot [smooth,tension=.1] coordinates{(19,1)(17,1)(16,0)(13,0)};
 		\draw[very thick,red] plot [smooth,tension=.1] coordinates{(16,1)(13,1)(12,0)(10,0)}; 
 		\draw[very thick,red] plot [smooth,tension=.1] coordinates{(12,1)(10,1)(9,0)(6,0)}; 
 		\draw[very thick,red] plot [smooth,tension=.1] coordinates{(8,1)(5,1)(4,0)(1,0)}; 
 		\draw[very thick,red] plot [smooth,tension=.1] coordinates{(4,1)(1,1)(0,0)(-.7,0)}; 
 		\draw[very thick,red] plot [smooth,tension=.1] coordinates{(0,1)(-.7,1)}; 
 	}\quad\dots$$
 	This graphically symbolizes the equality $\tau(2^{28})=\tilde{a}_{(1^4)}(\emp.\emp)$. By Theorem \ref{inducing the type A Steinberg}, the cuspidal support of the unipotent $\mathbbm{k}\GU_{56}(q)$-module $X_{2^{28}}$ is equal to $(\GL_7(q^2)^4,\st_7^4)$, the number of copies of $(\GL_7(q),\st_7)$ in the cuspidal support of $X_{2^{28}}$ being equal to the depth of $\tau(2^{28})$ in the $\slinf$-crystal.
 \end{example}

\section*{Acknowledgments} The author thanks Ting Xue and Olivier Dudas for stimulating discussions, 
Thomas Gerber and Gerhard Hiss for answering some questions about their work and the state of the art, and Gunter Malle for reading the manuscript and suggesting some improvements. Many computations were done in Sage \cite{sage}. During the first two weeks of work on this paper, the author was supported by Max Planck Institute for Mathematics, Bonn. The rest of the time the author was supported at TU Kaiserslautern by the grant SFB-TRR 195.
 
\bibliography{SlinfCrystalGUn}{}
\bibliographystyle{plain}

\end{document}